\documentclass[reqno]{amsart}

\usepackage{tikz-cd}
\usepackage{tikz}

\usepackage{fullpage,dsfont}
\usepackage{leftindex}
\usepackage{tabularx}
\usepackage{mathtools,leftindex,tensor}
\usepackage{mathabx}   % provides \otop and other symbols
\usepackage{graphicx}  % provides \rotatebox

\usepackage{hyperref} % this should come before amsrefs

\usepackage{parskip}
\makeatletter % need this to avoid the conflict between amsthm and parskip
\def\thm@space@setup{%
 \thm@preskip=\parskip \thm@postskip=0pt
}
\def\th@remark{%
  \thm@headfont{\itshape}%
  \normalfont % body font
  \thm@preskip\parskip \thm@postskip=0pt
}
\makeatother

\usepackage[nobysame,alphabetic,initials,msc-links]{amsrefs}

\usepackage[final]{pdfpages}

\usepackage{multirow}

\usepackage{stmaryrd}

\usepackage{array}

\DefineSimpleKey{bib}{how}
\DefineSimpleKey{bib}{mrclass}
\DefineSimpleKey{bib}{mrnumber}
\DefineSimpleKey{bib}{fjournal}
\DefineSimpleKey{bib}{mrreviewer}

\renewcommand{\PrintDOI}[1]{%
  \href{http://dx.doi.org/#1}{{\tt DOI:#1}}%
}
\renewcommand{\eprint}[1]{#1}
\BibSpec{book}{%
    +{}  {\PrintPrimary}                {transition}
    +{.} { \PrintDate}                  {date}
    +{.} { \textit}                     {title}
    +{.} { }                            {part}
    +{:} { \textit}                     {subtitle}
    +{,} { \PrintEdition}               {edition}
    +{}  { \PrintEditorsB}              {editor}
    +{,} { \PrintTranslatorsC}          {translator}
    +{,} { \PrintContributions}         {contribution}
    +{,} { }                            {series}
    +{,} { \voltext}                    {volume}
    +{,} { }                            {publisher}
    +{,} { }                            {organization}
    +{,} { }                            {address}
    +{,} { }                            {status}
    +{,} { \PrintDOI}                   {doi}
    +{,} { \PrintISBNs}                 {isbn}
    +{}  { \parenthesize}               {language}
    +{}  { \PrintTranslation}           {translation}
    +{;} { \PrintReprint}               {reprint}
    +{.} { }                            {note}
    +{.} {}                             {transition}
    +{}  {\SentenceSpace \PrintReviews} {review}
}
\BibSpec{article}{%
    +{}  {\PrintAuthors}                {author}
    +{,} { \textit}                     {title}
    +{.} { }                            {part}
    +{:} { \textit}                     {subtitle}
    +{,} { \PrintContributions}         {contribution}
    +{.} { \PrintPartials}              {partial}
    +{,} { }                            {journal}
    +{}  { \textbf}                     {volume}
    +{}  { \PrintDatePV}                {date}
    +{,} { \issuetext}                  {number}
    +{,} { \eprintpages}                {pages}
    +{,} { }                            {status}
    +{,} { \PrintDOI}                   {doi}
    +{,} { \eprint}        {eprint}
    +{}  { \parenthesize}               {language}
    +{}  { \PrintTranslation}           {translation}
    +{;} { \PrintReprint}               {reprint}
    +{.} { }                            {note}
    +{.} {}                             {transition}
    +{}  {\SentenceSpace \PrintReviews} {review}
}
\BibSpec{collection.article}{%
    +{}  {\PrintAuthors}                {author}
    +{,} { \textit}                     {title}
    +{.} { }                            {part}
    +{:} { \textit}                     {subtitle}
    +{,} { \PrintContributions}         {contribution}
    +{,} { \PrintConference}            {conference}
    +{}  {\PrintBook}                   {book}
    +{,} { }                            {booktitle}
    +{,} { \PrintDateB}                 {date}
    +{,} { pp.~}                        {pages}
    +{,} { }                            {publisher}
    +{,} { }                            {organization}
    +{,} { }                            {address}
    +{,} { }                            {status}
    +{,} { \PrintDOI}                   {doi}
    +{,} { \eprint}        {eprint}
    +{}  { \parenthesize}               {language}
    +{}  { \PrintTranslation}           {translation}
    +{;} { \PrintReprint}               {reprint}
    +{.} { }                            {note}
    +{.} {}                             {transition}
    +{}  {\SentenceSpace \PrintReviews} {review}
}
\BibSpec{misc}{%
  +{}{\PrintAuthors}  {author}
  +{,}{ \textit}      {title}
  +{.}{ }             {how}
  +{}{ \parenthesize} {date}
  +{,} { available at \eprint}        {eprint}
  +{,}{ available at \url}{url}
  +{,}{ }             {note}
  +{.}{}              {transition}
}
\usepackage{amssymb, amsfonts, amsxtra, amsmath}
\usepackage{mathrsfs}
\usepackage{mathdots}
\usepackage{wasysym}
\usepackage[all]{xy}
\usepackage{bbm}
\usepackage{calc}
\usepackage{accents}

\usepackage{stackengine} 

\numberwithin{equation}{section}

\DeclareSymbolFontAlphabet{\mathbb}{AMSb}	% these two magic spells take care of the problem

\newtheorem{Theorem}{Theorem}[section]
\newtheorem*{Theorem*}{Theorem}
\newtheorem{Def}[Theorem]{Definition}
\newtheorem*{Def*}{Def}
\newtheorem{Lem}[Theorem]{Lemma}
\newtheorem{Prop}[Theorem]{Proposition}
\newtheorem{Cor}[Theorem]{Corollary}

\newtheorem{Rem}[Theorem]{Remark}

\newtheorem{Exa}[Theorem]{Example}

\mathchardef\mhyph="2D

\DeclareMathOperator{\End}{\mathrm{End}}
\DeclareMathOperator{\Fun}{\mathrm{Fun}}
\DeclareMathOperator{\Hilb}{\mathrm{Hilb}}

\DeclareMathOperator{\id}{\mathrm{id}}
\DeclareMathOperator{\Rep}{\mathrm{Rep}}

\newcommand{\mcG}{\mathcal{G}}
\newcommand{\mcH}{\mathcal{H}}
\newcommand{\mcK}{\mathcal{K}}
\newcommand{\mcL}{\mathcal{L}}
\newcommand{\mK}{\mathcal{K}}
\newcommand{\C}{\mathbb{C}}
\newcommand{\G}{\mathbb{G}}
\newcommand{\mG}{\mathcal{G}}
\newcommand{\mH}{\mathcal{H}}

\newcommand{\Uu}{\mathds{U}}
\newcommand{\uU}{\text{\reflectbox{$\Uu$}}\:\!}
\newcommand{\Ww}{\mathds{W}}
\newcommand{\wW}{\text{\reflectbox{$\Ww$}}\:\!}
\newcommand{\WW}{{\mathds{V}\!\!\text{\reflectbox{$\mathds{V}$}}}}
\newcommand{\Vv}{\mathds{V}}
\newcommand{\vV}{\text{\reflectbox{$\Vv$}}\:\!}
\newcommand{\VV}[1][1]{%
  \tikz[scale=0.115, line width=0.4pt, baseline={([yshift=-0.7ex]current bounding box.center)}]{
    \draw (-0.1,0) -- (1,0);
    \draw (1.3,0) -- (2.4,0);
    \draw[line width=0.9pt] (0.2,0) -- (1.15,-2);
    \draw[line width=0.9pt] (2.1,0) -- (1.15,-2);
    \draw (0.7,0) -- (1.15,-1.3);
    \draw (1.6,0) -- (1.15,-1.3);
  }%
}

\newcommand{\Nat}{\operatorname{Nat}}
\newcommand{\opp}{\mathrm{op}}
\newcommand{\Corr}{\mathrm{Corr}}
\newcommand{\ovot}{\bar{\otimes}}

\allowdisplaybreaks

\begin{document}

\title{Equivariant Eilenberg-Watts theorems for locally compact quantum groups}
\author{Joeri De Ro}
\address{Institute of Mathematics of the Polish Academy of Sciences}
\email{jdero@impan.pl}

\begin{abstract} Given two von Neumann algebras $A$ and $B$, the $W^*$-algebraic Eilenberg-Watts theorem, due to M.\ Rieffel, asserts that there is a canonical equivalence $\Corr(A,B)\simeq \Fun(\Rep(B), \Rep(A))$ of categories, where $\Corr(A,B)$ denotes the category of all $A$-$B$-correspondences, $\Rep(A)$ is the category of all unital normal $*$-representations of $A$ on Hilbert spaces and $\Fun(\Rep(B), \Rep(A))$ denotes the category of all normal $*$-functors $\Rep(B)\to \Rep(A)$. In this paper, we upgrade the von Neumann algebras $A$ and $B$ with actions $A\curvearrowleft \G$ and $B\curvearrowleft \G$ of a locally compact quantum group $\G$, and we provide several equivariant versions of the $W^*$-algebraic Eilenberg-Watts theorem using the language of module categories. We also prove that for a locally compact quantum group $\G$ with Drinfeld double $D(\G)$, the category of unitary $D(\G)$-representations is isomorphic to the Drinfeld center of $\Rep(\G)$, generalizing a result by Neshveyev-Yamashita from the compact to the locally compact setting.
\end{abstract}

\maketitle

\section{Introduction}

Given von Neumann algebras $A$ and $B$, let us write $\Corr(A,B)$ for the $W^*$-category of $A$-$B$-correspondences (sometimes also called $A$-$B$-bimodules), that is Hilbert spaces $\mathcal{H}$ endowed with commuting unital, normal $*$-representations of $A$ and $B^{\operatorname{op}}$ \cite{Con80}. We also write $\Rep(A)= \Corr(A, \C)$ for the $W^*$-category of unital, normal $*$-representations of $A$ on Hilbert spaces. It follows from the seminal work \cite{Rie74} (see also \cites{Sau83, Bro03}), that there is a canonical equivalence 
\begin{equation}\label{Rieffel}
    \Corr(A,B) \simeq \Fun(\Rep(B), \Rep(A)): \mathcal{G}\mapsto (F_\mathcal{G}: \mathcal{H}\mapsto \mathcal{G}\boxtimes_B \mathcal{H})
\end{equation}
of categories (Theorem \ref{nonequivariant}), where $\Fun(\Rep(B), \Rep(A))$ denotes the category of normal $*$-functors $\Rep(B)\to \Rep(A)$. This is a $W^*$-algebraic version of the celebrated Eilenberg-Watts theorem \cites{Ei60, Wa60}.

In \cite{DCDR24}, the notion of an $A$-$B$-correspondence was generalized to the equivariant setting, where the von Neumann algebras $A$ and $B$ are upgraded with actions of a locally compact quantum group $\G$. This leads to the $W^*$-category $\Corr^\G(A,B)$ of $\G$-$A$-$B$-correspondences, whose objects consist of $A$-$B$-correspondences $\mathcal{H}$ that are upgraded with a unitary $\G$-representation compatible with the relevant structures. It is then natural to ask if the equivalence \eqref{Rieffel} has appropriate analogons in the equivariant setting, where the left hand side is replaced by $\Corr^\G(A,B)$.

This was first investigated in \cite{DR25a}, where it was argued that it is necessary to take into account the extra module-categorical structure that the equivariant setting offers.
More precisely, it was observed in \cite{DR25a} that every $\G$-$W^*$-algebra $A\curvearrowleft \G$ induces a natural action $\Corr^\G(A,\C)=: \Rep^\G(A)\curvearrowleft \Rep(\G)$, turning $\Rep^\G(A)$ in a right $\Rep(\G)$-$W^*$-module category. This then leads to the canonical functor 
\begin{equation}\label{DR}
    \Corr^\G(A,B)\to \Fun_{\Rep(\G)}(\Rep^\G(B), \Rep^\G(A)): \mathcal{G}\mapsto (F_\mG: \mcH\mapsto \mG\boxtimes_B \mcH),
\end{equation}
where $\Fun_{\Rep(\G)}(\Rep^\G(B), \Rep^\G(A))$ is the category of normal $\Rep(\G)$-module $*$-functors $\Rep^\G(B)\to \Rep^\G(A)$. It was shown in \cite{DR25a}*{Theorem 4.5} that \eqref{DR} is an equivalence of categories when $\G$ is a \emph{compact} quantum group, but it was left open if \eqref{DR} also defines an equivalence of categories for a general locally compact quantum group. In this paper, we settle this question in the affirmative. Reasoning along the same lines, we will also show that $\Rep(A\rtimes \G)\curvearrowleft \Rep(\G)$ in a natural way, and that there is a natural equivalence
$$\Corr^\G(A,B)\simeq \Fun_{\Rep(\G)}(\Rep(B\rtimes \G), \Rep(A\rtimes \G)).$$
These are all $\Rep(\G)$-module versions of the equivariant Eilenberg-Watts theorem (see Theorem \ref{main1}).

Given a $\G$-$W^*$-algebra $\alpha: A\curvearrowleft \G$,
a universal version $\alpha^u: A \to A \ovot C_0^u(\G)^{**}$ of the coaction $\alpha: A \to A \ovot L^\infty(\G)$ was constructed in the recent paper \cite{DCK24}. We use this construction to prove that there is a natural action $\Rep(A)\curvearrowleft \Rep(\hat{\G})$, turning $\Rep(A)$ into a right $\Rep(\hat{\G})$-$W^*$-module category. We then prove that there is a natural equivalence
$$\Corr^\G(A,B)\simeq \Fun_{\Rep(\hat{\G})}(\Rep(B), \Rep(A)),$$ providing also a $\Rep(\hat{\G})$-module version of the equivariant Eilenberg-Watts theorem (see Theorem \ref{main2}). The analogon of this fact in the purely algebraic setting was first established in the setting of finite-dimensional Hopf algebras in \cite{AM07}*{Proposition 1.23} and later for general bialgebras in \cite{NSS25}*{Appendix A}. Moreover, it is worth noting that taking $A=B=\C$ endowed with the trivial $\G$-actions leads to the equivalence
$$\Rep(\G)\simeq\End_{\Rep(\hat{\G})}(\Hilb),$$
where the $W^*$-category $\Hilb$ of Hilbert spaces is acted upon by $\Rep(\hat{\G})$ through the forgetful functor $\Rep(\hat{\G})\to \Hilb$ (see also\ \cite{EGNO15}*{Example 7.12.26} for the analogon in the setting of finite-dimensional Hopf algebras).

Let us note that the construction of categorical structures arising from (topological) quantum group actions has attracted considerable attention in recent years (see e.g.\ \cites{PR08, DCY13, Nes14, NY14, HY25}), with most results formulated in the $C^*$-algebraic framework for compact and/or discrete quantum groups. The present paper contributes to this line of research by developing analogous constructions in the $W^*$-algebraic framework, for general locally compact quantum groups.

Given a tensor category $\mathscr{C}$, we denote its Drinfeld center by $\mathcal{Z}(\mathscr{C})$ (see \cite{EGNO15}*{Subsection 7.13} or \cite{Kas95}*{XIII.4}). If $H$ is a finite-dimensional Hopf algebra, consider the tensor category $\prescript{}{H}{\operatorname{Mod}}$ of left $H$-modules. It is well-known (see e.g.\ \cite{Kas95}*{Theorem XIII.5.1}) that
$\mathcal{Z}(\prescript{}{H}{\operatorname{Mod}})\cong \prescript{}{D(H)}{\operatorname{Mod}},$
where $D(H)$ is the Drinfeld double of $H$. An analogon of this fact was proven in the setting of compact quantum groups in \cite{NY18}*{Section 3}, where the authors prove that 
$\mathcal{Z}(\Rep(\G)) \cong \Rep(D(\G))$
for $\G$ a compact quantum group with Drinfeld double $D(\G)$ (note that the Drinfeld center is defined using \emph{unitary} half-braidings in this setting). In the last section of this paper, we generalize this result to the setting of locally compact quantum groups (see Theorem \ref{main3} and Proposition \ref{mainbraiding}). 

\section{Preliminaries}

All vector spaces in this paper are defined over the field of complex numbers. We assume that inner products of Hilbert spaces are anti-linear in the first variable. Given a subset $S$ of a normed linear space $V$, we write $[S]$ for the norm-closure of the linear span of $S$. More generally, if $(V, \tau)$ is a topological vector space and $S\subseteq V$, we write $[S]^\tau$ for the $\tau$-closure of $S$ inside $V$. Given $C^*$-algebras $C,D$, their minimal tensor product is denoted by $C\otimes D$. The multiplier $C^*$-algebra of $C$ is denoted by $M(C)$. Given von Neumann algebras $A$ and $B$, their von Neumann algebra tensor product is denoted by $A \ovot B$.  Given a von Neumann algebra $A$, we consider the standard Hilbert space $L^2(A)$ with modular conjugation $J_A$. It carries a unital, normal, faithful $*$-representation $\pi_A: A\to B(L^2(A))$ and a unital, normal, faithful anti-$*$-representation $\rho_A: A\to B(L^2(A)): x \mapsto J_A \pi_A(x)^* J_A$. We have $\rho_A(A) = \pi_A(A)'$. Given Hilbert spaces $\mcH, \mcK$, we write $\Sigma= \Sigma_{\mcH, \mcK}: \mcH\otimes \mcK\to \mcK\otimes \mcH$ for the switch map. The space of compact operators on $\mcH$ is denoted by $\mathcal{K}(\mcH)$.

\subsection{W*-categories}

For the standard definition of a $C^*$-category, we refer the reader to \cite{GLR85}*{Definition 1.1}. 

A linear functor $F: \mathscr{C}\to \mathscr{D}$ between $C^*$-categories is called \emph{$*$-functor} if $F(f^*) = F(f)^*$ for all $X,Y \in \mathscr{C}$ and all $f\in \mathscr{C}(X,Y)$.

A \emph{$W^*$-category} $\mathscr{C}$ is a $C^*$-category  such that for every two objects $X,Y\in \mathscr{C}$, the Banach space $\mathscr{C}(X,Y)$ has a predual. In particular, the morphism spaces $\mathscr{C}(X)=\mathscr{C}(X,X)$ are $W^*$-algebras.  A $*$-functor $F: \mathscr{C}\to \mathscr{D}$ between $W^*$-categories is called \emph{normal} if for every object $X\in \mathscr{C}$, the unital $*$-homomorphism
$$\mathscr{C}(X)\to \mathscr{D}(F(X)): f \mapsto F(f)$$
is normal. We write $\Fun(\mathscr{C}, \mathscr{D})$ for the class of all normal $*$-functors $\mathscr{C}\to \mathscr{D}$. Two $W^*$-categories $\mathscr{C}, \mathscr{D}$ are called \emph{equivalent} if there exist normal $*$-functors $F: \mathscr{C}\to \mathscr{D}$ and $G: \mathscr{D}\to \mathscr{C}$ together with natural unitary isomorphisms $F\circ G \cong \id_{\mathscr{D}}$ and $G\circ F \cong \id_{\mathscr{C}}.$

Given a $W^*$-category $\mathscr{C}$, a \emph{direct sum} of a set of objects $\{C_i\}_{i\in I}\subseteq \mathscr{C}$ is an object $C=\bigoplus_{i\in I} C_i\in \mathscr{C}$ together with isometries $\lambda_i: C_i\hookrightarrow C$ such that $\sum_{i\in I} \lambda_i \lambda_i^*= 1_C$ (the convergence of the sum is to be understood in the $\sigma$-weak topology of $\mathscr{C}(C)$). Assume that $\mathscr{C}, \mathscr{D}$ are $W^*$-categories and that $F: \mathscr{C}\to \mathscr{D}$ is a normal $*$-functor. If $(C, \{\lambda_i: C_i\to C\}_{i\in I})$ is a direct sum of $\{C_i\}_{i\in I}$, then $(F(C), \{F(\lambda_i): F(C_i)\to F(C)\}_{i\in I})$ is a direct sum of $\{F(C_i)\}_{i\in I}$. 

Let $\mathscr{C}$ be a $W^*$-category admitting (arbitrary) direct sums. We call $C\in \mathscr{C}$ a \emph{generator}, if for every object $D\in \mathscr{C}$, there exists an index set $I$ and an isometry $D \hookrightarrow \bigoplus_{i\in I} C$. 

For much more information about $W^*$-categories, the reader is referred to \cite{GLR85}. 

We now discuss an important example of a $W^*$-category. Given a von Neumann algebra $A$, we write $\Rep(A)$ for the $W^*$-category of unital, normal $*$-representations of $A$ on Hilbert spaces. Its objects consist of pairs $(\mcH, \pi_\mcH)$ where $\mcH$ is a Hilbert space and $\pi_\mcH: A \to B(\mcH)$ is a unital, normal $*$-homomorphism. Often, we will simply write $\mcH \in \Rep(A)$, leaving the representation $\pi_\mcH: A \curvearrowright \mcH$ implicit. Given $\mcH, \mcH'\in \Rep(A)$, the associated morphism space is
$${}_A\mathscr{L}(\mcH, \mcH'):= \{x\in B(\mcH, \mcH')\mid \forall a \in A:  x\pi_\mcH(a)= \pi_{\mcH'}(a)x\}.$$

Let us recall some results about this $W^*$-category that will be relevant to us. An object $(\mcH, \pi)\in \Rep(A)$ is a generator if and only if $\pi$ is faithful \cite{Rie74}*{Proposition 1.3}. If $B$ is another von Neumann algebra, $\mathcal{H}\in \Rep(B)$ is a generator and $F,G \in \Fun(\Rep(B),\Rep(A))$, then the map
\begin{equation}\label{Rieffelbijection}
    \Nat(F,G) \to \{x\in {}_A\mathscr{L}(F(\mcH), G(\mcH))\mid \forall y \in {}_B\mathscr{L}(\mcH): xF(y) = G(y)x\}: (\eta_\mcK)_{\mcK\in \Rep(B)} \mapsto \eta_\mcH
\end{equation}
is a well-defined bijection \cite{Rie74}*{Proposition 5.4}, where $\Nat(F,G)$ denotes the set of natural transformations $F \implies G$. 

If $\mcH= (\mcH, \pi_\mcH)\in \Rep(B)$ and if $\mcK$ is an arbitrary Hilbert space, then $(\mcH\otimes \mcK, b\mapsto \pi_\mcH(b)\otimes 1)\in \Rep(B)$. Given a normal $*$-functor $F: \Rep(B)\to \Rep(A)$, there then is a canonical $A$-linear isomorphism
\begin{equation}\label{canonicalisomorphismmultiplicity}C_{\mcH, \mcK}^F:  F(\mcH) \otimes \mcK\to F(\mcH\otimes \mcK).\end{equation}
More concretely, fix an orthonormal basis $\{e_i\}_{i\in I}$ for $\mcK$, and consider for every $i\in I$ the induced $B$-linear isometry
$$\iota_i: \mcH \to  \mcH\otimes \mcK: \xi \mapsto  \xi\otimes e_i.$$
Then 
$$C_{\mcH, \mcK}^F(\xi\otimes e_i)= F(\iota_i)\xi, \quad \xi \in F(\mcH), \quad i\in I$$  defines an $A$-linear unitary, and it does not depend on the choice of orthonormal basis $\{e_i\}_{i\in I}$ for $\mcK$. Given $F,G \in \Fun(\Rep(B), \Rep(A))$ and $\eta\in \Nat(F,G)$, it is then easily verified that the diagram
\begin{equation}\label{canmultiiso}
\begin{tikzcd}
F(\mcH) \otimes \mcK \arrow[rr, "{C_{\mcH, \mcK}^F}"] \arrow[d, "\eta_{\mcH}\otimes 1"'] &  & F(\mcH\otimes \mcK)  \arrow[d, "\eta_{\mcH\otimes \mcK}"] \\
G(\mcH) \otimes \mcK \arrow[rr, "{C_{\mcH,\mcK}^G}"]                                  &  & G(\mcH\otimes \mcK)                                  
\end{tikzcd}
\end{equation}
commutes.

\subsection{Locally compact quantum groups} In this subsection, we recall some notations and conventions from the theory of locally compact quantum groups \cites{KV00,KV03,VV03}.

A \emph{von Neumann bialgebra} is a pair $(M, \Delta)$ where $M$ is a von Neumann algebra and $\Delta: M \to M \ovot M$ is a unital, normal, isometric $*$-homomorphism such that $(\Delta \otimes \id)\circ \Delta = (\id \otimes \Delta)\circ \Delta$.

A \emph{locally compact quantum group} $\G$ is a von Neumann bialgebra $(L^\infty(\G), \Delta)$ for which there exist normal, semifinite, faithful weights $\varphi, \psi: L^\infty(\G)_+\to [0, \infty]$ such that $(\id\otimes \varphi)\Delta(x)= \varphi(x)1$ for all $x\in \mathscr{M}_\varphi^+$ and $(\psi\otimes \id)\Delta(x)= \psi(x)1$ for all $x\in \mathscr{M}_\psi^+$. These are called the \emph{left Haar weight} and the \emph{right Haar weight}, and they can be shown to be unique up to a non-zero positive scalar multiple. We denote the standard predual of the von Neumann algebra $L^\infty(\G)$ by $L^1(\G)$, so that $L^1(\G)^*\cong L^\infty(\G)$.

We denote the GNS-Hilbert space of $\varphi$ by $L^2(\G)$. One can then canonically identify the GNS-Hilbert space of $\psi$ with $L^2(\G)$. Using the GNS-maps
$$\Lambda_\varphi: \mathscr{N}_\varphi\to L^2(\G), \quad \Lambda_\psi: \mathscr{N}_\psi\to L^2(\G),$$
we can then define the unitaries 
\[
V\in B(L^2(\G))\ovot L^ \infty(\G), \quad W \in L^\infty(\G)\ovot B(L^2(\G)),
\]
called respectively \emph{right} and \emph{left} regular unitary representation. They are uniquely determined by
\begin{align*}
(\id\otimes \omega)(V) \Lambda_{\psi}(x) &= \Lambda_{\psi}((\id\otimes \omega)\Delta(x)),
\qquad \omega \in L^1(\G),\quad x\in \mathscr{N}_{\psi},\\
(\omega \otimes \id)(W^*)\Lambda_{\phi}(x) &= \Lambda_{\phi}((\omega\otimes \id)\Delta(x)),\qquad \omega \in L^1(\G),\quad x\in \mathscr{N}_{\phi}.
\end{align*} They are \emph{multiplicative unitaries} \cite{BS93} meaning that
\[
V_{12}V_{13}V_{23} = V_{23}V_{12},\qquad W_{12}W_{13}W_{23}= W_{23}W_{12},
\]
and they implement the coproduct of $L^\infty(\G)$ in the sense that
$$\label{EqComultImpl}
W^*(1\otimes x)W = \Delta(x) = V(x\otimes 1)V^*,\qquad x\in L^\infty(\G).$$
Moreover, we have
$$C_0^r(\G):=[(\omega\otimes \id)(V) \mid \omega \in B(L^2(\G))_*] = [(\id\otimes \omega)(W) \mid \omega \in B(L^2(\G))_*],$$
which is a $\sigma$-weakly dense C$^*$-subalgebra of $L^\infty(\G)$. Then $\Delta(C_0^r(\G))\subseteq M(C_0^r(\G)\otimes C_0^r(\G))$. 

We also define the von Neumann algebra
\begin{align*}
    L^\infty(\hat{\G}) := [(\omega\otimes \id)(W) \mid \omega \in L^1(\G)]^{\sigma\textrm{-weak}}
\end{align*}
together with the coproduct
\begin{align*}
    \hat{\Delta}: L^\infty(\hat{\G})\to L^\infty(\hat{\G})\ovot L^\infty(\hat{\G}): \hat{x}\mapsto \Sigma W(\hat{x}\otimes 1)W^*\Sigma.
\end{align*}
The pair $(L^\infty(\hat{\G}), \hat{\Delta})$ then defines the dual locally compact quantum group $\hat{\G}$. The left invariant weight $\hat{\varphi}$ is constructed in a way that ensures the canonical identification $L^2(\G)= L^2(\hat{\G})$. It will also be useful to consider the locally compact quantum group $\check{\G}:=\hat{\G}'$, so that
$$L^\infty(\hat{\G})'=L^\infty(\check{\G}) = [(\id\otimes \omega)(V) \mid \omega \in L^1(\G)]^{\sigma\textrm{-weak}}
 , \quad \check{\Delta}: L^\infty(\check{\G})\to L^\infty(\check{\G})\ovot L^\infty(\check{\G}): \check{x}\mapsto V^*(1\otimes \check{x})V.$$
We have $$\check{W}=V \in L^\infty(\check{\G})\ovot L^\infty(\G), \quad W \in L^\infty(\G)\ovot L^\infty(\hat{\G}), \quad \check{V}\in L^\infty(\G)'\ovot L^\infty(\check{\G}), \quad \hat{W}= W_{21}^*\in L^\infty(\hat{\G})\ovot L^\infty(\G).$$

We write $J$ for the modular conjugation on $L^2(\G)$ associated to the weight $\varphi$ and we write $\hat{J}$ for the modular conjugation on $L^2(\G)$ associated with the weight $\hat{\varphi}$ (this also coincides with the modular conjugation $\check{J}$ on $L^2(\G)$ in the canonical GNS-construction associated with the left invariant weight $\check{\varphi}$ on $L^\infty(\check{\G})$ given by $\check{\varphi}(\check{x})= \hat{\varphi}(\hat{J}\check{x}\hat{J})$ for $\check{x}\in L^\infty(\check{\G})$). We will write $\rho_\G(x):= Jx^*J$ for $x\in L^\infty(\G)$. We have 
$\hat{J}L^\infty(\G)\hat{J}=L^\infty(\G),$
so we obtain the anti-$*$-homomorphism
$$R: L^\infty(\G)\to L^\infty(\G): x \mapsto \hat{J}x^*\hat{J}.$$
We call $R$ the \emph{unitary antipode} of $\G$. There is a canonical unimodular complex number $c\in \mathbb{C}$ (cfr.\ \cite{KV03}*{Corollary 2.12}) satisfying 
$$c \hat{J}J=\overline{c}J \hat{J}.$$
We write $u_\G:= c \hat{J} J$ for the associated self-adjoint unitary. The following identities will be useful:
\begin{align*}
    (\hat{J}\otimes J)W(\hat{J}\otimes J)=W^*, \quad (J\otimes \hat{J})V(J\otimes \hat{J})=V^*, \quad (u_\G\otimes 1)V(u_\G\otimes 1)=W_{21}.
\end{align*} Pontryagin biduality holds: $\hat{\hat{\G}}= \G$ and $\check{\check{\G}}\cong \G$ via the isomorphism $x\mapsto u_\G x u_\G$. 

\subsection{Representation theory and universal quantum groups}

A unitary (right) $\G$-representation consists of a pair $(\mcH, U_\mcH)$ where $\mcH$ is a Hilbert space and $U_\mcH\in B(\mcH)\ovot L^\infty(\G)$ is a unitary satisfying $(\id \otimes \Delta)(U_\mcH)= U_{\mcH,12}U_{\mcH,13}$. To simplify notation, we  will often simply write $\mcH \in \Rep(\G)$ or $U_\mcH \in \Rep(\G)$. In fact, we automatically have $U_\mcH \in M(\mathcal{K}(\mcH)\otimes C_0^r(\G))$. Note that $(L^2(\G),V), (L^2(\G), W_{21})\in \Rep(\G)$. When we write $L^2(\G)\in \Rep(\G)$, we will always assume that it is endowed with the right regular representation $V$, unless explicitly mentioned otherwise. If $\mcH$ is any Hilbert space, then $\mathbb{I}= \mathbb{I}_\mcH:= 1\otimes 1 \in B(\mcH)\ovot L^\infty(\G)$ is called the \emph{trivial} $\G$-representation on $\mcH$. If $(\mcH, U_\mcH)$ and $(\mcK, U_\mcK)$ are unitary $\G$-representations, then so are $(\mcH, U_\mcH)\otop (\mcK, U_\mcK):= (\mcH\otimes \mcK, U_{\mcH, 13}U_{\mcK, 23})$ and $(\mcH, U_\mcH)\obot (\mcK, U_\mcK):=(\mcH\otimes \mcK, U_{\mcK,23}U_{\mcH,13})$. Given $\mcH, \mcK\in \Rep(\G)$, we write 
$$\mathscr{L}^\G(\mcH, \mcK)= \mathscr{L}^\G(U_\mcH, U_\mcK):= \{x\in B(\mcH, \mcK): (x\otimes 1)U_\mcH= U_\mcK(x\otimes 1)\}$$
to denote the space of \emph{$\G$-intertwiners}. In this way, we obtain the $W^*$-category $\Rep(\G)$. Given $\omega \in L^1(\G)$, we will use the notation $U_\mcH(\omega):= (\id \otimes \omega)(U_\mcH)$.

The following well-known fact will be absolutely crucial to obtain our main results. Its proof consists of elementary calculations. 

\begin{Prop}[Fell's absorption principle]\label{Fell absorption} Let $(\mcH, U)$ be a unitary $\G$-representation. Then 
\begin{align*}
    &U \in \mathscr{L}^\G(U_{13} V_{23}, V_{23}), \quad (u_\G\otimes 1)U_{21}^*(u_\G\otimes 1)\in \mathscr{L}^\G(V_{13}U_{23}, V_{13}).
\end{align*}
\end{Prop} 

The following is trivial to prove:
\begin{Prop}\label{superelementary}
    Given a unital, normal $*$-representation $\pi: L^\infty(\check{\G})\to B(\mcH)$, $U_\pi:= (\pi\otimes \id)(V) \in B(\mcH)\ovot L^\infty(\G)$ defines a unitary $\G$-representation. The assignment $(\mathcal{H}, \pi) \mapsto (\mathcal{H},U_\pi)$ defines a fully faithful normal $*$-functor $\Rep(L^\infty(\check{\G}))\to \Rep(\G)$ that it is the identity on morphisms.
\end{Prop}

Proposition \ref{superelementary} is illustrated by the following table, where $U\in B(\mcH)\ovot L^\infty(\G)$ denotes an arbitrary $\G$-representation:

{
\centering
    \setlength{\arrayrulewidth}{0.1mm}
\setlength{\tabcolsep}{18pt}
\renewcommand{\arraystretch}{1.3}
\begin{tabular}{|c|c|c|}
\hline
$\mathcal{K}$ & $\pi: L^\infty(\check{\G}) \to B(\mathcal{K})$ & $U_\pi \in B(\mathcal{K})\ovot L^\infty(\G)$ \\
\hline
$L^2(\G)$ & $\pi(\check{x}) = \check{x}$ & $V$ \\
$L^2(\G)$ & $\pi(\check{x}) = u_\G \check{x} u_\G$ & $W_{21}$ \\
$\mcH\otimes L^2(\G)$ & $\pi(\check{x}) = U^*(1\otimes \check{x})U$ & $U_{13} V_{23}$ \\
$L^2(\G)\otimes \mcH$ & $\pi(\check{x}) = (u_\G\otimes 1)U_{21}(u_\G\check{x}u_\G\otimes 1)U_{21}^*(u_\G\otimes 1)$ & $V_{13} U_{23}$ \\
\hline
\end{tabular}
\par
}

From now on, we will always identify $\Rep(L^\infty(\check{\G}))\subseteq \Rep(\G)$. This inclusion is strict if and only if $\G$ is non-compact, see \cite{DR25a}*{Proposition 4.1, Remark 4.2}.

The representation theory of $\G$ is governed by a particular $C^*$-algebra $C_0^u(\hat{\G})$ \cite{Kus01}. We recall some of the theory that will be relevant to us. Consider the (unbounded) antipode $S: \mathcal{D}(S)\to L^2(\G)$ of $\G$ and define $$L^1_\sharp(\G):=\{\omega\in L^1(\G)\mid \exists \omega^*\in L^1(\G): \forall x\in \mathcal{D}(S): \omega^*(x)= \overline{\omega(S(x)^*)}\}.$$
The  subspace $L_\sharp^1(\G)$ is norm-dense in $L^1(\G)$. The assignment $\omega \mapsto \omega^*$ turns the convolution algebra $L_\sharp^1(\G)$ into a Banach $*$-algebra for the norm $\|\omega\|_\sharp:= \max\{\|\omega\|, \|\omega^*\|\}$, and we define $C_0^u(\hat{\G})$ to be the $C^*$-completion of $L_\sharp^1(\G)$ with respect to the norm
$$\|\omega\|_u:= \sup_{\phi} \|\phi(\omega)\|, \quad \omega \in L^1_\sharp(\G),$$
where $\phi$ ranges over all non-degenerate $*$-representations of $L_\sharp^1(\G)$ on Hilbert spaces. There exists a unique unitary $\wW\in M(C_0^r(\G)\otimes C_0^u(\hat{\G}))$ such that $(\omega \otimes \id)(\wW)= \omega$ for all $\omega \in L_\sharp^1(\G)$.  It satisfies $(\Delta\otimes \id)(\wW)= \wW_{13}\wW_{23}$. Given a unitary $\G$-representation $\mcH=(\mcH, U)$, there is a unique non-degenerate $*$-representation
$\phi_\mcH= \phi_U : C_0^u(\hat{\G})\to B(\mcH)$ such that $\phi_{\mcH}(\omega)= (\id\otimes \omega)(U)$ for all $\omega\in L^1_\sharp(\G)$. Conversely, given a non-degenerate $*$-representation $\phi: C_0^u(\hat{\G})\to B(\mcH)$, the element $U_\phi:= (\id \otimes \phi)(\wW)_{21}\in  B(\mcH)\ovot L^\infty(\G)$ is a unitary $\G$-representation. The assignments $\mcH\mapsto \phi_\mcH$ and $\phi\mapsto U_\phi$ set up a bijective correspondence between the unitary $\G$-representations and the non-degenerate $*$-representations of $C_0^u(\hat{\G})$ on Hilbert spaces. There is a unique non-degenerate $*$-homomorphism
$\hat{\Delta}^u: C_0^u(\hat{\G})\to M(C_0^u(\hat{\G})\otimes C_0^u(\hat{\G}))$
such that $(\id \otimes \hat{\Delta}^u)(\wW) = \wW_{13}\wW_{12}$, and the pair $(C_0^u(\hat{\G}), \hat{\Delta}^u)$ is a $C^*$-bialgebra. We then have (cfr.\ \cite{BT03}*{Theorem 2.1})
$$\phi_{\mcH \obot \mcK}= (\phi_\mcH \otimes \phi_\mcK)\circ \hat{\Delta}^u, \quad \phi_{\mcH \otop \mcK}= (\phi_\mcH \otimes \phi_\mcK)\circ \hat{\Delta}^{u, \opp}, \quad  \mcH, \mcK\in \Rep(\G).$$
Repeating this discussion with $\G$ replaced by $\hat{\G}$, we obtain the $C^*$-bialgebra $(C_0^u(\G), \Delta^u)$ and the unitary $\hat{\wW}\in M(C_0^r(\hat{\G})\otimes C_0^u(\G))$. We write $\Ww:= \hat{\wW}_{21}^*\in M(C_0^u(\G)\otimes C_0^r(\hat{\G}))$. Consider the canonical surjective $*$-homomorphisms
\begin{align*}
    \pi_\G = \phi_{(L^2(\G),\hat{W}_{21})}: C_0^u(\G)\to C_0^r(\G), \quad \pi_{\hat{\G}}:= \phi_{(L^2(\G), W_{21})}: C_0^u(\hat{\G})\to C_0^r(\hat{\G}).
\end{align*}
We have $(\id \otimes \pi_{\hat{\G}})(\wW)= W$ and $(\pi_\G\otimes \id)(\Ww)= W$.
We also have the half-lifted comultiplications
$$\Delta^{r,u}: C_0^r(\G)\to M(C_0^r(\G)\otimes C_0^u(\G)), \quad \Delta^{u,r}: C_0^r(\G) \to M(C_0^u(\G)\otimes C_0^r(\G))$$
uniquely determined by
$\Delta^{r,u}\circ \pi_\G = (\pi_\G \otimes \id)\circ \Delta^u$ and $ \Delta^{u,r}\circ \pi_\G = (\id \otimes \pi_\G)\circ \Delta^u.$
More precisely, we can define
$$\Delta^{r,u}(x)= \vV(x\otimes 1)\vV^*, \quad \Delta^{u,r}(x)= \Ww^*(1\otimes x)\Ww, \quad x \in C_0^r(\G),$$
where $\vV:= (u_\G\otimes 1)\Ww_{21}(u_\G \otimes 1)\in M(C_0^r(\check{\G})\otimes C_0^u(\G))$.

There is a unique unitary $\WW \in M(C_0^u(\G)\otimes C_0^u(\hat{\G}))$ such that $\Ww_{12}\WW_{13}\wW_{23}= \wW_{23}\Ww_{12}$. It satisfies 
$$(\Delta^u\otimes \id)(\WW)= \WW_{13}\WW_{23}, \quad (\id \otimes \hat{\Delta}^u)(\WW)= \WW_{13}\WW_{12}, \quad (\pi_\G \otimes \id)(\WW)= \wW, \quad (\id \otimes \pi_{\hat{\G}})(\WW)= \Ww.$$
If $U\in B(\mcH)\ovot L^\infty(\G)$ is a unitary $\G$-representation, then $\uU:= (\id \otimes \phi_U)(\WW)_{21}\in M(\mathcal{K}(\mcH)\otimes C_0^u(\G))$ is the unique corepresentation of the $C^*$-bialgebra $(C_0^u(\G), \Delta^u)$ satisfying $(\id \otimes \pi_\G)(\uU)= U$.

\subsection{Actions and crossed products} A right \emph{$\G$-$W^*$-algebra} consists of a pair $(A, \alpha)$ where $A$ is a von Neumann algebra and $\alpha: A \to A\ovot L^\infty(\G)$ is a unital, normal, isometric $*$-homomorphism satisfying the coaction property $(\id\otimes \Delta)\alpha = (\alpha \otimes \id)\alpha$. We will work exclusively with right $\G$-$W^*$-algebras in this paper, so we leave out the prefix `right' from the terminology, and we will often employ the more intuitive notation $\alpha: A\curvearrowleft \G$.

If $(A, \alpha)$ is a $\G$-$W^*$-algebra, we define the \emph{crossed product von Neumann algebra}
$$A\rtimes_\alpha \G:=A^\rtimes:=  [\alpha(A)(1\otimes L^\infty(\check{\G}))]^{\sigma\text{-weak}}\subseteq A\ovot B(L^2(\G)).$$ We obtain the induced \emph{dual} action $\alpha^\rtimes: A^\rtimes \curvearrowleft \check{\G}$ by
$$\alpha^\rtimes: A^\rtimes\to A^\rtimes \ovot L^\infty(\check{\G}) :z \mapsto \check{V}_{23}z_{12}\check{V}_{23}^*.$$
It satisfies $\alpha^\rtimes(\alpha(a))= \alpha(a)\otimes 1$ and $\alpha^\rtimes(1\otimes\check{y})= 1\otimes \check{\Delta}(\check{y})$ for all $a\in A$ and $\check{y}\in L^\infty(\check{\G})$.

There exists a canonical unitary  $\G$-representation $U_\alpha\in B(L^2(A))\ovot L^\infty(\G)$ such that 
$(\pi_A \otimes \id)(\alpha(a))= U_\alpha(\pi_A(a)\otimes 1)U_\alpha^*$ for all $a \in A$ \cite{Va01}. We call $U_\alpha$ the \emph{unitary implementation} of $\alpha$. In fact, we can identify $L^2(A^\rtimes)= L^2(A)\otimes L^2(\G)$ in such a way that the standard $*$-representation $\pi_{A^\rtimes}$ is exactly $\pi_A \otimes \id$ and such that the anti-$*$-representation $\rho_{A^\rtimes}$ is given by
$$\rho_{A^\rtimes}(\alpha(a))= \rho_A(a)\otimes 1, \quad \rho_{A^\rtimes}(1\otimes \check{y}) = U_\alpha(1\otimes \rho_{\check{\G}}(\check{y}))U_\alpha^*, \quad a \in A, \quad \check{y}\in L^\infty(\check{\G}).$$

Given a $\G$-$W^*$-algebra $(A, \alpha)$, we associate to it the $W^*$-category $\Rep^\G(A)$. Its objects consist of triples $\mcH=(\mcH, \pi, U)$ where $\mcH$ is a Hilbert space, $\pi: A \to B(\mcH)$ is a unital, normal $*$-representation and $U\in B(\mcH)\ovot L^\infty(\G)$ is a unitary $\G$-representation satisfying $(\pi\otimes \id)(\alpha(a))= U(\pi(a)\otimes 1)U^*$ for all $a\in A$. For example, $L^2(A)= (L^2(A), \pi_A, U_\alpha)\in \Rep^\G(A)$. Given $\mcH, \mcH'\in \Rep^\G(A)$, the associated morphism space of intertwiners is denoted by 
${}_A\mathscr{L}^\G(\mcH, \mcH')$ and consists of all bounded linear operators $x: \mcH\to \mcH'$ satisfying $x\pi(a)= \pi'(a)x$ for all $a\in A$ and $(x\otimes 1)U= U'(x\otimes 1)$.

Given $\mcH\in \Rep^\G(A)$, we see that \begin{equation}\label{notation}
    S^\G(\mcH):= (\mcH\otimes L^2(\G), a \mapsto U_\mcH(\pi_\mcH(a)\otimes 1)U_{\mcH}^* = (\pi_\mcH\otimes \id)(\alpha(a)), V_{23})\in \Rep^\G(A).
\end{equation}
We have
\begin{equation}\label{semicoarse}
    {}_A \mathscr{L}^\G(S^\G(\mcH))= (\pi_\mcH \otimes \id)(A^\rtimes)'.
\end{equation}

The following is proven in \cite{DR25a}*{Lemma 3.8}:
\begin{Prop}\label{relation categories} Let $(A, \alpha)$ be a $\G$-$W^*$-algebra.
 Given $(\mathcal{H}, \kappa)\in \Rep(A^\rtimes)$, there exists a unique $(\mathcal{H}, \pi_\kappa, U_\kappa)\in \Rep^\G(A)$ such that 
    $\kappa\circ \alpha = \pi_\kappa$ and $(\kappa\otimes \id)(V_{23}) = U_\kappa.$
   The assignment $(\mathcal{H}, \kappa)\mapsto (\mathcal{H}, \pi_\kappa, U_\kappa)$  defines a fully faithful normal $*$-functor $\Rep(A^\rtimes)\to \Rep^\G(A)$ that is the identity on morphisms.
\end{Prop}

In what follows, we will always identify $\Rep(A^\rtimes)\subseteq \Rep^\G(A)$. If $(\mathcal{H},\kappa)\in \Rep(A^\rtimes)$, we will sometimes write $\kappa= \pi_\kappa\rtimes U_\kappa$. For example, the standard representation $\pi_A\otimes \id: A^\rtimes \to B(L^2(A)\otimes L^2(\G))$ of $A^\rtimes$ is $(\pi_A \otimes \id)\alpha\rtimes V_{23}$.

\subsection{Equivariant correspondences} In this subsection, we recall some theory from the article \cite{DCDR24}.

Let $(A, \alpha)$ and $(B, \beta)$ be $\G$-$W^*$-algebras. A \emph{$\G$-$A$-$B$-correspondence} consists of the data $(\mcH, \pi, \rho, U)$ where $\mcH$ is a Hilbert space, $\pi: A\to B(\mcH)$ is a unital, normal $*$-representation, $\rho: B \to B(\mcH)$ is a unital, normal, anti-$*$-representation and $U\in B(\mcH)\ovot L^\infty(\G)$ is a unitary $\G$-representation such that
$$\pi(a)\rho(b) =\rho(b)\pi(a), \quad (\pi\otimes \id)(\alpha(a))= U(\pi(a)\otimes 1)U^*, \quad (\rho\otimes R)(\beta(b))= U^*(\rho(b)\otimes 1)U$$
for all $a\in A$ and $b\in B$. Often, we will simply write $\mcH \in \Corr^\G(A,B)$, and leave the representation theoretic data $\pi, \rho, U$ implicit. 
A $\G$-$A$-$B$-correspondence $(\mcH, \pi, \rho, U)$ is called $\G$-$A$-$B$-\emph{Morita correspondence} if $\pi$ and $\rho$ are faithful and if $\pi(A)'= \rho(B)$. If such a Morita correspondence exists, then $(A,\alpha)$ and $(B, \beta)$ are called $\G$-$W^*$-Morita equivalent and we write $(A, \alpha) \sim_\G (B, \beta)$. We refer the reader to \cites{DR25a,DR25b} for more information about this notion. 

If $\mcH, \mcH'\in \Corr^\G(A,B)$, we write ${}_A\mathscr{L}_B^\G(\mcH, \mcH')$ for the space of bounded linear operators $x: \mcH \to \mcH'$ such that $x\pi(a)= \pi'(a)x, x\rho(b)= \rho'(b)x$ and $(x\otimes 1)U= U'(x\otimes 1)$ for all $a\in A$ and all $b\in B$. In this way, $\Corr^\G(A,B)$ becomes a $W^*$-category. Note also that taking $\G$ to be the trivial group, we simply recover the category $\Corr(A,B)$ of $A$-$B$-correspondence as introduced by Connes \cite{Con80}.

\begin{Exa} Let $(A, \alpha)$ and $(B, \beta)$ be two given $\G$-$W^*$-algebras.
    \begin{enumerate}\setlength\itemsep{-0.5em}
        \item The trivial $\G$-$A$-$A$-correspondence is $L^2(A):=(L^2(A), \pi_A, \rho_A, U_\alpha)$. 
        \item The semi-coarse $\G$-$A$-$A$-correspondence is $S_A^\G:= (L^2(A)\otimes L^2(\G), (\pi_A\otimes \id) \alpha, (\rho_A\otimes \rho_\G) \alpha, V_{23}).$
    \end{enumerate}
\end{Exa}

There are some non-trivial ways to construct new equivariant correspondences from given ones. We end this subsection by discussing two such constructions.

\textbf{Crossed products of equivariant correspondences} \cite{DCDR24}*{Proposition 5.21}. Given $\mcH=(\mcH, \pi,\rho, U)\in \Corr^\G(A,B)$, we define  $\mcH^\rtimes:=(\mcH\otimes L^2(\G), \pi^\rtimes, \rho^\rtimes, \check{V}_{23})\in \Corr^{\check{\G}}(A^\rtimes, B^\rtimes)$, where 
$$\pi^\rtimes(s):= (\pi \otimes \id)(s), \quad \rho^\rtimes(t) = U(\rho\otimes \check{J}(-)^*\check{J})(t)U^*, \quad s \in A^\rtimes, \quad t \in B^\rtimes.$$

\textbf{Composition of equivariant correspondences} \cite{DCDR24}*{Proposition 5.6}. There is a natural operation
$\boxtimes_B: \Corr^\G(A,B)\times \Corr^\G(B, C)\to \Corr^\G(A,C).$
More concretely, if $\mcH= (\mcH, \pi_\mcH, \rho_\mcH, U_\mcH)\in \Corr^\G(A,B)$ and $\mcK= (\mcK, \pi_\mcK, \rho_\mcK, U_\mcK)\in \Corr^\G(B,C)$, then $\mathscr{L}_B(L^2(B), \mcH)\curvearrowleft B$ via $xb:= x\pi_B(b)$ and $B\curvearrowright \mcK$ via $b\xi:= \pi_\mcK(b)\xi$ where $b\in B, x\in \mathscr{L}_B(L^2(B), \mcH)$ and $\xi \in \mcK$. It thus makes sense to form the (algebraic) balanced tensor product $\mathscr{L}_B(L^2(B), \mcH)\odot_B \mcK$, and we endow it with the semi-inner product uniquely determined by
$$\langle x\otimes_B \xi, x'\otimes_B \xi'\rangle:= \langle \xi, \pi_\mcK(\langle x, x'\rangle_B)\xi'\rangle, \quad x, x'\in \mathscr{L}_B(L^2(B), \mcH), \quad \xi, \xi'\in \mcK,$$
where $\langle x,x'\rangle_B$ is the unique element of $B$ satisfying $\pi_B(\langle x,x'\rangle_B)= x^*x'$. By separation-completion, we obtain the Hilbert space $\mcH\boxtimes_B \mcK$. It carries a natural $A$-$C$-correspondence structure, given by
$$\pi_\boxtimes(a)\rho_\boxtimes(c) (x\otimes_B \xi):= \pi_\mcH(a)x\otimes_B \rho_\mcK(c)\xi, \quad a\in A,\quad c\in C, \quad x\in \mathscr{L}_B(L^2(B), \mcH), \quad \xi \in \mcK.$$
Further, viewing $\mcH\otimes L^2(\G)\in \Corr(A\ovot L^\infty(\G), B \ovot L^\infty(\G))$ and $\mcK\otimes L^2(\G)\in \Corr(B \ovot L^\infty(\G), C\ovot L^\infty(\G))$ in the obvious way, we have $$\mathscr{L}_{B\ovot L^\infty(\G)}(L^2(B\ovot L^\infty(\G)), \mcH\otimes L^2(\G))= \mathscr{L}_{B\ovot L^\infty(\G)}(L^2(B)\otimes L^2(\G), \mcH\otimes L^2(\G))= \mathscr{L}_B(L^2(B),\mcH)\ovot L^\infty(\G),$$
together with the natural identification
\begin{equation}\label{naturalunitary}(\mcH\boxtimes_B \mcK)\otimes L^2(\G)\to ( \mcH\otimes L^2(\G))\boxtimes_{B\ovot L^\infty(\G)} (\mcK\otimes L^2(\G)): (x\otimes_B \xi)\otimes \eta\mapsto (x\otimes 1)\otimes_{B\ovot L^\infty(\G)} (\xi \otimes \eta)\end{equation}
of $A\ovot L^\infty(\G)$-$C \ovot L^\infty(\G)$-correspondences. The unitary 
$$U_\boxtimes: (\mcH\boxtimes_B \mcK)\otimes L^2(\G) \to (\mcH\otimes L^2(\G))\boxtimes_{B\ovot L^\infty(\G)} (\mcK\otimes L^2(\G)) \cong  (\mcH\boxtimes_B \mcK)\otimes L^2(\G)$$
uniquely determined by
$$U_\boxtimes((x\otimes_B\xi)\otimes \eta)= U_\mcH(x\otimes 1)U_\beta^* \otimes_{B\ovot L^\infty(\G)} U_\mcK(\xi \otimes \eta)$$
endows the Hilbert space $\mcH\boxtimes_B \mcK$ with a $\G$-representation such that $\mcH\boxtimes_B \mcK := (\mcH\boxtimes_B \mcK, \pi_\boxtimes, \rho_\boxtimes, U_\boxtimes)\in \Corr^\G(A,C)$. We refer to this operation as the \emph{Connes fusion tensor product} of equivariant correspondences. The Connes fusion tensor product is associative in a natural way. 

\subsection{The W*-algebraic Eilenberg-Watts theorem} 

We believe it is beneficial for the reader to recall the (non-equivariant) $W^*$-algebraic Eilenberg-Watts theorem, due to M. Rieffel \cite{Rie74}. Note that instead of the Hilbert space picture of correspondences, Rieffel rather works in the equivalent framework of (self-dual) Hilbert $W^*$-$A$-$B$-bimodules \cite{BDH88}*{Théorème 2.2}. 

Let $A,B$ be von Neumann algebras. If $\mcG\in \Corr(A,B)$, define a functor $P(\mcG):= F_\mcG\in \Fun(\Rep(B), \Rep(A))$ as follows:
\begin{itemize}\setlength\itemsep{-0.5em}
    \item If $\mcH\in \Rep(B)= \Corr(B,\C)$, put $F_\mcG(\mcH)= \mcG\boxtimes_B \mcH \in \Corr(A,\C)= \Rep(A)$.
    \item If $\mcH, \mcH'\in \Rep(B)$ and $x\in {}_B\mathscr{L}(\mcH, \mcH')$, define
    $$F_\mcG(x):= 1\boxtimes_B x: \mcG\boxtimes_B \mcH\to \mcG\boxtimes_B \mcH', \quad F_\mcG(x)(y\otimes_B \xi):= y\otimes_B x\xi, \quad y \in \mathscr{L}_B(L^2(B), \mcG), \quad \xi \in \mcH.$$
\end{itemize}
If $\mcG, \mcG'\in \Corr(A,B)$, $x\in {}_A\mathscr{L}_B(\mcG, \mcG')$ and $\mcH\in \Rep(B)$, we define the $A$-linear morphism
$$P(x)_\mcH: \mcG \boxtimes_B \mcH\to \mcG'\boxtimes_B \mcH, \quad P(x)_\mcH(y\otimes_B \xi) = xy\otimes_B \xi, \quad y\in \mathscr{L}_B(L^2(B), \mcG), \quad \xi\in \mcH.$$
Then $P(x)=\{P(x)_\mcH\}_{\mcH\in \Rep(B)}\in \Nat(F_\mcG, F_{\mcG'})$, and in this way we obtain the functor
$$P: \Corr(A,B)\to \Fun(\Rep(B), \Rep(A)).$$

On the other hand, if $F\in \Fun(\Rep(B), \Rep(A))$, then put $(\mcG_F, \pi_F):= F((L^2(B), \pi_B))\in \Rep(A)$. If $b\in B$, then $\rho_B(b)\in \pi_B(B)'={}_B\mathscr{L}(L^2(B))$, so it makes sense to define
$$\rho_F(b):= F(\rho_B(b))\in {}_A\mathscr{L}(\mcG_F).$$
Thus, by construction, $\rho_F(b)\pi_F(a)= \pi_F(a)\rho_F(b)$ for all $a\in A$ and all $b\in B$, so that $(\mcG_F, \pi_F, \rho_F)\in \Corr(A,B)$. If $F,G \in \Fun(\Rep(B), \Rep(A))$ and $\eta \in \Nat(F,G)$, then
$$\eta_{L^2(B)}\rho_F(b) = \eta_{L^2(B)} F(\rho_B(b))= G(\rho_B(b)) \eta_{L^2(B)}= \rho_G(b)\eta_{L^2(B)},$$
so it is clear that $\eta_{L^2(B)}\in {}_A\mathscr{L}_B(\mcG_F, \mcG_G)$. In this way, we obtain the functor
$$Q: \Fun(\Rep(B), \Rep(A))\to \Corr(A,B).$$

\begin{Theorem}[W*-algebraic Eilenberg-Watts theorem]\label{nonequivariant}
    The canonical functors
    \begin{align*}
        &P: \Corr(A,B)\to \Fun(\Rep(B), \Rep(A)): \mcG \mapsto F_\mcG: (\mcH\mapsto \mcG\boxtimes_B \mcH),\\
        &Q: \Fun(\Rep(B), \Rep(A))\to \Corr(A,B): F \mapsto F(L^2(B))
    \end{align*}
    are quasi-inverse to each other.
\end{Theorem}
\begin{proof}
    Given $\mcG\in \Corr(A,B)$, it is clear that $\rho_{F_\mcG}(b) = \rho_\boxtimes(b)$ for $b\in B$, so we have the unitary isomorphism
    $$QP(\mcG)= \mcG\boxtimes_B L^2(B)\cong \mcG: y \otimes_B \xi \mapsto y\xi$$
    of $A$-$B$-correspondences. Thus, it is clear that there is a natural unitary isomorphism $Q\circ P\cong \id_{\Corr(A,B)}$. On the other hand, the bijectivity of the map \eqref{Rieffelbijection} (applied to the generator $\mcH = L^2(B)\in \Rep(B)$) can be restated by saying that $Q$ is fully faithful, so $P$ and $Q$ must be quasi-inverse to each other.
\end{proof}

\subsection{W*-tensor and module categories}

\begin{Def}\label{W*tensorcategory}
    A $W^*$-tensor category consists of a $W^*$-category $\mathscr{C}$,  a functor $\otimes: \mathscr{C}\times \mathscr{C}\to \mathscr{C}$, which is a normal $*$-functor in each variable, and unitary isomorphisms
    $$a_{X,Y,Z}: X\otimes (Y\otimes Z)\cong (X\otimes Y)\otimes Z, \quad X,Y,Z\in \mathscr{C},$$
    called \emph{associators}, which are natural in $X,Y,Z$ and satisfy the pentagon axiom.
\end{Def}

Note that we do not ask for the existence of a monoidal unit. 

The following examples of $W^*$-tensor categories arising from a locally compact quantum group $\G$ are the main examples of interest for this paper:
\begin{Exa}
   Consider the $W^*$-category $\Rep(\G)$ of unitary $\G$-representations. It becomes a $W^*$-tensor category for 
    $$(\mcH, U_\mcH)\otop (\mcK, U_\mcK) := (\mcH\otimes \mcK, U_{\mcH, 13}U_{\mcK, 23})$$
(on morphisms, the tensor product $\otop$ agrees with the tensor product on the underlying category of Hilbert spaces).
    The associators in $\Rep(\G)$ are the usual associators in the underlying category of Hilbert spaces, and we will always treat them as identities. Similarly, $\Rep(\G)$ becomes a $W^*$-tensor category for 
    $$(\mcH, U_\mcH)\obot (\mcK, U_\mcK):= (\mcH \otimes\mcK, U_{\mcK, 23}U_{\mcH,13}).$$
    Both of these tensor products on $\Rep(\G)$ will occur in this paper. The relation between them is described as follows: 
    $$\Sigma_{\mcH, \mcK}\in \mathscr{L}^\G(\mcH \otop \mcK, \mcK \obot \mcH).$$
\end{Exa}

\begin{Exa}\label{ideal property} We have already seen that $\Rep(L^\infty(\check{\G}))\subseteq \Rep(\G)$. In fact, we even have that $\Rep(L^\infty(\check{\G}))$ is an \emph{ideal} inside $\Rep(\G)$, in the sense that 
$$\Rep(L^\infty(\check{\G}))\otop \Rep(\G)\subseteq \Rep(L^\infty(\check{\G})), \quad \Rep(\G)\otop \Rep(L^\infty(\check{\G}))\subseteq \Rep(L^\infty(\check{\G})).$$
 Indeed, this follows from the table below Proposition \ref{superelementary}, together with the fact that the identity representation $L^\infty(\check{\G})\subseteq B(L^2(\G))$ (corresponding to the $\G$-representation $V$) generates $\Rep(L^\infty(\check{\G}))$. In particular, $(\Rep(L^\infty(\check{\G})), \otop)$  becomes a $W^*$-tensor category in its own right. The tensor product $\otop$ on $\Rep(L^\infty(\check{\G}))$ is more explicitly described by
 $$(\mcH, \pi)\otop (\mcH', \pi') = (\mcH\otimes \mcH', (\pi\otimes \pi')\circ \check{\Delta}).$$
 %Note that unlike $\Rep(\G)$, the $W^*$-tensor category $\Rep(L^\infty(\hat{\G}))$ typically has no monoidal unit: it admits one (and agrees with the one of $\Rep(\G)$) if and only if $\G$ is compact.
\end{Exa}

\begin{Def}
    A braided $W^*$-tensor category $\mathscr{C}$ is a $W^*$-tensor category $(\mathscr{C}, \otimes)$ together with unitary isomorphisms
    $$\gamma_{X,Y}: X\otimes Y \to Y \otimes X, \quad X,Y \in \mathscr{C},$$
    natural in $X,Y$, such that the braid relations (see \cite{EGNO15}*{Definition 8.1.1}, up to a convention change for the associators) are satisfied.
\end{Def}

We will encounter examples of braided $W^*$-tensor categories in Section \ref{5}.
\begin{Def}\label{module category}
    Let $\mathscr{C}$ be a $W^*$-tensor category. By a (right) $W^*$-module category $\mathscr{M}\curvearrowleft \mathscr{C}$, we mean a $W^*$-category $\mathscr{M}$ endowed with a functor
$\lhd: \mathscr{M}\times \mathscr{C}\to \mathscr{M}$
and unitary isomorphisms
$$\phi_{M,X,Y}: M\lhd (X\otimes Y) \cong (M\lhd X) \lhd Y, \quad M\in \mathscr{M}, \quad X,Y\in \mathscr{C},$$
natural in $M,X,Y$, such that:

\begin{enumerate}\setlength\itemsep{-0.5em}
    \item For all $M\in \mathscr{M}$, $M\lhd -: \mathscr{C}\to \mathscr{M}$ defines a normal $*$-functor.
    \item For all $X\in \mathscr{C}$, $-\lhd X: \mathscr{M}\to \mathscr{M}$ defines a normal $*$-functor.
    \item The following diagram commutes for all $M\in \mathscr{M}$ and all $X,Y,Z \in \mathscr{C}$:
$$
\begin{tikzcd}
M\lhd (X\otimes (Y\otimes Z)) \arrow[d, "1\lhd {a_{X,Y,Z}}"'] \arrow[rrrr, "{\phi_{M,X,Y\otimes Z}}"] &  &                                                                 &  & (M\lhd X)\lhd (Y\otimes Z) \arrow[d, "{\phi_{M\lhd X,Y,Z}}"] \\
M\lhd ((X\otimes Y)\otimes Z) \arrow[rr, "{\phi_{M, X\otimes Y,Z}}"]                                  &  & (M\lhd (X\otimes Y))\lhd Z \arrow[rr, "{\phi_{M,X,Y}\lhd 1}"] &  & ((M\lhd X)\lhd Y)\lhd Z                                      
\end{tikzcd}$$
\end{enumerate}

\end{Def}

\begin{Rem} All module categories in this paper will be categories of Hilbert spaces endowed with extra structure. The isomorphisms $\phi_{M,X,Y}$ occurring in Definition \ref{module category} will always be the canonical associators in the underlying category of Hilbert spaces, and will therefore always be treated as identities (i.e.\ we will not care about bracketing).
\end{Rem}

\begin{Def}
    Given a $W^*$-tensor category $\mathscr{C}$ and two $W^*$-module categories $\mathscr{M} \curvearrowleft \mathscr{C}$ and $\mathscr{N}\curvearrowleft \mathscr{C}$, a normal $\mathscr{C}$-module $*$-functor consists of a normal $*$-functor $F: \mathscr{M}\to \mathscr{N}$ together with unitary isomorphisms
    $$S_{M,X}: F(M\lhd X) \to F(M) \lhd X,\quad M\in \mathscr{M}, \quad X\in \mathscr{C},$$
    natural in $M,X$, such that the diagram
    \begin{equation}\label{module functor}
\begin{tikzcd}
F(M\lhd (X\otimes Y)) \arrow[rrrr, "{F(\phi_{M,X,Y})}"] \arrow[d, "{S_{M,X\otimes Y}}"] &  &                    &  & F((M\lhd X)\lhd Y) \arrow[d, "{S_{M\lhd X,Y}}"] \\
F(M)\lhd (X\otimes Y) \arrow[rr, "{\phi_{F(M),X,Y}}"']                                 &  & (F(M)\lhd X)\lhd Y &  & F(M\lhd X)\lhd Y \arrow[ll, "{S_{M,X}\lhd 1}"] 
\end{tikzcd}
    \end{equation}
commutes for all $M\in \mathscr{M}$ and all $X,Y \in \mathscr{C}$. We will write $\Fun_\mathscr{C}(\mathscr{M}, \mathscr{N})$ for the class of normal $\mathscr{C}$-module $*$-functors. 
\end{Def}

Often, we will simply write $F\in \Fun_{\mathscr{C}}(\mathscr{M}, \mathscr{N})$, leaving the unitaries $\{S_{M,X}\}$ implicit.\textbf{}

\begin{Def}
    Given two $W^*$-module categories $\mathscr{M}\curvearrowleft \mathscr{C}$ and $\mathscr{N}\curvearrowleft \mathscr{C}$ and $F,G\in \Fun_{\mathscr{C}}(\mathscr{M}, \mathscr{N})$, we denote the collection of natural transformations $\eta: F \implies G$ such that the diagrams
    \begin{equation}\label{rio}
        \begin{tikzcd}
F(M\lhd X) \arrow[d, "\eta_{M\lhd X}"'] \arrow[rr, "{S_{M,X}}"] &  & F(M)\lhd X \arrow[d, "\eta_M\lhd 1"] \\
G(M\lhd X) \arrow[rr, "{S_{M,X}}"]                              &  & G(M)\lhd X                          
\end{tikzcd}
    \end{equation}
commute for all $M\in \mathscr{M}$ and all $X\in \mathscr{C}$ by $\Nat_\mathscr{C}(F,G)$.
\end{Def}

Note that we should designate the unitaries $S_{X,M}$ in \eqref{rio} with a label referring to which functor they are associated. We will often refrain from doing this, for ease of notation.

Fix $F,G,H \in \Fun_\mathscr{C}(\mathscr{M}, \mathscr{N})$. Given $\eta \in \Nat_\mathscr{C}(F,G)$ and $\theta \in \Nat_\mathscr{C}(G,H)$, it is easily verified that $\theta\circ \eta:= (\theta_M\circ \eta_M)_{M\in \mathscr{M}}\in \Nat_\mathscr{C}(F,H)$. Thus, $\Fun_\mathscr{C}(\mathscr{M}, \mathscr{N})$ becomes a category in an obvious way.

Two $W^*$-module categories $\mathscr{M}\curvearrowleft \mathscr{C}$ and $\mathscr{N}\curvearrowleft \mathscr{C}$ are called \emph{equivalent} if there exist $F\in \Fun_\mathscr{C}(\mathscr{M}, \mathscr{N})$ and $G\in \Fun_\mathscr{C}(\mathscr{N}, \mathscr{M})$ together with unitary isomorphisms in $\Nat_\mathscr{C}(GF, \id_\mathscr{M})$ and $\Nat_\mathscr{C}(FG, \id_\mathscr{N})$.\footnote{Here, the compositions $GF$ and $FG$ can be viewed as $\mathscr{C}$-module functors in a natural way.}

Plenty of examples of $W^*$-module categories and their associated module $*$-functors will be discussed in the following sections.

\section{Equivariant Eilenberg-Watts theorem - \texorpdfstring{$\Rep(\G)$}{TEXT}-module version}\label{Section 3}

Throughout this section, let $\G$ be a locally compact quantum group and let $(A, \alpha)$ and $(B, \beta)$ be $\G$-$W^*$-algebras. We endow the $W^*$-category $\Rep(\G)$ with the tensor product
$$(\mcH,U_\mcH)\otop (\mcK, U_\mcK):= (\mcH \otimes \mcK, U_{\mcH, 13}U_{\mcK, 23}).$$

\subsection{Module category associated to dynamical system}\label{3.1} We start by recalling the construction of the $W^*$-module category $\Rep^\G(A)\curvearrowleft \Rep(\G)$ from \cite{DR25a}.

Given $\mcH = (\mcH, \pi_\mcH, U_\mcH)\in \Rep^\G(A)$ and $\mcK= (\mcK, U_\mcK)\in \Rep(\G)$, it is elementary to verify that $\mcH\lhd \mcK:= (\mcH \otimes \mcK, \pi_{\mcH\rhd \mcK}: a \mapsto \pi_\mcH(a)\otimes 1, U_\mcH \otop U_\mcK)\in \Rep^\G(A)$. If $x\in {}_A\mathscr{L}^\G(\mcH, \mcH')$ and $y \in \mathscr{L}^\G(\mcK, \mcK')$, then $x\lhd y := x\otimes y\in {}_A \mathscr{L}^\G(\mcH\lhd \mcK, \mcH'\lhd \mcK')$. In this way, we obtain the functor
$\lhd: \Rep^\G(A)\times \Rep(\G)\to \Rep^\G(A),$
and it is also obvious that
$$\mcH \lhd (\mcK \otop \mcK') = (\mcH\lhd \mcK)\lhd \mcK',\quad \mcH \in \Rep^\G(A), \quad \mcK, \mcK'\in \Rep(\G),$$
so that we obtain the right $W^*$-module category $\Rep^\G(A)\curvearrowleft \Rep(\G)$.

The following can be seen as a generalization of the ideal property of $\Rep(L^\infty(\check{\G}))\subseteq \Rep(\G)$ (cfr.\ Example \ref{ideal property}):
\begin{Lem}\label{restrictionn} 
$\Rep(A^\rtimes)\lhd \Rep(\G)\subseteq \Rep(A^\rtimes)$ and $\Rep^\G(A)\lhd \Rep(L^\infty(\check{\G}))\subseteq \Rep(A^\rtimes).$
\end{Lem}
\begin{proof}
    If $(\mcH,U)\in \Rep(\G)$,  Fell's absorption principle (Proposition \ref{Fell absorption}) shows that $$\tilde{U}:= (u_\G\otimes 1)U_{21}^*(u_\G\otimes 1)\in \mathscr{L}^\G(V_{13}U_{23}, V_{13}).$$ We then define the $*$-representation 
    $$A^\rtimes\to B(L^2(A)\otimes L^2(\G)\otimes \mcH): z \mapsto \tilde{U}_{23}^*((\pi_A\otimes \id)(z)\otimes 1)\tilde{U}_{23},$$
   showing that $(L^2(A)\otimes L^2(\G), \pi_A\otimes \id)\lhd (\mcH, U) \in \Rep(A^\rtimes)$.
Since $(L^2(A)\otimes L^2(\G), \pi_A \otimes \id)$ is a generator for $\Rep(A^\rtimes)$, it follows that 
    $\Rep(A^\rtimes)\lhd \Rep(\G)\subseteq \Rep(A^\rtimes).$

    On the other hand, if $(\mcK, \pi_\mcK, U_\mcK)\in \Rep^\G(A)$, defining the $*$-representation
    $$A^\rtimes \to B(\mcK\otimes L^2(\G)): z \mapsto U_\mcK^*  (\pi_\mcK \otimes \id)(z) U_\mcK$$
   shows that $(\mcK, \pi_\mcK, U_\mcK)\lhd (L^2(\G),V)\in \Rep(A^\rtimes)$. Since $(L^2(\G), V)$ is a generator for $\Rep(L^\infty(\check{\G}))$, it follows that $\Rep^\G(A)\lhd \Rep(L^\infty(\check{\G}))\subseteq \Rep(A^\rtimes)$.
\end{proof}

By considering the restriction $\Rep(L^\infty(\check{\G}))\subseteq \Rep(\G)$, we can also view $\Rep^\G(A)$ as a right $\Rep(L^\infty(\check{\G}))$-module $W^*$-category and $\Rep(A^\rtimes)$ as a right $\Rep(L^\infty(\check{\G}))$-module $W^*$-category. In conclusion, to every $\G$-$W^*$-algebra $A$, we have associated four module $W^*$-categories, namely
$$\Rep^\G(A)\curvearrowleft \Rep(\G), \quad \Rep(A^\rtimes)\curvearrowleft \Rep(\G), \quad \Rep^\G(A)\curvearrowleft \Rep(L^\infty(\check{\G})), \quad \Rep(A^\rtimes)\curvearrowleft \Rep(L^\infty(\check{\G})).$$

\subsection{From equivariant correspondence to module functor.}

Fix $\mcG\in \Corr^\G(A,B)$. If $\mcH, \mcH'\in \Rep^\G(B)= \Corr^\G(B, \C)$, we have $\mcG \boxtimes_B \mcH\in \Corr^\G(A, \C)= \Rep^\G(A)$ and if $s\in {}_B\mathscr{L}^\G(\mcH, \mcH')$, then we have the induced map
$$1\boxtimes_B s: \mcG\boxtimes_B \mcH\to \mcG\boxtimes_B \mcH', \quad (1\boxtimes_B s)(y \otimes_B \xi)= y\otimes_B s\xi, \quad y \in \mathscr{L}_B(L^2(B), \mcG), \quad \xi\in \mcH.$$
Then $1\boxtimes_B s \in {}_A\mathscr{L}^\G(\mcG\boxtimes_B \mcH, \mcG\boxtimes_B \mcH')$. In this way, we obtain the normal $*$-functor
$P(\mcG)=F_\mcG: \Rep^\G(B)\to \Rep^\G(A).$
Given $\mcH\in \Rep^\G(A)$ and $\mcK\in \Rep(\G)$, we have unitary isomorphisms
$$S_{\mcG, \mcH,\mcK}: F_\mcG(\mcH\lhd \mcK)= \mcG\boxtimes_B (\mcH \lhd \mcK)\cong (\mcG\boxtimes_B \mcH)\lhd \mcK= F_\mcG(\mcH)\lhd \mcK: y\otimes_B (\xi \otimes \eta)\mapsto (y\otimes_B \xi)\otimes \eta.$$
In this way, we can view $F_\mcG \in \Fun_{\Rep(\G)}(\Rep^\G(B), \Rep^\G(A))$. Given $\mcG, \mcG'\in \Corr^\G(A,B)$, $x\in {}_A\mathscr{L}^\G_B(\mcG, \mcG')$ and $\mcH\in \Rep^\G(B)$, we define
$$P(x)_\mcH = x\boxtimes_B 1: F_\mcG(\mcH)=\mcG\boxtimes_B \mcH\to F_{\mcG'}(\mcH)=\mcG'\boxtimes_B \mcH: y \otimes_B\eta \mapsto  xy\otimes_B \eta.$$
Then $P(x):= \{P(x)_\mcH\}_{\mcH\in \Rep^\G(B)}\in \Nat_{\Rep(\G)}(F_\mcG, F_{\mcG'})$. Combining all this, we find the functor
$$P: \Corr^\G(A,B)\to \Fun_{\Rep(\G)}(\Rep^\G(B), \Rep^\G(A)).$$

If $F\in \Fun_{\Rep(L^\infty(\check{\G}))}(\Rep^\G(B), \Rep^\G(A))$, then $F(\Rep(B^\rtimes))\subseteq \Rep(A^\rtimes)$ (cfr.\ \cite{DR25a}*{Lemma 3.9}). Thus, we have the following list of restriction functors:
\begin{align*}
    (-)_r&: \Fun_{\Rep(\G)}(\Rep^\G(B), \Rep^\G(A))\to \Fun_{\Rep(\G)}(\Rep(B^\rtimes), \Rep(A^\rtimes)), \\
    (-)_m&: \Fun_{\Rep(\G)}(\Rep^\G(B), \Rep^\G(A)) \to \Fun_{\Rep(L^\infty(\check{\G}))}(\Rep^\G(B), \Rep^\G(A)),\\
    (-)_{r,m}&: \Fun_{\Rep(\G)}(\Rep^\G(B), \Rep^\G(A))\to \Fun_{\Rep(L^\infty(\check{\G}))}(\Rep(B^\rtimes), \Rep(A^\rtimes)).
\end{align*}
The $r$-subscript means that we restrict a functor w.r.t.\ the inclusion $\Rep(B^\rtimes)\subseteq \Rep^\G(B)$ and the $m$-subscript means that we restrict the module action w.r.t.\ the inclusion $\Rep(L^\infty(\check{\G}))\subseteq \Rep(\G)$. This leads to the following list of canonical functors:
\begin{align}
   \label{1} P&: \Corr^\G(A,B)\to \Fun_{\Rep(\G)}(\Rep^\G(B), \Rep^\G(A)),\\
    \label{2}P_r = (-)_r\circ P&: \Corr^\G(A,B)\to \Fun_{\Rep(\G)}(\Rep(B^\rtimes), \Rep(A^\rtimes)),\\
    \label{3}P_{m} =(-)_m\circ P&: \Corr^\G(A,B) \to \Fun_{\Rep(L^\infty(\check{\G}))}(\Rep^\G(B), \Rep^\G(A)),\\
    \label{4}P_{r,m} =(-)_{r,m}\circ P&: \Corr^\G(A,B)\to \Fun_{\Rep(L^\infty(\check{\G}))}(\Rep(B^\rtimes), \Rep(A^\rtimes)).
\end{align}

\subsection{From module functor to equivariant correspondence}

One of the main goals of this paper is to show that each of the functors \eqref{1}, \eqref{2}, \eqref{3}, \eqref{4} defines an equivalence of categories. To do this, we will construct appropriate quasi-inverses for each of these functors. A first step in this direction was taken in \cite{DR25a}*{Theorem 3.5, Proposition 3.7}, where the following result was proven:

\begin{Prop}\label{alreadyproven}
   Let $F,G\in \Fun_{\Rep(\G)}(\Rep^\G(B), \Rep^\G(A))$.
    \begin{enumerate}\setlength\itemsep{-0.5em}
        \item Consider $F((L^2(B), \pi_B, U_\beta))=(\mcG_F, \pi_F, U_F)\in \Rep^\G(A)$. There exists a unique unital, normal anti-$*$-representation $\rho_F: B \to B(\mcG_F)$ such that 
$$\rho_F(b)\otimes 1 = U_F S_{L^2(B), L^2(\G)} F((\rho_B\otimes R)\beta(b)) S_{L^2(B), L^2(\G)}^* U_F^*, \quad b \in B.$$
Then $(\mcG_F, \pi_F, \rho_F, U_F)\in \Corr^\G(A,B)$.
\item If $\eta\in \Nat_{\Rep(\G)}(F,G)$, then $\eta_{L^2(B)}\in {}_A\mathscr{L}_B^\G(\mcG_F, \mcG_G).$
\item In this way, we obtain the functor
$$Q: \Fun_{\Rep(\G)}(\Rep^\G(B), \Rep^\G(A))\to \Corr^\G(A,B).$$
\item We have $Q\circ P \cong \id_{\Corr^\G(A,B)}$.
    \end{enumerate} 
\end{Prop}

In \cite{DR25a}*{Theorem 4.5}, it was proven that $Q$ is quasi-inverse to $P$, if $\G$ is a compact quantum group. We will prove that this is also true for general locally compact quantum groups in this paper. The proof in \cite{DR25a} in the compact case makes use of the fact that $\Rep(A^\rtimes)= \Rep^\G(A)$. However, if $\G$ is non-compact, the inclusion $\Rep(A^\rtimes)\subseteq \Rep^\G(A)$ becomes strict, so we will have to do additional work to circumvent the issue.

It turns out that once we construct a quasi-inverse for \eqref{4}, then finding quasi-inverses for \eqref{1}, \eqref{2}, \eqref{3} will follow relatively quickly. Thus, given $F \in \Fun_{\Rep(L^\infty(\check{\G}))}(\Rep(B^\rtimes), \Rep(A^\rtimes))$, we want to find a way to associate a $\G$-$A$-$B$-correspondence to it. If we try to mimique Proposition \ref{alreadyproven}, one issue becomes apparent: if we consider $L^2(B)=(L^2(B), \pi_B, U_\beta)\in \Rep^\G(B)$, it may happen that $L^2(B)\notin \Rep(B^\rtimes)$ (this is already clear by taking $B= \C$ and $\G$ non-compact), so it does not make sense to evaluate the functor $F$ in this object. It follows from \cite{Va01}*{Theorem 5.3} that a $\G$-$W^*$-algebra $(B, \beta)$ is \emph{integrable}, in the sense that the canonical normal, faithful operator-valued weight $T_\beta: B^+ \to \overline{B^\beta}^+$ \cite{Va01}*{Definition 3.1} is semi-finite, exactly when $L^2(B) \in \Rep(B^\rtimes)$. For an integrable $\G$-$W^*$-algebra $(B, \beta)$, we will now prove a version of Proposition \ref{alreadyproven}. Using the equivariant Morita theory, we will then be able to reduce the general case to the integrable case at a later stage.

\begin{Prop}\label{integrablequasi}
    Suppose that $(B, \beta)$ is integrable. Let $F,G\in \Fun_{\Rep(L^\infty(\check{\G}))}(\Rep(B^\rtimes), \Rep(A^\rtimes))$.
    \begin{enumerate}\setlength\itemsep{-0.5em}
        \item Consider $F(L^2(B))=(\mcG_F, \pi_F, U_F)\in \Rep(A^\rtimes)\subseteq \Rep^\G(A)$. There exists a unique unital, normal anti-$*$-representation $\rho_F: B \to B(\mcG_F)$ such that 
$$\rho_F(b)\otimes 1 = U_F S_{L^2(B), L^2(\G)} F((\rho_B\otimes R)\beta(b)) S_{L^2(B), L^2(\G)}^* U_F^*, \quad b \in B.$$
Then $(\mcG_F, \pi_F, \rho_F, U_F)\in \Corr^\G(A,B)$.
\item If $\eta\in \Nat_{\Rep(L^\infty(\check{\G}))}(F,G)$, then $\eta_{L^2(B)}\in {}_A\mathscr{L}_B^\G(\mcG_F, \mcG_G).$
\item In this way, we obtain the functor
$$Q_{r,m}: \Fun_{\Rep(L^\infty(\check{\G}))}(\Rep(B^\rtimes), \Rep(A^\rtimes))\to \Corr^\G(A,B).$$
\item $Q_{r,m}\circ P_{r,m}\cong \id_{\Corr^\G(A,B)}.$
    \end{enumerate}
\end{Prop}

\begin{proof} We will follow the same strategy as in \cite{DR25a}. We give some details, since at one point, the proof has to be adapted because we are dealing with the smaller subcategory $\Rep(L^\infty(\check{\G}))\subseteq \Rep(\G)$.

Recalling the notation introduced in \eqref{notation}, we have that $U_F \in {}_A \mathscr{L}^\G(\mcG_F\lhd L^2(\G), S^\G(\mcG_F))$ is a unitary intertwiner. Consider then the unitary $\tilde{S}: F(S^\G(L^2(B)))\to S^\G(\mcG_F)$ defined by the composition
    $$
\begin{tikzcd}
\tilde{S}: F(S^\G(L^2(B))) \arrow[rr, "F(U_\beta^*)"] &  & F(L^2(B)\lhd L^2(\G)) \arrow[rr, "{S_{L^2(B), L^2(\G)}}"] &  & \mcG_F\lhd L^2(\G) \arrow[rr, "U_F"] &  & S^\G(\mcG_F).
\end{tikzcd}$$
Invoking \eqref{semicoarse}, we have
$\rho_{B^\rtimes}(B^\rtimes)= \pi_{B^\rtimes}(B^\rtimes)'= {}_B\mathscr{L}^\G(S^\G(L^2(B))).$ We then endow the Hilbert space $\mcG_F\otimes L^2(\G)$ with the $A^\rtimes$-$B^\rtimes$-correspondence structure given by:
    \begin{itemize}\setlength\itemsep{-0.5em}
        \item the normal, unital $*$-representation $\pi_F^\rtimes: A^\rtimes \to B(\mcG_F\otimes L^2(\G)): z \mapsto (\pi_F\otimes \id)(z)$,
        \item the normal, unital, anti-$*$-representation $\tilde{\rho}_F: B^\rtimes \to B(\mcG_F\otimes L^2(\G)): z \mapsto \tilde{S}F(\rho_{B^\rtimes}(z))\tilde{S}^*$.
    \end{itemize}
Next, we introduce the following notations:
 \begin{align*}
       S^{2,2}_{V_{23}} &:= S_{L^2(B)\lhd L^2(\G),(L^2(\G), \mathbb{I})\otop (L^2(\G), V)}, &  S_{V_{13}V_{23}}^{2,2}&:= S_{L^2(B)\lhd L^2(\G), (L^2(\G), V)\otop (L^2(\G),V)},\\
       S_{V_{14}V_{34}}^{1,3}&:= S_{L^2(B),(L^2(\G), V)\otop (L^2(\G), \mathbb{I})\otop (L^2(\G),V)}, & S_{V_{14}V_{24}V_{34}}^{1,3}&:= S_{L^2(B),(L^2(\G), V)\otop (L^2(\G), V)\otop (L^2(\G),V)}.
   \end{align*}

   It is important to note that all $\G$-representations occurring in the second factor of the above list of morphisms actually belong to $\Rep(L^\infty(\check{\G}))$, and thus these morphisms make sense.\footnote{Comparing with the proof of \cite{DR25a}*{Theorem 3.5}, this explains why there is an additional tensor factor $L^2(\G)$, since it bypasses the fact that $(L^2(\G), \mathbb{I})\notin \Rep(L^\infty(\check{\G}))$ if $\G$ is not compact.} By the commutativity of the diagram \eqref{module functor}, these morphisms are related by the following relations:
   \begin{equation}\label{modulefunctorconsequence}
      S^{1,3}_{V_{14}V_{34}}S^{2,2,*}_{V_{23}}=S_{L^2(B), L^2(\G)}\otimes 1 \otimes 1= S^{1,3}_{V_{14}V_{24}V_{34}}S^{2,2,*}_{V_{13}V_{23}}.
   \end{equation}

    By Fell's absorption principle (Proposition \ref{Fell absorption}), we have that $U:= (u_\G\otimes 1)V_{21}^*(u_\G\otimes 1)\in \mathscr{L}^\G(V_{13}V_{23},V_{13})$. Then also $U\otimes 1\in \mathscr{L}^\G(V_{14}V_{24}V_{34}, V_{14}V_{34})$. 
We then calculate
\begin{align*}
   &(1\otimes U\otimes 1)(S_{L^2(B), L^2(\G)}F((\rho_B\otimes R)(\beta(b)))S_{L^2(B), L^2(\G)}^*\otimes 1 \otimes 1)(1\otimes U^*\otimes 1)\\
   &= (1\otimes U\otimes 1)S_{V_{14}V_{24}V_{34}}^{1,3} S_{V_{13}V_{23}}^{2,2,*} (F((\rho_B\otimes R)\beta(b))\otimes 1 \otimes 1)S_{V_{13}V_{23}}^{2,2} S_{V_{14}V_{24}V_{34}}^{1,3,*} (1\otimes U^* \otimes 1)\\
   &= (1\otimes U \otimes 1) S_{V_{14}V_{24}V_{34}}^{1,3} F((\rho_B\otimes R)\beta(b)\otimes 1 \otimes 1) S_{V_{14}V_{24}V_{34}}^{1,3,*} (1\otimes U^* \otimes 1)\\
   &= S_{V_{14}V_{34}}^{1,3} F(1\otimes U\otimes 1) F((\rho_B\otimes R)\beta(b)\otimes 1 \otimes 1) F(1\otimes U^*\otimes 1) S_{V_{14}V_{34}}^{1,3,*}\\
   &= S_{V_{14}V_{34}}^{1,3} F((\rho_B\otimes R)\beta(b)\otimes 1 \otimes 1)S_{V_{14}V_{34}}^{1,3,*}\\
   &= S_{V_{14}V_{34}}^{1,3}S_{V_{23}}^{2,2,*} (F((\rho_B\otimes R)\beta(b))\otimes 1 \otimes 1) S_{V_{23}}^{2,2} S_{V_{14}V_{34}}^{1,3,*}\\
   &= S_{L^2(B), L^2(\G)} F((\rho_B\otimes R)\beta(b)) S_{L^2(B), L^2(\G)}^*\otimes 1 \otimes 1.
\end{align*}
Here, the first equality follows from \eqref{modulefunctorconsequence}, the second equality follows from naturality of $\{S_{\mH, \mcK}\}$ in the first variable, the third equality follows from naturality of $\{S_{\mH, \mcK}\}$ in the second variable, the fourth equality follows from the fact that $U\in L^\infty(\G)'\ovot B(L^2(\G))$, the fifth equality follows from naturality of $\{S_{\mH, \mcK}\}$ in the first variable and the sixth equality follows from \eqref{modulefunctorconsequence}. Since the first leg of $U$ generates the von Neumann algebra $L^\infty(\G)'$, it follows from the above computation that $$S_{L^2(B), L^2(\G)} F((\rho_B\otimes R)(\beta(b))) S_{L^2(B), L^2(\G)}^*\in B(\mcG_F)\ovot L^\infty(\G).$$
Exactly as in the proof of \cite{DR25a}*{Theorem 3.5}, we then see that $(\mcG_F\otimes L^2(\G), \pi_F^\rtimes, \tilde{\rho}_F, \check{V}_{23})\in \Corr^{\check{\G}}(A^\rtimes, B^\rtimes)$ and $\tilde{\rho}_F(1\otimes \check{x})= U_F(1\otimes \rho_{\check{\G}}(\check{x}))U_F$ for all $\check{x}\in L^\infty(\check{\G})$. By \cite{DR25a}*{Lemma 3.6}, it follows that there exists a unique unital, normal, anti-$*$-representation $\rho_F: B\to B(\mcG_F)$ such that $\tilde{\rho}_F= \rho_F^\rtimes$ and $(\mcG_F, \pi_F, \rho_F, U_F)\in \Corr^\G(A,B)$. Thus, $(1)$ is proven. The proof of $(2)$ follows by the same computation as in \cite{DR25a}*{Proposition 3.7}. The statement $(3)$ follows immediately by $(1)$ and $(2)$, and $(4)$ follows from standard calculations.
\end{proof}

\subsection{Equivariant Eilenberg-Watts theorem}

\begin{Prop}\label{integrablefullyfaithful}
   Suppose that $(B, \beta)$ is integrable. The functor
    $$Q_{r,m}: \Fun_{\Rep(L^\infty(\check{\G}))}(\Rep(B^\rtimes), \Rep(A^\rtimes))\to \Corr^\G(A,B)$$
    is fully faithful. 
\end{Prop}

\begin{proof}

Fix $F,G \in \Fun_{\Rep(L^\infty(\check{\G}))}(\Rep(B^\rtimes), \Rep(A^\rtimes))$. We will prove that the map
\begin{equation}\label{biject}
    \Nat_{\Rep(L^\infty(\check{\G}))}(F,G)\to {}_A\mathscr{L}^\G_B(\mcG_F, \mcG_G): \eta \mapsto \eta_{L^2(B)}
\end{equation}
is bijective. The argument presented here follows the strategy of the proof of \cite{DR25a}*{Proposition 4.4}, but we prefer to include most details for the convenience of the reader, since we will have to make small modifications.

If $\eta,\theta \in \operatorname{Nat}_{\Rep(L^\infty(\check{\G}))}(F,G)$ with $\eta_{L^2(B)}= \theta_{L^2(B)}$, then $\eta_{L^2(B)\lhd L^2(\G)}= \theta_{L^2(B)\lhd L^2(\G)}$. Since $L^2(B)\lhd L^2(\G)$ generates $\Rep(B^\rtimes)$, it follows from the injectivity of \eqref{Rieffelbijection} that $\eta = \theta$. Thus, the map \eqref{biject} is injective. We now show that it is surjective as well.

 Given $\eta_1\in {}_A\mathscr{L}^\G_B(\mcG_F, \mcG_G)$, we must prove that there exists $\eta\in \Nat_{\Rep(L^\infty(\check{\G}))}(F,G)$ such that $\eta_{L^2(B)}= \eta_1$. We define
   \begin{align*}
       &\eta_2: = S_{L^2(B), L^2(\G)}^*  (\eta_1\otimes 1) S_{L^2(B), L^2(\G)},\\
    &  \eta_3:= S_{L^2(B)\lhd L^2(\G), L^2(\G)}^*(\eta_2\otimes 1)S_{L^2(B)\lhd L^2(\G), L^2(\G)}= S_{L^2(B), L^2(\G)\otop L^2(\G)}^*(\eta_1\otimes 1 \otimes 1)S_{L^2(B), L^2(\G)\otop L^2(\G)}.
   \end{align*}

Exactly as in the proof of \cite{DR25a}*{Proposition 4.4}, we then see that
$$\eta_2 F((\rho_B\otimes R)\beta(b)) = G((\rho_B\otimes R)\beta(b))\eta_2, \quad b\in B.$$

We now claim that
\begin{equation}\label{commutations}
    \eta_3 F(x)= G(x) \eta_3, \quad x \in {}_{B^\rtimes}\mathscr{L}(L^2(B)\lhd (L^2(\G)\otop L^2(\G))), 
\end{equation}
which as in \cite{DR25a}*{Proposition 4.4} boils down to showing the commutation relation
\begin{equation}\label{comm3}
    \eta_3 F(V_{23}^* U_{\beta,13}^*(\rho_B(b)\otimes 1\otimes 1)U_{\beta,13}V_{23}) = G(V_{23}^* U_{\beta,13}^*(\rho_B(b)\otimes 1\otimes 1)U_{\beta,13}V_{23}) \eta_3, \quad b \in B.
\end{equation}
The proof of this commutation relation in  \cite{DR25a}*{Proposition 4.4} no longer works in the non-compact case, since $(L^2(\G), \mathbb{I})\notin \Rep(L^\infty(\check{\G}))$. To remedy this, we introduce an extra tensor factor, although for the rest the argument is mostly the same, but we prefer to include the lengthy calculation for the convenience of the reader.

As before, we will simplify some notation. Therefore, we write 
\begin{align*}
    S_V^{1,1}&:= S_{L^2(B), (L^2(\G),V)}, & 
    S_{V_{13}V_{23}}^{1,2}&:= S_{L^2(B), (L^2(\G),V)\otop (L^2(\G),V)},\\
      S_{V_{23}}^{2,2}&:= S_{L^2(B)\lhd (L^2(\G),V), (L^2(\G), \mathbb{I})\otop (L^2(\G),V)},&
    S_{V_{14}V_{34}}^{1,3}&:= S_{L^2(B), (L^2(\G),V)\otop (L^2(\G),\mathbb{I})\otop (L^2(\G),V)},\\
   S_{V_{14}V_{24}V_{34}}^{1,3}&:= S_{L^2(B), (L^2(\G),V)\otop (L^2(\G),V)\otop (L^2(\G),V)}, & S_{V}^{3,1}&:= S_{L^2(B)\lhd ((L^2(\G),V)\otop (L^2(\G),V)), (L^2(\G),V)}
  .
\end{align*}

By Fell's absorption principle (Proposition \ref{Fell absorption}), we have $$V^* \Sigma \in \mathscr{L}^\G((L^2(\G)\otimes L^2(\G), V_{13}), (L^2(\G)\otimes L^2(\G), V_{13}V_{23})).$$
We then compute 
\begin{align*}
    &(\eta_3\otimes 1)(F(V_{23}^*U_{\beta, 13}^*(\rho_B(b)\otimes 1^{\otimes 2})U_{\beta, 13} V_{23})\otimes 1)\\
    &=(S_{V_{13}V_{23}}^{1,2,*}\otimes 1)(\eta_1\otimes 1^{\otimes 3}) (S_{V_{13}V_{23}}^{1,2}\otimes 1) (F(V_{23}^*U_{\beta, 13}^*(\rho_B(b)\otimes 1^{\otimes 2})U_{\beta, 13} V_{23})\otimes 1)\\
    &=(S_{V_{13}V_{23}}^{1,2,*}\otimes 1)(\eta_1\otimes 1^{\otimes 3}) S_{V_{14}V_{24}V_{34}}^{1,3} S_{V}^{3,1,*} (F(V_{23}^*U_{\beta, 13}^*(\rho_B(b)\otimes 1^{\otimes 2})U_{\beta, 13} V_{23})\otimes 1)\\
    &= (S_{V_{13}V_{23}}^{1,2,*}\otimes 1)(\eta_1\otimes 1^{\otimes 3})S_{V_{14}V_{24}V_{34}}^{1,3} F(V_{23}^*U_{\beta, 13}^*(\rho_B(b)\otimes 1^{\otimes 3})U_{\beta, 13}V_{23}) S_V^{3,1,*}\\
    &= (S_{V_{13}V_{23}}^{1,2,*}\otimes 1)(\eta_1\otimes 1^{\otimes 3})S_{V_{14}V_{24}V_{34}}^{1,3} F((V^*\Sigma)_{23}) F(U_\beta^*(\rho_B(b)\otimes 1)U_\beta \otimes 1^{\otimes 2}) F((\Sigma V)_{23}) S_{V}^{3,1,*}\\
    &= (S_{V_{13}V_{23}}^{1,2,*}\otimes 1)(\eta_1\otimes 1^{\otimes 3})(V^*\Sigma)_{23} S_{V_{14}V_{34}}^{1,3} F(U_\beta^*(\rho_B(b)\otimes 1)U_\beta \otimes 1^{\otimes 2}) F((\Sigma V)_{23}) S_{V}^{3,1,*}\\
    &= (S_{V_{13}V_{23}}^{1,2,*}\otimes 1)(V^*\Sigma)_{23}(\eta_1\otimes 1^{\otimes 3}) S_{V_{14}V_{34}}^{1,3} F(U_\beta^*(\rho_B(b)\otimes 1)U_\beta \otimes 1^{\otimes 2}) F((\Sigma V)_{23}) S_{V}^{3,1,*}\\
    &= (S_{V_{13}V_{23}}^{1,2,*}\otimes 1)(V^*\Sigma)_{23}(S_{V}^{1,1}\otimes 1^{\otimes 2}) (\eta_2\otimes 1^{\otimes 2})(S_V^{1,1,*}\otimes 1^{\otimes2})S_{V_{14}V_{34}}^{1,3} F(U_\beta^*(\rho_B(b)\otimes 1)U_\beta \otimes 1^{\otimes 2}) F((\Sigma V)_{23}) S_{V}^{3,1,*}\\
    &=   (S_{V_{13}V_{23}}^{1,2,*}\otimes 1)(V^*\Sigma)_{23}(S_{V}^{1,1}\otimes 1^{\otimes 2}) (\eta_2\otimes 1^{\otimes 2})S^{2,2}_{V_{23}}F(U_\beta^*(\rho_B(b)\otimes 1)U_\beta \otimes 1^{\otimes 2}) F((\Sigma V)_{23}) S_{V}^{3,1,*}\\
    &=   (S_{V_{13}V_{23}}^{1,2,*}\otimes 1)(V^*\Sigma)_{23}(S_{V}^{1,1}\otimes 1^{\otimes 2}) (\eta_2\otimes 1^{\otimes 2})(F(U_\beta^*(\rho_B(b)\otimes 1)U_\beta) \otimes 1^{\otimes 2}) S^{2,2}_{V_{23}}F((\Sigma V)_{23}) S_{V}^{3,1,*}\\
    &=  (S_{V_{13}V_{23}}^{1,2,*}\otimes 1)(V^*\Sigma)_{23}(S_{V}^{1,1}\otimes 1^{\otimes 2}) (G(U_\beta^*(\rho_B(b)\otimes 1)U_\beta) \otimes 1^{\otimes 2})(\eta_2\otimes 1^{\otimes 2}) S^{2,2}_{V_{23}}F((\Sigma V)_{23}) S_{V}^{3,1,*}\\
    &= S_V^{3,1} S_{V_{14} V_{24}V_{34}}^{1,3,*}(V^*\Sigma)_{23}(S_{V}^{1,1}\otimes 1^{\otimes 2}) (G(U_\beta^*(\rho_B(b)\otimes 1)U_\beta) \otimes 1^{\otimes 2})(\eta_2\otimes 1^{\otimes 2}) S_{V_{23}}^{2,2}F((\Sigma V)_{23}) S_V^{3,1,*}\\
    &= S_V^{3,1} G((V^*\Sigma)_{23})S_{V_{14} V_{34}}^{1,3,*} (S_{V}^{1,1}\otimes 1^{\otimes 2}) (G(U_\beta^*(\rho_B(b)\otimes 1)U_\beta) \otimes 1^{\otimes 2})(\eta_2\otimes 1^{\otimes 2}) S_{V_{23}}^{2,2}F((\Sigma V)_{23}) S_V^{3,1,*}\\
    &= S_V^{3,1} G((V^*\Sigma)_{23})S_{V_{23}}^{2,2,*} (G(U_\beta^*(\rho_B(b)\otimes 1)U_\beta) \otimes 1^{\otimes 2})(\eta_2\otimes 1^{\otimes 2}) S_{V_{23}}^{2,2}F((\Sigma V)_{23}) S_V^{3,1,*}\\
    &= S_V^{3,1} G((V^*\Sigma)_{23}) G(U_\beta^*(\rho_B(b)\otimes 1)U_\beta \otimes 1^{\otimes 2})S_{V_{23}}^{2,2,*}(\eta_2\otimes 1^{\otimes 2}) S_{V_{23}}^{2,2}F((\Sigma V)_{23}) S_V^{3,1,*}\\   
    &=S_V^{3,1} G((V^*\Sigma)_{23}) G(U_\beta^*(\rho_B(b)\otimes 1)U_\beta \otimes 1^{\otimes 2})S_{V_{23}}^{2,2,*}(S_V^{1,1,*}\otimes 1^{\otimes 2})(\eta_1\otimes 1^{\otimes 3})(S_V^{1,1}\otimes 1^{\otimes 2}) S_{V_{23}}^{2,2}F((\Sigma V)_{23}) S_V^{3,1,*}\\    
    &=S_V^{3,1} G((V^*\Sigma)_{23}) G(U_\beta^*(\rho_B(b)\otimes 1)U_\beta \otimes 1^{\otimes 2})S_{V_{14}V_{34}}^{1,3,*}(\eta_1\otimes 1^{\otimes 3})S_{V_{14}V_{34}}^{1,3}F((\Sigma V)_{23}) S_V^{3,1,*}\\    
    &=S_V^{3,1} G((V^*\Sigma)_{23}) G(U_\beta^*(\rho_B(b)\otimes 1)U_\beta \otimes 1^{\otimes 2})S_{V_{14}V_{34}}^{1,3,*}(\eta_1\otimes 1^{\otimes 3})(\Sigma V)_{23} S_{V_{14}V_{24}V_{34}}^{1,3} S_V^{3,1,*}\\    
    &=S_V^{3,1} G((V^*\Sigma)_{23}) G(U_\beta^*(\rho_B(b)\otimes 1)U_\beta \otimes 1^{\otimes 2})S_{V_{14}V_{34}}^{1,3,*} (\Sigma V)_{23}(\eta_1\otimes 1^{\otimes 3}) S_{V_{14}V_{24}V_{34}}^{1,3} S_V^{3,1,*}\\    
    &=S_V^{3,1} G((V^*\Sigma)_{23}) G(U_\beta^*(\rho_B(b)\otimes 1)U_\beta \otimes 1^{\otimes 2})G((\Sigma V)_{23}) S_{V_{14}V_{24}V_{34}}^{1,3,*}(\eta_1\otimes 1^{\otimes 3}) S_{V_{14}V_{24}V_{34}}^{1,3} S_V^{3,1,*}\\    
    &=S_V^{3,1} G((V^*\Sigma)_{23}) G(U_\beta^*(\rho_B(b)\otimes 1)U_\beta \otimes 1^{\otimes 2})G((\Sigma V)_{23}) S_V^{3,1,*}(S_{V_{13}V_{23}}^{1,2,*}\otimes 1)(\eta_1\otimes 1^{\otimes 3}) (S_{V_{13}V_{23}}^{1,2}\otimes 1)\\ 
    &= S_V^{3,1} G(V_{23}^* U_{\beta, 13}^* (\rho_B(b)\otimes 1^{\otimes 3})U_{\beta, 13}V_{23}) S_{V}^{3,1,*} (\eta_3\otimes 1)\\
    &= (G(V_{23}^*U_{\beta, 13}^*(\rho_B(b)\otimes 1^{\otimes 2})U_{\beta, 13}V_{23})\otimes 1) (\eta_3\otimes 1),
\end{align*}
and \eqref{comm3} follows.

Since $L^2(B)\lhd (L^2(\G)\otop L^2(\G))$ is a generator for $\Rep(B^\rtimes)$, it follows from \eqref{commutations} and the bijectivity of \eqref{Rieffelbijection} that there exists a unique natural transformation $\eta\in \operatorname{Nat}(F,G)$ such that $\eta_3=\eta_{L^2(B)\lhd (L^2(\G)\otop L^2(\G))}$. 

By Fell's absorption principle (Proposition \ref{Fell absorption}), we may pick an isometric intertwiner 
$$T: (L^2(\G),V)\to (L^2(\G),V)\otop (L^2(\G),V).$$
The calculation
\begin{align*}
    G(1\otimes T) \eta_{L^2(B)\lhd L^2(\G)} &= \eta_{L^2(B)\lhd (L^2(\G)\otop L^2(\G))} F(1\otimes T)\\
    &= S_{L^2(B), L^2(\G)\otop L^2(\G)}^* (\eta_1\otimes 1\otimes 1)S_{L^2(B), L^2(\G)\otop L^2(\G)} F(1\otimes T)\\
    &= S_{L^2(B), L^2(\G)\otop L^2(\G)}^* (1\otimes T) (\eta_1\otimes 1) S_{L^2(B), L^2(\G)}\\
    &= G(1\otimes T) S_{L^2(B), L^2(\G)}^*(\eta_1\otimes 1)S_{L^2(B), L^2(\G)} = G(1\otimes T) \eta_2,
\end{align*}
combined with the fact that $T$ is isometric ensures that $\eta_{L^2(B)\lhd L^2(\G)}= \eta_2$.

 Given $\mcH \in \Rep(B^\rtimes)$ and $\mcK\in \Rep(L^\infty(\check{\G}))$, we now prove that the rectangle 
\begin{equation}\label{rectangleco}
    \begin{tikzcd}
F(\mcH\lhd \mcK) \arrow[rr, "\eta_{\mcH\lhd \mcK}"'] \arrow[d, "{S_{\mcH, \mcK}}"'] &  & G(\mcH\lhd \mcK) \arrow[d, "{S_{\mcH, \mcK}}"] \\
F(\mcH)\lhd \mcK \arrow[rr, "\eta_{\mcH}\otimes 1"]                                    &  & G(\mcH)\lhd \mcK                              
\end{tikzcd}
\end{equation}
commutes.
By construction, \eqref{rectangleco} commutes if $\mathcal{H}= L^2(B)\lhd L^2(\G)$ and $\mathcal{K}= L^2(\G)\in \Rep(L^\infty(\check{\G}))$. Both compositions $F(\mcH\lhd L^2(\G))\to G(\mcH)\lhd L^2(\G)$ in \eqref{rectangleco} define natural transformations $F(-\lhd L^2(\G))\implies G(-)\lhd L^2(\G)$, and they agree in the generator $L^2(B)\lhd L^2(\G)\in\Rep(B^\rtimes)$. By the injectivity of \eqref{Rieffelbijection}, it thus follows that \eqref{rectangleco} commutes when $\mcH \in \Rep(B^\rtimes)$ is arbitrary and $\mcK= L^2(\G)$. Similarly, fixing $\mcH \in \Rep(B^\rtimes)$, we obtain natural transformations $F(\mcH \lhd -)\implies G(\mcH)\lhd -$ that agree on the generator $L^2(\G)\in \Rep(L^\infty(\check{\G}))$, so that by the injectivity of \eqref{Rieffelbijection} we find that \eqref{rectangleco} commutes for all $\mcH \in \Rep(B^\rtimes)$ and all $\mcK\in \Rep(L^\infty(\check{\G}))$.

In conclusion, we find that $\eta \in \Nat_{\Rep(L^\infty(\check{\G}))}(F,G)$. By the definition of $\eta_2$, we then also find $\eta_{L^2(B)}= \eta_1$, and the surjectivity of \eqref{biject} is proven.
\end{proof}

We emphasize that in the following results, integrability of $(B, \beta)$ no longer plays a role.

\begin{Theorem}\label{restrictionsisos} The restriction functors
\begin{align}
    \label{r1}&\Fun_{\Rep(\G)}(\Rep^\G(B), \Rep^\G(A))\to \Fun_{\Rep(L^\infty(\check{\G}))}(\Rep^\G(B), \Rep^\G(A))\\
    \label{r2}&\Fun_{\Rep(\G)}(\Rep(B^\rtimes), \Rep(A^\rtimes))\to \Fun_{\Rep(L^\infty(\check{\G}))}(\Rep(B^\rtimes), \Rep(A^\rtimes))
\end{align}
are isomorphisms of categories.
\end{Theorem}

\begin{proof} Let either $\mathscr{C}_A= \Rep^\G(A)$ and $\mathscr{C}_B= \Rep^\G(B)$ or $\mathscr{C}_A= \Rep(A^\rtimes)$ and $\mathscr{C}_A= \Rep^\G(A)$.

Let a functor $F\in \Fun_{\Rep(L^\infty(\check{\G}))}(\mathscr{C}_B, \mathscr{C}_A)$ be given. Implicitly, this means we have natural unitaries
    $$S_{\mcH, \mcK}: F(\mcH\lhd \mcK)\to F(\mcH) \lhd \mcK, \quad \mcH \in \mathscr{C}_B, \quad \mcK\in \Rep(L^\infty(\check{\G}))$$
    satisfying the compatibilities \eqref{module functor}.

    Fix for the moment $\mcH\in \mathscr{C}_B$ and $\mcK\in \Rep(\G)$. We write
    $$Y:= S_{\mcH, \mcK \otop L^2(\G)}\circ S_{\mcH \lhd \mcK, L^2(\G)}^*: F(\mcH \lhd \mcK) \lhd L^2(\G)\to (F(\mcH)\lhd \mcK)\lhd L^2(\G).$$    
    By Fell's absorption principle (Proposition \ref{Fell absorption}), $X:=\Sigma V\in \mathscr{L}^\G(V_{13}V_{23}, V_{13})$. We then calculate
\begin{align*}
    &(1\otimes 1 \otimes X\otimes 1) (Y \otimes 1 \otimes 1)(1_{F(\mcH \lhd \mcK)} \otimes X^* \otimes 1)\\
    &= (1\otimes 1 \otimes X\otimes 1) S_{\mcH, \mcK\otop L^2(\G)\otop L^2(\G)\otop L^2(\G)} S_{\mcH\lhd (\mcK\otop L^2(\G)), L^2(\G)\otop L^2(\G)}^* (S_{\mcH \lhd \mcK, L^2(\G)}^*\otimes 1 \otimes 1)(1_{F(\mcH \lhd \mcK)} \otimes X^* \otimes 1)\\
    &= (1\otimes 1 \otimes X\otimes 1)S_{\mcH, \mcK\otop L^2(\G)\otop L^2(\G)\otop L^2(\G)} S_{\mcH \lhd \mcK, L^2(\G)\otop L^2(\G)\otop L^2(\G)}^*(1_{F(\mcH \lhd \mcK)}\otimes X^*\otimes 1)\\
    &= S_{\mcH, \mcK \otop L^2(\G)\otop (L^2(\G),\mathbb{I})\otop L^2(\G)}F(1\otimes 1 \otimes X\otimes 1)  F(1\otimes 1 \otimes X^* \otimes 1) S_{\mcH\lhd \mcK, L^2(\G)\otop (L^2(\G), \mathbb{I})\otop L^2(\G)}^*\\
    &= S_{\mcH, \mcK \otop L^2(\G)\otop (L^2(\G),\mathbb{I}) \otop L^2(\G)} S_{\mcH \lhd (\mcK\otop L^2(\G)), (L^2(\G),\mathbb{I})\otop L^2(\G)}^* (S_{\mcH\lhd \mcK, L^2(\G)}^*\otimes 1 \otimes 1)= Y \otimes 1 \otimes 1.
\end{align*}
We view $Y\in B(F(\mcH\lhd \mcK), F(\mcH)\otimes \mcK)\ovot B(L^2(\G))$ as having two tensor legs. Then it follows from the above computation that
$V_{23}Y_{12}V_{23}^*= Y_{13}$. In particular, $Y$ has its second leg inside $L^\infty(\G)$. On the other hand, $Y_{13}= (\id \otimes \Delta)(Y) = W_{23}^*Y_{13}W_{23}$ so $Y$ also has its second leg in $L^\infty(\check{\G})$. Since $L^\infty(\G)\cap L^\infty(\check{\G})= \C1$, it follows that there exists a unique unitary $S_{\mcH, \mcK}: F(\mcH \lhd \mcK)\to F(\mcH)\lhd \mcK$ such that
\begin{equation}\label{unique unitary}S_{\mcH, \mcK \otop L^2(\G)}= (S_{\mcH, \mcK}\otimes 1)\circ S_{\mcH\lhd \mcK, L^2(\G)}.\end{equation}
The notation is not ambiguous, because it coincides with the already defined operator $S_{\mcH, \mcK}$ in case that $\mcK \in \Rep(L^\infty(\check{\G}))$, as follows from the commutativity of the diagram \eqref{module functor}.
  It is straightforward to verify that the unitaries $\{S_{\mcH, \mcK}\}_{\mcH \in \mathscr{C}_B, \mcK\in \Rep(\G)}$ are natural in both variables.

Our next goal is to show that the commutation relation \eqref{module functor} is satisfied for the unitaries $\{S_{\mcH, \mcK}\}_{\mcH\in \mathscr{C}_B, \mcK\in \Rep(\G)}$. In other words, given $\mcH\in \mathscr{C}_B$ and $\mcK, \mcK' \in \Rep(\G)$, we will prove commutativity of the triangle
\begin{equation}\label{trianglee}
\begin{tikzcd}
& F((\mcH \lhd \mcK)\lhd\mcK')= F(\mcH\lhd (\mcK\otop \mcK')) \arrow[ld, "{S_{\mcH\lhd \mcK,\mcK'}}"'] \arrow[rd, "{S_{\mcH, \mcK \otop \mcK'}}"] &                                  \\
F(\mcH\lhd \mcK) \lhd \mcK' \arrow[rr, "{S_{\mcH, \mcK}\otimes 1}"'] &                                                                                                                           & (F(\mcH)\lhd \mcK)\lhd \mcK' = F(\mcH)\lhd (\mcK\otop \mcK')
\end{tikzcd}
\end{equation}

By Fell's absorption principle (Proposition \ref{Fell absorption}), we may pick a unitary $T\in \mathscr{L}^\G(\mcK'\otop L^2(\G), L^2(\G)\otop \mcK')$. We then consider the diagram
  \begin{equation}\label{largecommutative}
\begin{tikzcd}
& F(\mcH \lhd (\mcK\otop \mcK'\otop L^2(\G))) \arrow[lddd, "{S_{\mcH \lhd \mcK, \mcK'\otop L^2(\G)}}"'] \arrow[dd, "{S_{\mcH \lhd (\mcK \otop \mcK'), L^2(\G)}}" description] \arrow[rddd, "{S_{\mcH, \mcK\otop \mcK'\otop L^2(\G)}}"] &         \\
                                           &         &      \\
& F(\mcH \lhd (\mcK\otop \mcK'))\lhd L^2(\G) \arrow[rd, "{S_{\mcH, \mcK\otop \mcK'}\otimes 1}"'] \arrow[ld, "{S_{\mcH \lhd \mcK, \mcK'}\otimes 1}"]          &                  \\
F(\mcH \lhd \mcK) \lhd (\mcK'\otop L^2(\G) \arrow[rr, "{S_{\mcH, \mcK}\otimes 1 \otimes 1}"'] \arrow[d, "1\otimes 1 \otimes T"']                         &                                                                & F(\mcH)\lhd (\mcK \otop \mcK'\otop L^2(\G)) \arrow[d, "1\otimes 1 \otimes T"]                           \\
F(\mcH \lhd \mcK)\lhd (L^2(\G)\otop \mcK') \arrow[rr, "{S_{\mcH, \mcK}\otimes 1 \otimes 1}"] \arrow[rddd, "{S_{\mcH \lhd \mcK, L^2(\G)\otop \mcK'}^*}"'] &                                                                                                       & F(\mcH)\lhd (\mcK \otop L^2(\G)\otop \mcK') \arrow[lddd, "{S_{\mcH, \mcK \otop L^2(\G)\otop \mcK'}^*}"] \\
& F(\mcH\lhd (\mcK \otop L^2(\G))\lhd \mcK' \ \arrow[ru, "{S_{\mcH, \mcK \otop L^2(\G)}\otimes 1}"] \arrow[lu, "{S_{\mcH\lhd \mcK, L^2(\G)}\otimes 1}"']                                              &  \\
&    &                                \\
& F(\mcH\lhd (\mcK \otop L^2(\G)\otop \mcK')) \arrow[uu, "{S_{\mcH\lhd (\mcK \otop L^2(\G)), \mcK'}}" description]&
\end{tikzcd}\end{equation}
The commutativity of \eqref{trianglee} will follow once we can show the commutativity of the small triangle immediately above the middle rectangle in the diagram \eqref{largecommutative}.

First, we note that the two small most upper triangles commute by \eqref{unique unitary}. Next, we note that the outer triangle appearing above the central rectangle commutes by \eqref{module functor}, under the additional assumption that $\mcK\in \Rep(L^\infty(\check{\G}))$ (this uses the fact that $\mcK'\otop L^2(\G)\in \Rep(L^\infty(\check{\G})$). By a diagram chase, we then see that the small triangle immediately above the middle rectangle commutes, provided that $\mcK\in \Rep(L^\infty(\check{\G}))$. Consequently, \eqref{trianglee} commutes when $\mcH\in \mathscr{C}_B$, $\mcK\in \Rep(L^\infty(\check{\G}))$ and $\mcK'\in \Rep(\G)$. Now, we return to the general case where $\mcK\in \Rep(\G)$, and we shift our attention to the other parts of the diagram \eqref{largecommutative} as well. The triangle appearing immediately below the middle rectangle commutes by \eqref{unique unitary}. The lower left and lower right triangles commute since we already argued that \eqref{trianglee} commutes when $\mcK\in \Rep(L^\infty(\check{\G}))$. Furthermore, the middle rectangle trivially commutes, and the outer hexagon commutes: both compositions equal $F(1\otimes 1\otimes T)$, as follows from the naturality of $\{S_{\mcH,\mcK}\}_{\mcH\in \mathscr{C}_B, \mcK\in \Rep(L^\infty(\check{\G}))}$ in the second variable. By a diagram chase, we then conclude the commutativity of the small triangle immediately above the middle rectangle, and the commutativity of \eqref{trianglee} follows.

Thus, we have proven that every $F\in \Fun_{\Rep(L^\infty(\check{\G}))}(\mathscr{C}_B, \mathscr{C}_A)$ can be uniquely extended to a functor $F\in \Fun_{\Rep(\G)}(\mathscr{C}_B, \mathscr{C}_A)$.

Next, let $F,G \in \Fun_{\Rep(L^\infty(\check{\G}))}(\mathscr{C}_B, \mathscr{C}_A)$ be given. We will now prove that 
$$\Nat_{\Rep(L^\infty(\check{\G}))}(F,G)= \Nat_{\Rep(\G)}(F,G).$$
For this, we fix $\eta \in \Nat_{\Rep(L^\infty(\check{\G}))}(F,G)$, $\mcH \in \mathscr{C}_B$ and $\mcK\in \Rep(\G)$. We also consider the diagram
$$
\begin{tikzcd}
F(\mcH\lhd \mcK)\lhd L^2(\G) \arrow[rr, "\eta_{\mcH\lhd \mcK}\otimes 1"] \arrow[d, "{S_{\mcH, \mcK}}\otimes 1"] \arrow[dd, "{S_{\mcH\lhd \mcK, L^2(\G)}^*}"', bend right=60, shift right=9] &  & G(\mcH\lhd \mcK)\lhd L^2(\G) \arrow[d, "{S_{\mcH, \mcK}\otimes 1}"'] \arrow[dd, "{S_{\mcH \lhd \mcK, L^2(\G)}^*}", bend left=60, shift left=9] \\
F(\mcH)\lhd (\mcK \otop L^2(\G)) \arrow[rr, "\eta_{\mcH}\otimes 1 \otimes 1"] \arrow[d, "{S_{\mcH, \mcK \otop L^2(\G)}^*}"]                                                                          &  & G(\mcH)\lhd (\mcK \otop L^2(\G)) \arrow[d, "{S_{\mcH, \mcK \otop L^2(\G)}^*}"']                                                                   \\
F(\mcH \lhd (\mcK \otop L^2(\G))) \arrow[rr, "\eta_{\mcH \lhd (\mcK \otop L^2(\G))}"]                                                                                                               &  & G(\mcH \lhd (\mcK \otop L^2(\G)))                                                                                                                  
\end{tikzcd}$$
The most left and most right part of the diagram clearly commute, as does the lower rectangle, by the fact that $\eta\in  \Nat_{\Rep(L^\infty(\check{\G}))}(F,G)$ and since $\mcK \otop L^2(\G)\in \Rep(L^\infty(\check{\G}))$. Furthermore, the outer part of the diagram also commutes, since $\eta \in \Nat_{\Rep(L^\infty(\check{\G}))}(F,G)$. By another diagram chase, we see that the upper rectangle commutes. As such, we deduce that $\eta\in \Nat_{\Rep(\G)}(F,G)$.
\end{proof}

\begin{Prop}\label{faithfulfunctors}
    The restriction functors
    \begin{align}
         \label{s1}&\Fun_{\Rep(L^\infty(\check{\G}))}(\Rep^\G(B), \Rep^\G(A)) \to  \Fun_{\Rep(L^\infty(\check{\G}))}(\Rep(B^\rtimes), \Rep(A^\rtimes)), \\
          \label{s2}&\Fun_{\Rep(\G)}(\Rep^\G(B), \Rep^\G(A)) \to  \Fun_{\Rep(L^\infty(\check{\G}))}(\Rep(B^\rtimes), \Rep(A^\rtimes)) 
    \end{align}
    are fully faithful.
\end{Prop}
\begin{proof}
    Since the functor \eqref{s2} is the composition of the functors \eqref{r1} and \eqref{s1}, it suffices to show that the functor \eqref{s1} is fully faithful. In other words, given $F,G \in \Fun_{\Rep(L^\infty(\check{\G}))}(\Rep^\G(B), \Rep^\G(A))$, we will prove that the map
    \begin{equation}\label{bijection}
       \Nat_{\Rep(L^\infty(\check{\G}))}(F, G)\to \Nat_{\Rep(L^\infty(\check{\G}))}(F, G): (\eta_\mcK)_{\mcK\in \Rep^\G(B)} \mapsto (\eta_{\mcK})_{\mcK\in \Rep(B^\rtimes)}
    \end{equation}
    is bijective.

    If $(\eta_\mcH)_{\mcH\in \Rep^\G(B)}, (\theta_\mcH)_{\mcH\in \Rep^\G(B)}$ agree on $\Rep(B^\rtimes)$, then for $\mcH \in \Rep^\G(B)$, we have $\mcH\lhd L^2(\G)\in \Rep(B^\rtimes)$, and thus
    \begin{align*}
        \eta_\mcH\otimes 1&= S_{\mcH, L^2(\G)}\circ  \eta_{\mcH\lhd L^2(\G)}\circ  S_{\mcH, L^2(\G)}^*=  S_{\mcH, L^2(\G)} \circ \theta_{\mcH\lhd L^2(\G)} \circ S_{\mcH,L^2(\G)}^* = \theta_\mcH\otimes 1.
    \end{align*}
    Therefore, $(\eta_\mcH)_{\mcH\in \Rep^\G(B)}, (\theta_\mcH)_{\mcH\in \Rep^\G(B)}$ agree everywhere, showing the injectivity of \eqref{bijection}.

    To show the surjectivity of \eqref{bijection}, we argue similarly as in the proof of Theorem \ref{restrictionsisos}. Let $(\eta_\mcH)_{\mcH\in \Rep(B^\rtimes)}\in \Nat_{\Rep(L^\infty(\check{\G}))}(F, G)$ be given. Given $\mcH \in \Rep^\G(B)$, we have $\mcH \lhd L^2(\G)\in \Rep(B^\rtimes)$ and it makes sense to consider the composition
    $$Y:=S_{\mcH, L^2(\G)}\circ  \eta_{\mcH \lhd L^2(\G)} \circ S_{ \mcH, L^2(\G)}^*: F(\mcH)\lhd L^2(\G)\to  G(\mcH)\lhd L^2(\G).$$
    Writing 
$X:= \Sigma V \in \mathscr{L}^\G((L^2(\G),V)\otop (L^2(\G),V), (L^2(\G),V)\otop (L^2(\G), \mathbb{I}))$
and introducing the notations
\begin{align*}
    S_{V_{14}V_{24}V_{34}}^{1,3}&:= S_{\mcH, (L^2(\G),V)\otop (L^2(\G),V)\otop (L^2(\G),V)}, & S_{V_{14}V_{34}}^{1,3} &:= S_{\mcH, (L^2(\G),V)\otop (L^2(\G),\mathbb{I})\otop (L^2(\G),V)},\\
    S_{V_{13}V_{23}}^{2,2}&:= S_{\mcH\lhd (L^2(\G),V), (L^2(\G),V)\otop (L^2(\G),V)}, & S_{V_{23}}^{2,2}&:=S_{\mcH\lhd  (L^2(\G),V), (L^2(\G),\mathbb{I})\otop (L^2(\G),V)},
\end{align*}
we compute
\begin{align*}
    &(1\otimes X\otimes 1) (Y\otimes 1 \otimes 1)(1\otimes X^* \otimes 1)\\
    &= (1\otimes X \otimes 1) S_{V_{14}V_{24}V_{34}}^{1,3} S_{V_{13}V_{23}}^{2,2,*} (\eta_{\mathcal{H}\lhd L^2(\G)} \otimes 1 \otimes 1) S_{V_{13}V_{23}}^{2,2} S_{V_{14}V_{24}V_{34}}^{1,3,*}(1\otimes X^*\otimes 1)\\
    &= (1\otimes X \otimes 1) S_{V_{14}V_{24}V_{34}}^{1,3} \eta_{\mathcal{H}\lhd (L^2(\G)\otop L^2(\G)\otop L^2(\G))} S_{V_{14}V_{24}V_{34}}^{1,3,*} (1\otimes X^*\otimes 1)\\
    &=S_{V_{14}V_{34}}^{1,3}G(1 \otimes X \otimes 1) \eta_{\mathcal{H}\lhd (L^2(\G)\otop L^2(\G)\otop L^2(\G))} F(1\otimes X^*\otimes 1) S_{V_{14}V_{34}}^{1,3,*}\\
    &= S_{V_{14}V_{34}}^{1,3} \eta_{\mcH\lhd ((L^2(\G),V)\otop (L^2(\G),\mathbb{I}) \otop (L^2(\G),V))} S_{V_{14}V_{34}}^{1,3,*}\\
    &= S_{V_{14}V_{34}}^{1,3} S_{V_{23}}^{2,2,*} (\eta_{\mathcal{H}\lhd L^2(\G)}\otimes 1 \otimes 1) S_{V_{23}}^{2,2} S_{V_{14}V_{34}}^{3,1,*}= Y\otimes 1 \otimes 1,
\end{align*}
from which it follows that 
$(1\otimes V)(Y\otimes 1)(1\otimes V^*)= Y_{13}.$ Consequently, we see that $Y \in B(F(\mcH), G(\mcH))\ovot \mathbb{C}1$, so there exists a unique bounded linear operator $\eta_\mcH: F(\mcH)\to G(\mcH)$ such that 
    \begin{equation}\label{comm}
\begin{tikzcd}
F(\mcH\lhd L^2(\G)) \arrow[rr, "{S_{\mcH, L^2(\G)}}"] \arrow[d, "\eta_{\mcH\lhd L^2(\G)}"'] &  & F(\mcH)\lhd L^2(\G) \arrow[d, "\eta_{\mcH}\otimes 1"] \\
G(\mcH\lhd L^2(\G)) \arrow[rr, "{S_{\mcH, L^2(\G)}}"]                                          &  & G(\mcH)\lhd L^2(\G)                                     
\end{tikzcd}
    \end{equation}
    commutes. Again, the notation $\eta_\mcH$ is unambiguous and it is straightforward to verify that $(\eta_\mcH)_{\mcH\in \Rep^\G(B)}$ defines a natural transformation.
    
    Given $\mcH \in \Rep^\G(B)$ and $\mcK \in \Rep(L^\infty(\check{\G}))$, we now prove that the rectangle
    \begin{equation}\label{rectangle}
\begin{tikzcd}
F(\mcH\lhd \mcK) \arrow[rr, "\eta_{\mcH\lhd \mcK}"'] \arrow[d, "{S_{\mcH, \mcK}}"'] &  & G(\mcH\lhd \mcK) \arrow[d, "{S_{\mcH, \mcK}}"] \\
F(\mcH)\lhd \mcK \arrow[rr, "\eta_{\mcH}\otimes 1"]                                    &  & G(\mcH)\lhd \mcK                              
\end{tikzcd}
    \end{equation}
    commutes. To do this, we fix a unitary $T \in \mathscr{L}^\G(\mcK \otop L^2(\G), L^2(\G)\otop \mcK)$ and we consider the diagram
    $$
\begin{tikzcd}
F(\mcH\lhd \mcK)\lhd L^2(\G) \arrow[rr, "\eta_{\mcH\lhd \mcK}\otimes 1"] \arrow[ddddd, "{S_{\mcH, \mcK}\otimes 1}"', bend right=49, shift right=13] \arrow[d, "{S_{\mcH\lhd\mcK, L^2(\G)}^*}"] &  & G(\mcH\lhd \mcK)\lhd L^2(\G) \arrow[ddddd, "{S_{\mcH, \mcK}\otimes 1}", bend left=49, shift left=13] \arrow[d, "{S_{\mcH\lhd\mcK, L^2(\G)}^*}"'] \\
F(\mcH \lhd (\mcK \otop L^2(\G)) \arrow[rr, "\eta_{\mcH \lhd (\mcK\otop L^2(\G))}"] \arrow[d, "F(1\otimes T)"]                                                                                 &  & G(\mcH \lhd (\mcK \otop L^2(\G)) \arrow[d, "G(1\otimes T)"']                                                                                     \\
F(\mcH \lhd (L^2(\G) \otop \mcK)) \arrow[rr, "\eta_{\mcH \lhd (L^2(\G)\otop \mcK)}"] \arrow[d, "{S_{\mcH\lhd L^2(\G),\mcK}}"]                                                                   &  & G(\mcH \lhd (L^2(\G) \otop \mcK)) \arrow[d, "{S_{\mcH\lhd L^2(\G),\mcK}}"']                                                                       \\
F(\mcH \lhd L^2(\G))\lhd \mcK \arrow[rr, "{\eta_{\mcH \lhd L^2(\G)}\otimes 1}"] \arrow[d, "{S_{\mcH, L^2(\G)}\otimes 1}"]                                                                         &  & G(\mcH \lhd L^2(\G))\lhd \mcK \arrow[d, "{S_{\mcH, L^2(\G)}\otimes 1}"']                                                                         \\
F(\mcH)\lhd (L^2(\G)\otop \mcK) \arrow[d, "1\otimes T^*"] \arrow[rr, "\eta_{\mcH}\otimes 1 \otimes 1"]                                                                                         &  & G(\mcH)\lhd (L^2(\G)\otop \mcK) \arrow[d, "1\otimes T^*"']                                                                                          \\
F(\mcH)\lhd (\mcK\otop L^2(\G)) \arrow[rr, "\eta_\mcH \otimes 1 \otimes 1"]                                                                                                                    &  & G(\mcH)\lhd (\mcK\otop L^2(\G))                                                                                                                 
\end{tikzcd}$$
Consider the five subrectangles stacked on top of each other. Numbering from top to bottom, we explain why each of them commutes. The first one commutes since $\mcH \lhd \mcK\in \Rep(B^\rtimes)$, the second one by naturality of $\eta$, the third one since $\mcH \lhd L^2(\G)\in \Rep(B^\rtimes)$, the fourth one by \eqref{comm} and the fifth one commutes trivially. Furthermore, it is easy to see that the most left and the most right part of the diagram commute. Consequently, the entire outer part of the diagram commutes. By a diagram chase, it follows that the rectangle \eqref{rectangle} commutes.
\end{proof}

The following useful result will allow us to transfer results established from the integrable case to the general case:

\begin{Lem}\label{moritaintegrable}
    Let $(B, \beta)$ be a $\G$-$W^*$-algebra. There exists an integrable $\G$-$W^*$-algebra $(C, \gamma)$ that is $\G$-$W^*$-Morita equivalent with $(B, \beta)$.
\end{Lem}
\begin{proof} If $(A, \check{\check{\alpha}})$ is an integrable $\check{\check{\G}}$-$W^*$-algebra, then it is easily verified that $(A, \alpha)$ is an integrable $\G$-$W^*$-algebra, where
$$\alpha(a):= (1\otimes u_\G)\check{\check{\alpha}}(a)(1\otimes u_\G), \quad a \in A.$$
It follows from \cite{Va01}*{Proposition 2.5} that the $\check{\check{\G}}$-$W^*$-algebra $(B^{\rtimes \rtimes}, \beta^{\rtimes \rtimes})$ is integrable. Consequently, writing $C:= B^{\rtimes \rtimes}$ and considering
$$\gamma(c):= (1\otimes u_\G)\beta^{\rtimes \rtimes}(c) (1\otimes u_\G), \quad c \in C,$$
it follows that the $\G$-$W^*$-algebra $(C, \gamma)$ is integrable. The von Neumann algebra $B \ovot B(L^2(\G))$ carries the $\G$-action
$$\delta(z):= W_{32}(\beta\otimes \id)(z)_{132}W_{32}^*, \quad z \in B\ovot B(L^2(\G)),$$
and the Takesaki-Takai isomorphism
$$\Phi:  (B\ovot B(L^2(\G)), \delta)\to (C, \gamma): z \mapsto V_{23}^*(\beta \otimes \id)(z)V_{23}$$
is $\G$-equivariant. It is then easily verified that we have $(B, \beta)\sim_\G (C, \gamma)$, with an explicit $\G$-$C$-$B$-Morita correspondence given by the Hilbert space $L^2(B)\otimes L^2(\G)$ endowed with the $C$-$B$-bimodule structure
$$c \xi b := (\pi_B\otimes \id)(\Phi^{-1}(c)) (\rho_B(b)\otimes 1)\xi, \quad \xi\in L^2(B)\otimes L^2(\G), \quad b\in B, \quad c \in C$$
and the unitary $\G$-representation 
$W_{32}U_{\beta,13}\in B(L^2(B)\otimes L^2(\G))\ovot L^\infty(\G)$, cfr.\ \cite{DCDR24}*{Theorem 5.25}.
\end{proof}

We then come to the first main result of this paper:

\begin{Theorem}[Equivariant Eilenberg-Watts theorem I] \label{main1}
    The functors \eqref{1}, \eqref{2}, \eqref{3} and \eqref{4} are all equivalences of categories. Moreover, $P$ and $Q$ are quasi-inverse to each other.
\end{Theorem}  %If $\mathscr{C}, \mathscr{D}$ are $W^*$-categories, if $F: \mathscr{C}\to \mathscr{D}$ and $G: \mathscr{D}\to \mathscr{C}$ are $*$-functors such that $G$ is fully faithful and we have the natural unitary isomorphism $G\circ F \cong \id$, then it is elementary to verify that $F\circ G \cong \id$ as well, so that $F$ and $G$ are quasi-inverse to each other.
\begin{proof}
We start by showing that the functor \eqref{4} is an equivalence of categories. To emphasize the dependency on the $\G$-$W^*$-algebras $A,B$, let us write 
$$P_{r,m}=P_{r,m}^{A,B}: \Corr^\G(A,B)\to \Fun_{\Rep(L^\infty(\check{\G}))}(\Rep(B^\rtimes), \Rep(A^\rtimes)).$$
By Lemma \ref{moritaintegrable}, we can fix an integrable $\G$-$W^*$-algebra $(C, \gamma)$ together with a $\G$-$W^*$-Morita correspondence $\mcG \in \Corr^\G(B,C)$. Considering the dual $\G$-$W^*$-Morita correspondence $\overline{\mcG}\in \Corr^\G(C,B)$ \cite{DCDR24}*{Definition 5.1}, we have unitary isomorphisms $\mcG\boxtimes_C \overline{\mcG}\cong L^2(B)$ and $\overline{\mcG}\boxtimes_B \mcG\cong L^2(C)$ of equivariant correspondences \cite{DCDR24}*{Proposition 5.10}. One then verifies that
$$F_\mcG \in \Fun_{\Rep(L^\infty(\check{\G}))}(\Rep(C^\rtimes), \Rep(B^\rtimes)), \quad F_{\overline{\mcG}}\in \Fun_{\Rep(L^\infty(\check{\G}))}(\Rep(B^\rtimes), \Rep(C^\rtimes))$$
implement equivalences of $\Rep(L^\infty(\check{\G}))$-module categories. Therefore, we obtain the induced equivalences
\begin{align*}
    &S: \Corr^\G(A,B)\to \Corr^\G(A,C): \mK \mapsto \mK\boxtimes_B \mG,\\
    &T: \Fun_{\Rep(L^\infty(\check{\G}))}(\Rep(C^\rtimes), \Rep(A^\rtimes))\to\Fun_{\Rep(L^\infty(\check{\G}))}(\Rep(B^\rtimes), \Rep(A^\rtimes)): F \mapsto F \circ F_{\overline{\mG}}
\end{align*}
of categories. We then verify that
$T \circ P_{r,m}^{A,C}\circ S\cong P_{r,m}^{A,B}$,
so we can assume without loss of generality that $(B, \beta)$ is integrable. In that case, it suffices to show that the functor $$Q_{r,m}: \Fun_{\Rep(L^\infty(\check{\G}))}(\Rep(B^\rtimes), \Rep(A^\rtimes))\to \Corr^\G(A,B)$$
from Proposition \ref{integrablequasi} is quasi-inverse to $P_{r,m}$. This follows from Proposition \ref{integrablefullyfaithful} and from the fact that $Q_{r,m} \circ P_{r,m}\cong \id_{\Corr^\G(A,B)}$.

Next, we show that $P: \Corr^\G(A,B)\to \Fun_{\Rep(\G)}(\Rep^\G(B), \Rep^\G(A))$ is an equivalence of categories. First, we prove this if $(B, \beta)$ is integrable. In that case, we have the factorization
$$
\begin{tikzcd}
{\Fun_{\Rep(\G)}(\Rep^\G(B), \Rep^\G(A))} \arrow[rd, "{(-)_{r,m}}"'] \arrow[rr, "Q"] &                                                                                                       & {\Corr^\G(A,B)} \\
& {\Fun_{\Rep(L^\infty(\check{\G}))}(\Rep(B^\rtimes), \Rep(A^\rtimes))} \arrow[ru, "Q_{r,m}"'] &              
\end{tikzcd}$$
Since both $Q_{r,m}$ and $(-)_{r,m}$ are fully faithful (Proposition \ref{faithfulfunctors}), it follows that $Q$ is fully faithful as well. Since $Q\circ P \cong \id_{\Corr^\G(A,B)}$, it thus follows that $P$ is an equivalence of categories. By a similar Morita argument as before, it follows that $P$ is an equivalence of categories in the non-integrable case as well. Since we have $Q \circ P \cong \id_{\Corr^\G(A,B)}$, it follows that $Q$ is quasi-inverse to $P$.

Keeping in mind the factorizations
$$\begin{tikzcd}
{\Corr^\G(A,B)} \arrow[rr, "P_r"] \arrow[rd, "{P_{r,m}}"'] &                                                                                  & {\Fun_{\Rep(\G)}(\Rep(B^\rtimes), \Rep(A^\rtimes))} \arrow[ld, "\cong"] \\
& {\Fun_{\Rep(L^\infty(\check{\G}))}(\Rep(B^\rtimes), \Rep(A^\rtimes))} &                                                                                     
\end{tikzcd}$$
$$
\begin{tikzcd}
{\Corr^\G(A,B)} \arrow[rd, "P"'] \arrow[rr, "P_m"] &                                                                             & {\Fun_{\Rep(L^\infty(\check{\G}))}(\Rep^\G(B), \Rep^\G(A))} \\
& {\Fun_{\Rep(\G)}(\Rep^\G(B), \Rep^\G(A))} \arrow[ru, "\cong"'] &                                                                       
\end{tikzcd}$$
where the isomorphisms come from Theorem \ref{restrictionsisos}, it follows that $P_r$ and $P_m$ are equivalences as well.
\end{proof}

\begin{Cor}
    The restriction functors \eqref{s1} and \eqref{s2} are equivalences of categories.
\end{Cor}

We can also extend \cite{DR25a}*{Theorem 3.10}.
The key idea behind the next result is that under the equivalences \eqref{1}, \eqref{2}, \eqref{3} and \eqref{4}, an equivalence of the appropriate module categories corresponds exactly with a $\G$-$A$-$B$-Morita correspondence.
\textbf{}
\begin{Prop}\label{Moritaeq}
The following statements are equivalent:
\begin{enumerate}\setlength\itemsep{-0.5em}
    \item $(A, \alpha)\sim_\G (B, \beta)$.
    \item There exist $\mcG\in \Corr^\G(A,B)$ and unitary isomorphisms $\mcG \boxtimes_B \overline{\mcG}\cong L^2(A)$ and $\overline{\mcG}\boxtimes_A \mcG\cong L^2(B)$ of equivariant correspondences. 
    \item $\Rep^\G(A)$ and $\Rep^\G(B)$ are equivalent as $\Rep(\G)$-$W^*$-module categories.
    \item $\Rep^\G(A)$ and $\Rep^\G(B)$ are equivalent as $\Rep(L^\infty(\check{\G}))$-$W^*$-module categories.
 \item $\Rep(A^\rtimes)$ and $\Rep(B^\rtimes)$ are equivalent as $\Rep(\G)$- $W^*$-module categories.
    \item $\Rep(A^\rtimes)$ and $\Rep(B^\rtimes)$ are equivalent as $\Rep(L^\infty(\check{\G}))$-$W^*$-module categories.
     \end{enumerate}
\end{Prop}

\begin{proof} The implication $(1)\implies (2)$ was established in \cite{DCDR24}*{Proposition 5.10}. The implication $(2)\implies (3)$ follows by considering the normal $\Rep(\G)$-module $*$-functors
$$F_\mcG: \Rep^\G(B)\to \Rep^\G(A), \quad F_{\overline{\mcG}}: \Rep^\G(A)\to \Rep^\G(B).$$

The implications $(3)\implies (4)\implies  (6)$ and $(3)\implies (5)\implies (6)$ are trivial. It remains to prove $(6)\implies (1)$.
Assume that $(6)$ holds and choose an equivalence 
$F: \Rep(B^\rtimes)\to \Rep(A^\rtimes)$ of $\Rep(L^\infty(\check{\G}))$-$W^*$-module categories. By Lemma \ref{moritaintegrable}, we may assume that $(B, \beta)$ is integrable. Consider the associated $\G$-$A$-$B$-correspondence $\mcG_F= F(L^2(B))$ constructed in Proposition \ref{integrablequasi}. It then follows exactly as in the proof of \cite{DR25a}*{Theorem 3.10} that $\mcG_F$ is a $\G$-$A$-$B$-Morita correspondence. 
\end{proof}

\section{Equivariant Eilenberg-Watts theorem - \texorpdfstring{$\Rep(\hat{\G})$}{TEXT}-module version}\label{Section 4}

Throughout this section, we fix a locally compact quantum group $\G$ and two (right) $\G$-$W^*$-algebras $(A, \alpha)$ and $(B, \beta)$. We will endow the $W^*$-category $\Rep(\hat{\G})$ with the tensor product
$$(\mcH, U_\mcH)\obot (\mcK, U_\mcK):= (\mcH \otimes \mcK, U_{\mcK,23}U_{\mcH, 13}).$$
The assignment
$$\Rep(L^\infty(\G))\to \Rep(\hat{\G}): (\mcH, \pi)\mapsto (\mcH, (\pi\otimes \id)(\hat{W}_{21}))$$
then defines a fully faithful normal $*$-functor that is the identity on morphism spaces, allowing us to view $\Rep(L^\infty(\G))$ as an ideal inside $(\Rep(\hat{\G}), \obot)$ (cfr.\ Example \ref{ideal property}).

If $(\mcH, U)\in \Rep(\G)$ and $\mcK \in \Rep(\hat{\G})$ with associated non-degenerate $*$-homomorphism $\phi_\mcK: C_0^u(\G)\to B(\mcK)$, we will employ the notation
$$\uU^\mcK:= (\id \otimes \phi_\mcK)(\uU) = (\phi_\mcK \otimes \phi_\mcH)(\WW)_{21}\in B(\mcH \otimes \mcK).$$
For example, given a $\G$-$W^*$-algebra $(A, \alpha)$ and $\mcK \in \Rep(\hat{\G})$, we have the associated unitary implementation $U_\alpha \in \Rep(\G)$ and we will write $\uU_\alpha^\mcK = (\id \otimes \phi_\mcK)(\uU_\alpha) \in B(L^2(A)\otimes \mcK)$.

\subsection{Module category associated to dynamical system.}\label{4.2}

Let $(A, \alpha)$ be a right $\G$-$W^*$-algebra. As in \cite{DCK24}*{Section 4}, there exists a unique unital normal injective $*$-homomorphism $\alpha^u: A \to A \ovot C_0^u(\G)^{**}$ such that for every $(\mcH, \pi, U)\in \Rep^\G(A)= \Corr^\G(A,\C)$ (where $\pi$ is not necessarily faithful), we have
\begin{equation}\label{implementinguniversalcoaction}
    (\pi\otimes \id)(\alpha^u(a))= \uU(\pi(a)\otimes 1)\uU^*, \quad a \in A,
\end{equation}
where we view $\uU\in M(\mathcal{K}(\mcH)\otimes C_0^u(\G))\subseteq B(\mcH)\ovot C_0^u(\G)^{**}$.

On the other hand, considering the universal version $R_u: C_0^u(\G)\to C_0^u(\G)$ of the unitary antipode uniquely determined by $(R_u\otimes \hat{R})(\Ww)= \Ww$ \cite{Kus01}*{Proposition 7.2}, and considering its unique normal extension $\tilde{R}_u: C_0^u(\G)^{**}\to C_0^u(\G)^{**}$, it follows similarly as in \cite{DCK24}*{Lemma 4.6} that for every $(\mcH, \rho, U)\in \Corr^\G(\C,B)$, we have that
\begin{equation}\label{implementinguniversalcoaction1}
    (\rho\otimes \tilde{R}_u)(\beta^u(b)) = \uU^*(\rho(b)\otimes 1)\uU, \quad b\in B.
\end{equation}

Note also that (cfr.\ \cite{DCK24}*{Lemma 4.7})
\begin{equation}\label{DCK}
    (\alpha\otimes \id)\circ \alpha^u = (\id \otimes \Delta^{r,u,\operatorname{VN}})\circ \alpha,
\end{equation}
where
$\Delta^{r,u, \operatorname{VN}}: L^\infty(\G)\to L^\infty(\G)\ovot C_0^u(\G)^{**}$ is given by $\Delta^{r,u, \operatorname{VN}}(x)= \vV(x\otimes 1)\vV^*$.

By the universal property of the bidual, the non-degenerate $*$-homomorphism $$\Delta^u: C_0^u(\G)\to M(C_0^u(\G)\otimes C_0^u(\G))$$ can be uniquely extended to a unital normal $*$-homomorphism $$\tilde{\Delta}^u: C_0^u(\G)^{**}\to C_0^u(\G)^{**}\ovot C_0^u(\G)^{**}.$$ Using this observation, we can make sense of the coaction condition for $\alpha^u$:

\begin{Lem}\label{coactionproperty}
    The rectangle
    $$
\begin{tikzcd}
A \arrow[d, "{\alpha^{u}}"'] \arrow[rr, "{\alpha^{u}}"] &  & A\ovot C_0^u(\G)^{**} \arrow[d, "{\alpha^{u}\otimes \id}"] \\
A\ovot C_0^u(\G)^{**} \arrow[rr, "\id \otimes \tilde{\Delta}^u"']    &  & A \ovot C_0^u(\G)^{**}\ovot C_0^u(\G)^{**}                    
\end{tikzcd}$$
commutes.
\end{Lem}
\begin{proof}
Let $\tilde{\pi}_\G: C_0^u(\G)^{**}\to L^\infty(\G)$ be the unique unital, normal $*$-homomorphism extending $\pi_\G: C_0^u(\G)\to C_0^r(\G)$. We then note that we have the commutation relation
$$\Delta^{r,u,\operatorname{VN}}\circ \tilde{\pi}_\G= (\tilde{\pi}_\G\otimes \id) \circ \tilde{\Delta}^u.$$
Indeed, this commutation relation follows from the fact that
$\Delta^{r,u, \operatorname{VN}}(\pi_\G(x)) = (\pi_\G \otimes \id)\Delta^u(x)$ for $x\in C_0^u(\G)$ and the $\sigma$-weak density of $C_0^u(\G)$ inside $C_0^u(\G)^{**}$. Consequently, we can calculate for $x\in C_0^u(\G)^{**}$ that
\begin{align*}
    (\id \otimes \tilde{\Delta}^u)\Delta^{r,u, \operatorname{VN}}(\tilde{\pi}_\G(x))&= (\id \otimes \tilde{\Delta}^u)(\tilde{\pi}_\G \otimes \id) \tilde{\Delta}^u(x)\\
    &= (\tilde{\pi}_\G \otimes \id \otimes \id)(\tilde{\Delta}^u\otimes \id)\tilde{\Delta}^u(x)\\
    &= (\Delta^{r,u, \operatorname{VN}} \otimes \id)(\tilde{\pi}_\G \otimes \id)\tilde{\Delta}^u(x)\\
    &= (\Delta^{r,u, \operatorname{VN}} \otimes \id) \Delta^{r,u, \operatorname{VN}}(\tilde{\pi}_\G(x)),
\end{align*}
so it follows that
$$(\id \otimes \tilde{\Delta}^u) \Delta^{r,u,\operatorname{VN}}= (\Delta^{r,u, \operatorname{VN}}\otimes \id)\Delta^{r,u, \operatorname{VN}}.$$
As a result, making use of \eqref{DCK}, we find
\begin{align*}
    (\alpha\otimes \id \otimes \id)(\alpha^u\otimes \id)\alpha^u &= (\id \otimes \Delta^{r,u,\operatorname{VN}}\otimes \id)(\alpha\otimes \id)\alpha^u\\
    &=(\id \otimes \Delta^{r,u,\operatorname{VN}}\otimes \id)(\id \otimes \Delta^{r,u,\operatorname{VN}})\alpha\\
    &= (\id \otimes \id \otimes \tilde{\Delta}^u) (\id \otimes \Delta^{r,u,\operatorname{VN}})\alpha\\
    &= (\id \otimes \id \otimes \tilde{\Delta}^u) (\alpha\otimes \id) \alpha^u\\
    &= (\alpha\otimes \id\otimes \id)(\id \otimes \tilde{\Delta}^u)\alpha^u.
\end{align*}
Since $\alpha \otimes \id \otimes \id$ is injective, we conclude that $(\alpha^u\otimes \id)\alpha^u = (\id \otimes\tilde{\Delta}^u)\alpha^u$.
\end{proof}
We can now define the $W^*$-module category $\Rep(A)\curvearrowleft \Rep(\hat{\G})$. Given $\mcK \in \Rep(\hat{\G})$, consider the associated non-degenerate $*$-representation $\phi_\mcK: C_0^u(\G)\to B(\mcK)$. By the universal property of the bidual, it extends uniquely to a normal, unital $*$-representation $\tilde{\phi}_\mcK: C_0^u(\G)^{**}\to B(\mcK)$. If also $\mcH\in \Rep(A)$, we define the normal, unital $*$-representation
$$\pi_{\mcH \lhd \mcK}: A \to B(\mcH\otimes \mcK): a \mapsto (\pi_\mcH \otimes \tilde{\phi}_\mcK)(\alpha^u(a)).$$
We then obtain the object $\mcH \lhd  \mcK:= (\mcH \otimes \mcK , \pi_{\mcH \lhd \mcK})\in \Rep(A)$. Moreover, it follows from Lemma \ref{coactionproperty} that 
$$\mcH \lhd (\mcK\obot \mcK') = (\mcH \lhd \mcK)\lhd \mcK',\quad \mcH\in \Rep(A), \quad \mcK, \mcK'\in \Rep(\hat{\G}).$$
If $\mcH, \mcH'\in \Rep(A)$  and $\mcK, \mcK'\in \Rep(\hat{\G})$, $x\in {}_A\mathscr{L}(\mcH, \mcH')$ and $y\in \mathscr{L}^{\hat{\G}}(\mcK, \mcK')$, it is easily verified that $x\lhd y := x\otimes y \in {}_A\mathscr{L}(\mcH \lhd \mcK, \mcH'\lhd \mcK')$. Consequently, we obtain
the $W^*$-module category $\Rep(A)\curvearrowleft \Rep(\hat{\G})$. 

By restriction, we then also obtain the $W^*$-module category $\Rep(A)\curvearrowleft \Rep(L^\infty(\G)).$ %Since $\G\cong \check{\check{\G}}$ via the map $x \mapsto u_\G x u_\G$, we obtain the isomorphism
%$$\Rep(L^\infty(\check{\check{\G}}))\cong \Rep(L^\infty(\G)): (\mcH,\pi) \mapsto (\mcH,\pi\circ \operatorname{Ad}(u_\G))$$
%of $W^*$-tensor categories. We thus obtain the $W^*$-module category $\Rep(A)\curvearrowleft \Rep(L^\infty(\G))$ in a natural way.

\begin{Rem}\label{algebraicpicture} Let $H$ be a bialgebra and let $(A, \alpha)$ be a (right) $H$-comodule algebra, i.e. $A$ is a unital algebra and $\alpha: A \to A \odot H: a \mapsto a_{(0)}\otimes a_{(1)}$ is a unital algebra homomorphism satisfying $(\id \odot \Delta)\alpha = (\alpha \odot \id)\alpha$ and $(\id \odot \epsilon)\alpha = \id$ . Given a left $A$-module $V$ and a left $H$-module $W$, the tensor product $V \odot W$ becomes a left $A$-module via 
    $$a\rhd (v \otimes w) := (a_{(0)}\rhd v)\otimes (a_{(1)}\rhd w), \quad a\in A, \quad v\in V, \quad w \in W.$$
The construction of the $W^*$-module category $\Rep(A)\curvearrowleft \Rep(\hat{\G})$ is an analytic version of this well-known algebraic construction. 
\end{Rem}

\begin{Rem} If $\mcH \in \Rep(A)$ and $\mcK\in \Rep(\hat{\G})$, there exists a Hilbert space $\mcG$ together with an isometry $\kappa\in {}_A\mathscr{L}(\mcH, \mcG\otimes L^2(A))$, where $A \curvearrowright \mcG\otimes L^2(A)$ via $a \mapsto 1\otimes \pi_A(a)$. We then have
\begin{equation}\label{alg}
    \pi_{\mcH\lhd \mcK}(a)= (\kappa^*\otimes 1) \uU_{\alpha,23}^{\mcK} (1\otimes \pi_A(a)\otimes 1)\uU_{\alpha,23}^{\mcK,*}(\kappa\otimes 1), \quad a \in A.
\end{equation}
In fact, the formula \eqref{alg} can be used to \emph{define} the action $\Rep(A)\curvearrowleft \Rep(\hat{\G})$ (using this as a starting point, it takes some work to show that it does not depend on the choice of Hilbert space $\mcG$ and the choice of isometry $\kappa$). The advantage of this approach is that it avoids the use of the bidual von Neumann algebra $C_0^u(\G)^{**}$ and the `coaction' $\alpha^u$, but the downside is that the analogy with the algebraic picture (Remark \ref{algebraicpicture}) becomes less transparent, and the proofs become harder to read.
\begin{Exa}
    \begin{enumerate}\setlength\itemsep{-0.5em}
    \item $\pi_{L^2(A)\lhd (L^2(\G), \hat{W}_{21})}(a)= U_\alpha(\pi_A(a)\otimes 1)U_\alpha^*= (\pi_A \otimes \id)\alpha(a)$ for all $a\in A$.
        \item If $\mcH\in \Rep(A)$ and if $(\mcK, \mathbb{I})$ is the trivial $\hat{\G}$-representation, then
$\pi_{\mcH\lhd \mcK}(a)= \pi_\mcH(a)\otimes 1$ for all $a \in A.$
\item  If $A\curvearrowleft \G$ trivially, $\mcH\in \Rep(A)$ and $\mcK \in \Rep(\hat{\G})$, then
$\pi_{\mcH\lhd \mcK}(a)= \pi_{\mcH}(a)\otimes 1$ for all $a\in A$.
    \end{enumerate}
\end{Exa}

\end{Rem}

\subsection{From equivariant correspondence to module functor.}

\begin{Prop}
    For every triple $(\mathcal{G}, \mathcal{H}, \mcK)\in \Corr^\G(A,B)\times \Rep(B)\times \Rep(\hat{\G})$, there is a canonical unitary isomorphism
    $$S_{\mcG, \mcH, \mcK}\in {}_A\mathscr{L} (\mathcal{G}\boxtimes_B(\mcH \lhd \mcK), (\mcG \boxtimes_B \mcH)\lhd \mcK).$$
    These unitaries are natural in $\mcG, \mcH, \mcK$. Fixing $\mathcal{G}\in \Corr^\G(A,B)$, we have that
    $$\hat{P}(\mcG):= (F_\mcG, \{S_{\mcG, \mcH, \mcK}\}_{\mcH \in \Rep(B), \mcK \in \Rep(\hat{\G})})\in \Fun_{\Rep(\hat{\G})}(\Rep(B), \Rep(A)).$$
Given $y \in {}_A\mathscr{L}^\G_B(\mcG, \mcG')$, we have
    $$\{\hat{P}(y)_\mcH:= y\boxtimes_B \id_\mcH: \mcG\boxtimes_B \mcH\to \mcG'\boxtimes_B \mcH: x\otimes_B \xi \mapsto yx \otimes_B \xi\}_{\mcH\in \Rep(B)}\in \Nat_{\Rep(\hat{\G})}(F_\mcG, F_{\mcG'}).$$
    In this way, we obtain the functor
    $$\hat{P}: \Corr^\G(A,B)\to \Fun_{\Rep(\hat{\G})}(\Rep(B), \Rep(A)).$$
\end{Prop}

\begin{proof}
    Fix $\mathcal{G} \in \Corr^\G(A,B)$ and $\mcK \in \Rep(\hat{\G})$. Consider the unitary isomorphism
$$I_\mcG: \mcG \boxtimes_B L^2(B)\to \mcG: x\otimes_B \xi \mapsto x\xi$$
of $\G$-$A$-$B$-correspondences. Next, we note that there is an $A$-linear unitary
 $$D_{\mcG, \mcK}: \mcG \boxtimes_B (L^2(B)\lhd \mcK)\cong \mcG \lhd \mcK: x\otimes_B (\xi \otimes \eta)\mapsto \uU_{\mcG}^{\mcK}(x\otimes 1)\uU_{\beta}^{\mcK, *}(\xi \otimes \eta).$$ 
 It follows from \eqref{implementinguniversalcoaction1} that
$$D_{\mcG, \mcK}\circ (1\boxtimes_B (\rho_B(b)\otimes 1)) = (\rho_\mcG(b)\otimes 1) \circ D_{\mcG, \mcK}, \quad b \in B.$$
 We then define the $A$-linear unitary
 $$S_{\mcG, \mcK}:= (I_\mcG^*\otimes \id_{\mcK})\circ D_{\mcG, \mcK}: \mcG \boxtimes_B (L^2(B)\lhd \mcK) \cong (\mcG \boxtimes_B L^2(B))\lhd \mcK.$$
For all $b \in B$, we have $S_{\mcG, \mcK}(1\boxtimes_B (\rho_B(b)\otimes 1)) = ((1\boxtimes_B \rho_B(b))\otimes 1) S_{\mcG, \mcK}$. By the bijectivity of \eqref{Rieffelbijection}, there exists a unique unitary natural transformation
$S_{\mcG, -, \mcK}: \mcG \boxtimes_B (-\lhd \mcK)\implies (\mcG \boxtimes_B -)\lhd \mcK$
such that $S_{\mcG, L^2(B), \mcK}= S_{\mcG, \mcK}$. We now show that the unitary isomorphisms
$$S_{\mcG, \mcH, \mcK}: \mcG \boxtimes_B (\mcH \lhd \mcK)\to (\mcG \boxtimes_B \mcH)\lhd \mcK, \quad \mcG \in \Corr^\G(A,B), \quad \mcH \in \Rep(B), \quad \mcK \in \Rep(\hat{\G})$$
are also natural in the variables $\mcG$ and $\mcK$. That is, fixing $\mcG, \mcG'\in \Corr^\G(A,B)$, $\mcH\in \Rep(B), \mcK, \mcK'\in \Rep(\hat{\G})$, $x\in {}_A \mathscr{L}_B^\G(\mcG, \mcG')$ and $y \in \mathscr{L}^{\hat{\G}}(\mcK, \mcK')$, we need to show that the  rectangle
\begin{equation}\label{commdiag}
\begin{tikzcd}
\mcG \boxtimes_B(\mcH\lhd \mcK) \arrow[d, "x\boxtimes_B (1\lhd y)"] \arrow[rr, "{S_{\mcG, \mcH, \mcK}}"] &  & (\mcG \boxtimes_B \mcH)\lhd \mcK \arrow[d, "(x\boxtimes_B 1)\lhd y"] \\
\mcG'\boxtimes_B(\mcH\lhd \mcK') \arrow[rr, "{S_{\mcG', \mcH, \mcK'}}"]                                   &  & (\mcG' \boxtimes_B \mcH)\lhd \mcK'                                   
\end{tikzcd}
\end{equation}
commutes. But a direct computation shows that the rectangle \eqref{commdiag} commutes for the generator $\mathcal{H}= L^2(B)\in \Rep(B)$. Since both compositions in the rectangle \eqref{commdiag} are natural in the variable $\mcH\in \Rep(B)$, it follows from the injectivity of \eqref{Rieffelbijection} that \eqref{commdiag} commutes. 

Fixing $\mathcal{G}\in \Corr^\G(A,B)$, $\mathcal{H}\in \Rep(B)$ and $\mathcal{K}, \mcK'\in \Rep(\hat{\G})$, we show next that the diagram
\begin{equation}\label{diag}
\begin{tikzcd}
\mcG \boxtimes_B (\mcH \lhd (\mcK \obot \mcK'))  \arrow[d, "="', no head] \arrow[rr, "{S_{\mcG, \mcH,\mcK\obot \mcK'}}"] &  & (\mcG \boxtimes_B \mcH)\lhd (\mcK\obot \mcK') \arrow[rr, "=", no head] &  & ((\mcG \boxtimes_B \mcH)\lhd \mcK)\lhd \mcK'                                                \\
\mcG \boxtimes_B ((\mcH \lhd \mcK)\lhd \mcK') \arrow[rrrr, "{S_{\mcG, \mcH\lhd \mcK, \mcK'}}"]                               &  &                                                                          &  & (\mcG \boxtimes_B (\mcH \lhd \mcK))\lhd \mcK' \arrow[u, "{S_{\mcG, \mcH, \mcK}\otimes 1}"']
\end{tikzcd}
\end{equation}
commutes. Again, it suffices to show the commutativity for $\mcH = L^2(B)$. To do this, we start by noting that
$\uU_\beta^\mcK \in {}_B\mathscr{L}(L^2(B)\otimes \mcK, L^2(B)\lhd \mcK),$
where $B \curvearrowright L^2(B)\otimes \mcK$ via $b \mapsto \pi_B(b)\otimes 1$. Consider then the diagram
\begin{equation}\label{diagrams}
\begin{tikzcd}
\mcG \boxtimes_B ((L^2(B) \lhd \mcK) \lhd \mcK') \arrow[rrr, "{S_{\mcG, L^2(B)\lhd \mcK, \mcK'}}"]                                                               &  &  & (\mcG \boxtimes_B (L^2(B) \lhd \mcK)) \lhd \mcK'                                                            \\
\mcG \boxtimes_B ((L^2(B) \otimes \mcK) \lhd \mcK') \arrow[rrr, "{S_{\mcG, L^2(B)\otimes \mcK, \mcK'}}"] \arrow[u, "{1\boxtimes_B (\uU_{\beta}^{\mcK} \otimes 1)}"] &  &  & (\mcG \boxtimes_B (L^2(B) \otimes \mcK)) \lhd \mcK' \arrow[u, "{(1\boxtimes_B \uU_{\beta}^{\mcK})\otimes 1}"'] \\
(\mcG \boxtimes_B (L^2(B)\lhd \mcK'))\otimes \mcK \arrow[u, "C_{L^2(B), \mcK}^{\mcG \boxtimes_B(-\lhd \mcK')}"] \arrow[rrr, "{S_{\mcG,L^2(B), \mcK'}\otimes \id_\mcK}"]                                  &  &  & ((\mcG \boxtimes_B L^2(B))\lhd \mcK')\otimes \mcK \arrow[u, "C_{L^2(B), \mcK}^{(\mcG \boxtimes_B -)\lhd \mcK'}"']                                     
\end{tikzcd}
\end{equation}
where the isomorphisms in the lower rectangle are the canonical isomorphisms from \eqref{canonicalisomorphismmultiplicity}. Here, the upper rectangle commutes by naturality in the middle variable, and the lower rectangle commutes by the commutativity of the diagram \eqref{canmultiiso}. 
It is straightforward to verify that
\begin{align*}
    &C_{L^2(B), \mcK}^{\mcG \boxtimes_B (-\lhd \mcK')}((x\otimes_B (\xi \otimes \eta'))\otimes \eta) = x\otimes_B (\xi \otimes \eta\otimes \eta'),\\
    &C_{L^2(B), \mcK}^{(\mcG\boxtimes_B -)\lhd \mcK'}((x\otimes_B \xi)\otimes \eta'\otimes \eta)= (x\otimes_B (\xi \otimes \eta))\otimes \eta'.
\end{align*}
for every $x \in \mathscr{L}_B(L^2(B), \mcG), \xi \in L^2(B), \eta \in \mcK$ and $\eta' \in \mcK'$. It follows from the commutativity of the diagram \eqref{diagrams} that the commutativity of the diagram \eqref{diag} for $\mcH = L^2(B)$ is equivalent with 
\begin{equation}\label{comp}
      (S_{\mcG, L^2(B), \mcK}\otimes 1)((1\underset{B}{\boxtimes} \uU_{\beta}^{\mcK})\otimes 1)C_{L^2(B), \mcK}^{(\mcG\underset{B}{\boxtimes} -)\lhd \mcK'} (S_{\mcG, L^2(B), \mcK'}\otimes \id_\mcK) C_{L^2(B), \mcK}^{\mcG\underset{B}{\boxtimes} (-\lhd \mcK'),*} (1\underset{B}{\boxtimes} (\uU_{\beta}^{\mcK,*}\otimes 1))=S_{\mcG, L^2(B), \mcK\obot \mcK'}.
\end{equation}

Let $x\in \mathscr{L}_B(L^2(B), \mcG), y \in B(\mcK'), \xi \in L^2(B), \eta\in \mcK$ and $\eta' \in \mcK'$. Then
\begin{align*}
    &(I_\mcG \otimes 1\otimes 1)C_{L^2(B), \mcK}^{(\mcG\boxtimes_B-)\lhd \mcK',*}((1\boxtimes_B \uU_{\beta}^{ \mcK,*})\otimes 1)(D_{\mcG, \mcK}^*\otimes 1)(\uU_{\mcG, 12}^{\mcK}(x\otimes 1 \otimes y)\uU_{\beta, 12}^{\mcK,*}(\xi \otimes \eta \otimes \eta'))\\
    &= (x\otimes y\otimes 1)\uU_{\beta,13}^{\mcK, *}(\xi \otimes \eta' \otimes \eta).
\end{align*}
Consequently,
\begin{align*}
    &(I_\mG \otimes 1 \otimes 1)C_{L^2(B), \mcK}^{(\mcG \boxtimes_B-) \lhd \mcK', *}((1\boxtimes_B \uU_{\beta}^{\mcK,*})\otimes 1)(D_{\mcG, \mcK}^*\otimes 1)D_{\mcG, \mcK \obot \mcK'}(x\otimes_B(\xi \otimes \eta \otimes \eta'))\\
    &=(I_\mG \otimes 1 \otimes 1)C_{L^2(B), \mcK}^{(\mcG \underset{B}{\boxtimes}-) \lhd \mcK', *}((1\underset{B}{\boxtimes} \uU_{\beta}^{\mcK,*})\otimes 1)(D_{\mcG, \mcK}^*\otimes 1)(\uU_{\mcG,12}^{\mcK}\uU_{\mcG, 13}^{\mcK'}(x\otimes 1\otimes 1)\uU_{\beta,13}^{\mcK',*}\uU_{\beta,12}^{\mcK,*}(\xi \otimes \eta \otimes \eta'))\\
    &= \uU_{\mG,12}^{\mcK'}(x\otimes 1\otimes 1)\uU_{\beta,12}^{\mcK', *} \uU_{\beta,13}^{\mcK,*} (\xi \otimes \eta' \otimes \eta)\\
    &= (D_{\mcG, \mcK'}\otimes 1)C_{L^2(B), \mcK}^{\mcG\boxtimes_B(-\lhd \mcK'),*}(1\boxtimes_B (\uU_{\beta}^{\mcK,*}\otimes 1))(x\otimes_B (\xi \otimes \eta \otimes  \eta')).
\end{align*}
This shows that 
$$(I_\mG \otimes 1 \otimes 1)C_{L^2(B), \mcK}^{(\mcG \boxtimes_B-) \lhd \mcK', *}((1\boxtimes_B \uU_{\beta}^{\mcK,*})\otimes 1)(D_{\mcG, \mcK}^*\otimes 1) D_{\mcG, \mcK\obot \mcK'}= (D_{\mcG, \mcK'}\otimes 1)C_{L^2(B), \mcK}^{\mcG \boxtimes_B (-\lhd \mcK'),*} (1\boxtimes_B (\uU_{\beta}^{\mcK,*}\otimes 1)).$$
Rearranging, we find
$$(D_{\mcG, \mcK}\otimes 1)((1\boxtimes_B \uU_{\beta}^{\mcK})\otimes 1)C_{L^2(B), \mcK}^{(\mcG\boxtimes_B-)\lhd \mcK'}(I_\mcG^*\otimes 1 \otimes 1)(D_{\mcG, \mcK'}\otimes 1)C_{L^2(B), \mcK}^{\mcG\boxtimes_B(-\lhd \mcK'),*}(1\boxtimes_B (\uU_{\beta}^{\mcK,*}\otimes 1)) = D_{\mcG, \mcK \obot \mcK'}.$$
Multiplying both sides on the left with $I_{\mG}^*\otimes 1 \otimes 1$ and plugging in the definition of the maps $S_{\mcG, L^2(B), \mcK}$, we find the equality \eqref{comp}.
\end{proof}
\subsection{From module functor to equivariant correspondence.}

\begin{Prop} Let $F,G\in \Fun_{\Rep(\hat{\G})}(\Rep(B), \Rep(A))$.
    \begin{enumerate}\setlength\itemsep{-0.5em}
 \item Consider $F((L^2(B), \pi_B))= (\mcG_F, \pi_F)\in \Rep(A)$. 
     Writing 
    $\rho_F(b):= F(\rho_B(b))$ for $b\in B$ and defining 
    \begin{equation}\label{definitionUF}
        U_F:= S_{L^2(B), (L^2(\G), \hat{W}_{21})}\circ F(U_\beta)\circ S_{L^2(B), (L^2(\G), \mathbb{I})}^*\in B(\mcG_F\otimes L^2(\G)),
    \end{equation}
    we have that $(\mcG_F, \pi_F, \rho_F, U_F)\in \Corr^\G(A,B)$. 
    \item If $\eta \in \Nat_{\Rep(\hat{\G})}(F,G)$, then $\eta_{L^2(B)}\in {}_A\mathscr{L}_B^\G(\mcG_F, \mcG_G).$
    \item The assignments $F\mapsto (\mcG_F, \pi_F, \rho_F, U_F)$ and $\eta \mapsto \eta_{L^2(B)}$
    define a functor 
    $$\hat{Q}: \Fun_{\Rep(\hat{\G})}(\Rep(B), \Rep(A))\to \Corr^\G(A,B).$$
    \item $\hat{Q}\circ \hat{P}\cong \id_{\Corr^\G(A,B)}.$
    \end{enumerate}
\end{Prop}
\begin{proof} (1) To simplify notation, we use the following notations in this proof:
    \begin{align*}
    S_{\mathbb{I}}^{1,1}&:= S_{L^2(B), (L^2(\G), \mathbb{I})}, &
    S_{\hat{W}_{21}}^{1,1}&:= S_{L^2(B), (L^2(\G), \hat{W}_{21})},\\ 
    S_{\mathbb{I}}^{1,2}&:= S_{L^2(B), (L^2(\G), \mathbb{I})\obot (L^2(\G), \mathbb{I})},&
    S_{\hat{W}_{31}}^{1,2}&:= S_{L^2(B), (L^2(\G),\hat{W}_{21})\obot (L^2(\G), \mathbb{I})},\\
        S_{\hat{W}_{32}}^{1,2}&:= S_{L^2(B), (L^2(\G), \mathbb{I})\obot (L^2(\G), \hat{W}_{21})},&
        S_{\hat{W}_{32}\hat{W}_{31}}^{1,2}&:= S_{L^2(B), (L^2(\G),\hat{W}_{21})\obot (L^2(\G),\hat{W}_{21})}.
    \end{align*}

 We start by noting that 
    $U_\beta \in {}_B\mathscr{L}(L^2(B)\lhd (L^2(\G),\mathbb{I}), L^2(B)\lhd (L^2(\G), \hat{W}_{21})),$
    since for $b\in B$, we have
    $$U_\beta \pi_{L^2(B)\lhd (L^2(\G), \mathbb{I})}(b) = U_\beta(\pi_B(b)\otimes 1) = (\pi_B\otimes \id)(\beta(b)) U_\beta = \pi_{L^2(B)\lhd (L^2(\G), \hat{W}_{21})}(b) U_\beta.$$
    Consequently, it makes sense to define the unitary $U_F: \mcG_F\otimes L^2(\G)\to \mcG_F\otimes L^2(\G)$ given by $U_F:= S_{\hat{W}_{21}}^{1,1}\circ F(U_\beta)\circ S_\mathbb{I}^{1,1,*}$. If $x\in L^\infty(\G)'= \mathscr{L}^{\hat{\G}}((L^2(\G), \hat{W}_{21}))$, we have
    \begin{align*}
        U_F(1\otimes x) &= S_{\hat{W}_{21}}^{1,1} F(U_\beta) S_{\mathbb{I}}^{1,1,*}(1\otimes x) = S_{\hat{W}_{21}}^{1,1} F(U_\beta) F(1\otimes x) S_\mathbb{I}^{1,1,*} = S_{\hat{W}_{21}}^{1,1} F(1\otimes x)F(U_\beta)S_\mathbb{I}^{1,1,*} \\
        &= (1\otimes x) S_{\hat{W}_{21}}^{1,1} F(U_\beta)S_{\mathbb{I}}^{1,1,*} = (1\otimes x) U_F, 
    \end{align*}
    and we conclude that $U_F \in B(\mcG_F) \ovot L^\infty(\G)$. It is easily verified that
    $$S^{1,2}_{\hat{W}_{31}} F(U_{\beta,12}) S^{1,2,*}_{\mathbb{I}}= U_F\otimes 1= S^{1,2}_{\hat{W}_{32}\hat{W}_{31}} F(U_{\beta,12}) S^{1,2,*}_{\hat{W}_{32}},$$
    so that also
    $$U_{F,13}= (1\otimes \Sigma)(U_F\otimes 1)(1\otimes \Sigma)= S^{1,2}_{\hat{W}_{32}} F(U_{\beta,13}) S^{1,2,*}_{\mathbb{I}}.$$
    Using the fact that
    $$W \in \mathscr{L}^{\hat{\G}}((L^2(\G), \hat{W}_{21})\obot (L^2(\G), \hat{W}_{21}), (L^2(\G), \mathbb{I})\obot (L^2(\G), \hat{W}_{21})),$$
    we then calculate
    \begin{align*}W_{23} U_{F,12}U_{F,13} &= W_{23} (S^{1,2}_{\hat{W}_{32}\hat{W}_{31}} F(U_{\beta,12}) S^{1,2,*}_{\hat{W}_{32}})(S^{1,2}_{\hat{W}_{32}} F(U_{\beta,13}) S^{1,2,*}_\mathbb{I})\\
    &= S^{1,2}_{\hat{W}_{32}} F(W_{23}U_{\beta,12}U_{\beta,13}) S_\mathbb{I}^{1,2,*}\\
    &= S_{\hat{W}_{32}}^{1,2} F(U_{\beta,13}W_{23}) S_\mathbb{I}^{1,2,*}\\
    &= S_{\hat{W}_{32}}^{1,2} F(U_{\beta,13}) S_\mathbb{I}^{1,2,*} W_{23} =  U_{F,13} W_{23},
        \end{align*}
and it follows that $U_F$ is a unitary $\G$-representation. 

By construction, we have that $U_F\in {}_A\mathscr{L}(\mcG_F\lhd (L^2(\G), \mathbb{I}), \mcG_F\lhd (L^2(\G), \hat{W}_{21}))$. Therefore, given $a\in A$, it follows that 
\begin{align*}
    U_F(\pi_F(a)\otimes 1) &= U_F \pi_{\mcG_F \lhd (L^2(\G), \mathbb{I})}(a)= \pi_{\mcG_F\lhd (L^2(\G), \hat{W}_{21})}(a) U_F = (\pi_F \otimes \id)(\alpha(a))U_F.
\end{align*}
On the other hand, we have
${}_B\mathscr{L}(L^2(B)\lhd (L^2(\G), \mathbb{I})) = \rho_B(B)\ovot B(L^2(\G))$. For $b\in B$ and $x\in B(L^2(\G))$, we find
$$S_{L^2(B), (L^2(\G),\mathbb{I})} F(\rho_B(b)\otimes x) S_{L^2(B), (L^2(\G),\mathbb{I})}^*= F(\rho_B(b))\otimes x = \rho_F(b)\otimes x,$$
so by normality of $F$, it follows that
$$S_{L^2(B), (L^2(\G),\mathbb{I})}F((\rho_B\otimes R)(\beta(b))) S_{L^2(B), (L^2(\G),\mathbb{I})}^*= (\rho_F\otimes R)(\beta(b)).$$
Consequently, for $b\in B$,
\begin{align*}
    (\rho_F(b)\otimes 1)U_F &= (F(\rho_B(b))\otimes 1)S_{L^2(B), (L^2(\G), \hat{W}_{21})} F(U_\beta)S_{L^2(B), (L^2(\G), \mathbb{I})}^* \\
    &= S_{L^2(B), (L^2(\G), \hat{W}_{21})} F((\rho_B(b)\otimes 1)U_\beta)S_{L^2(B), (L^2(\G),\mathbb{I})}^*\\
    &= S_{L^2(B), (L^2(\G),\hat{W}_{21})} F(U_\beta(\rho_B\otimes R)(\beta(b))) S_{L^2(B),(L^2(\G),\mathbb{I})}^*\\
    &= S_{L^2(B), (L^2(\G),\hat{W}_{21})} F(U_\beta) F((\rho_B\otimes R)(\beta(b)))S_{L^2(B), (L^2(\G), \mathbb{I})}^*\\
    &= S_{L^2(B), (L^2(\G),\hat{W}_{21})} F(U_\beta) S_{L^2(B), (L^2(\G), \mathbb{I})}^* (\rho_F\otimes R)(\beta(b))= U_F(\rho_F\otimes R)(\beta(b)).
\end{align*}
We conclude that $(\mcG_F, \pi_F, \rho_F, U_F)\in \Corr^\G(A,B)$, as desired.

(2)   We calculate
    \begin{align*}
        (\eta_{L^2(B)}\otimes 1)U_F &= (\eta_{L^2(B)}\otimes 1)S_{L^2(B), (L^2(\G), \hat{W}_{21})}F(U_\beta) S_{L^2(B), (L^2(\G), \mathbb{I})}^* \\
        &= S_{L^2(B), (L^2(\G), \hat{W}_{21})} \eta_{L^2(B)\lhd (L^2(\G), \hat{W}_{21})} F(U_\beta) S_{L^2(B), (L^2(\G), \mathbb{I})}^*\\
        &= S_{L^2(B), (L^2(\G), \hat{W}_{21})} G(U_\beta) \eta_{L^2(B)\lhd (L^2(\G), \mathbb{I})} S_{L^2(B), (L^2(\G), \mathbb{I})}^*= U_G(\eta_{L^2(B)}\otimes 1).
    \end{align*}

    (3) Clear from the two previous steps.

    (4) Given $\mcG \in \Corr^\G(A,B)$, it suffices to show that $\hat{Q}(F_\mcG) = (\mcG\boxtimes_B L^2(B), \pi_\boxtimes, \rho_\boxtimes, U_\boxtimes)$. Consider the unitary identification
    $$T: (\mcG\otimes L^2(\G))\boxtimes_{B \ovot L^\infty(\G)} (L^2(B)\otimes L^2(\G))\to (\mcG \boxtimes_B L^2(B))\otimes L^2(\G)$$
    from \eqref{naturalunitary}. With $\mcK = (L^2(\G), \mathbb{I})\in \Rep(\hat{\G})$ and a unitary $\G$-representation $U\in B(\mcH)\ovot L^\infty(\G)$, we have $\uU^\mcK = 1 \otimes 1$ and with $\mcK= (L^2(\G),\hat{W}_{21})\in \Rep(\hat{\G})$, we have $\uU^{\mcK}= U$. 
    Then we calculate for $x\in \mathscr{L}_B(L^2(B), \mcG), \xi \in L^2(B)$ and $\eta \in L^2(\G)$ that
    \begin{align*}
        (I_\mcG \otimes 1)U_\boxtimes S_{\mcG,L^2(B), (L^2(\G), \mathbb{I})}(x\otimes_B (\xi \otimes \eta)) &= (I_\mcG \otimes 1)U_\boxtimes((x\otimes_B \xi)\otimes \eta)\\
        &= (I_\mcG \otimes 1) T(U_{\mcG}(x\otimes 1)U_\beta^*\otimes_{B \ovot L^\infty(\G)} U_\beta(\xi \otimes \eta))\\
        &= U_\mcG(x\otimes 1)U_\beta^* U_\beta(\xi \otimes \eta)\\
        &= (I_\mcG\otimes 1) S_{\mcG, L^2(B), (L^2(\G),\hat{W}_{21})}(1\boxtimes_B U_\beta)(x\otimes_B (\xi \otimes \eta)).
    \end{align*}
    It follows that 
       $U_\boxtimes = S_{\mcG, L^2(B), (L^2(\G), \hat{W}_{21})} \circ F_\mcG(U_\beta) \circ S_{\mcG, L^2(B), (L^2(\G), \mathbb{I})}^*  = U_{F_\mcG}.$
\end{proof}
\subsection{Equivariant Eilenberg-Watts theorem}

%\begin{Theorem}\label{copy}
%The restriction functor
%$$\Fun_{\Rep(\hat{\G})}(\Rep(B), \Rep(A))\to \Fun_{\Rep(L^\infty(\G))}(\Rep(B), \Rep(A))$$
%is an isomorphism of categories.
%\end{Theorem}
%\begin{proof} The proof is similar to the proof of Theorem \ref{restrictionsisos}.
%\end{proof}

\begin{Theorem}[Equivariant Eilenberg-Watts theorem II]\label{main2}
    The canonical functors
    \begin{align*}
        &\hat{P}: \Corr^\G(A,B)\to \Fun_{\Rep(\hat{\G})}(\Rep(B), \Rep(A)), \quad \hat{Q}: \Fun_{\Rep(\hat{\G})}(\Rep(B), \Rep(A))\to \Corr^\G(A,B)
    \end{align*}
    are quasi-inverse to each other.
\end{Theorem}
\begin{proof} Since $\hat{Q}\circ \hat{P}\cong \id_{\Corr^\G(A,B)}$, it suffices to prove that $\hat{Q}$ is fully faithful. In other words, given $F,G \in \Fun_{\Rep(\hat{\G})}(\Rep(B), \Rep(A))$, we need to prove that the map
    \begin{equation}\label{injsur}\Nat_{\Rep(\hat{\G})}(F,G)\to {}_A\mathscr{L}_B^\G(\mcG_F, \mcG_G): \eta \mapsto \eta_{L^2(B)}\end{equation}
    is bijective. We will follow the strategy from the proof of Proposition \ref{integrablefullyfaithful} for this. Since $L^2(B)$ is a generator for $\Rep(B)$, it is clear that \eqref{injsur} is injective. Next, we let $\eta_1\in {}_A\mathscr{L}^\G_B(\mcG_F, \mcG_G)$ and we define the $A$-linear operators
    \begin{align*}
        &\eta_2:= S_{L^2(B),(L^2(\G),\hat{W}_{21})}^*(\eta_1\otimes 1) S_{L^2(B), (L^2(\G),\hat{W}_{21})},\\
        &\eta_3:= S_{L^2(B)\lhd (L^2(\G),\hat{W}_{21}), (L^2(\G),\hat{W}_{21})}^*(\eta_2\otimes 1) S_{L^2(B)\lhd (L^2(\G), \hat{W}_{21}), (L^2(\G),\hat{W}_{21})}.
    \end{align*} 
    Note that
\begin{align*}
    {}_B\mathscr{L}(L^2(B)\lhd (L^2(\G),\hat{W}_{21})) &= U_\beta(\rho_B(B)\ovot B(L^2(\G)))U_\beta^*,\\
    {}_B\mathscr{L}(L^2(B)\lhd ((L^2(\G),\hat{W}_{21})\obot (L^2(\G), \hat{W}_{21}))) &= U_{\beta,12}U_{\beta,13}(\rho_B(B)\ovot B(L^2(\G))\ovot B(L^2(\G)))U_{\beta,13}^*U_{\beta,12}^*.
\end{align*}
We then calculate for $b\in B$ and $y \in B(L^2(\G))$ that
\begin{align*}
    \eta_2 F(U_\beta(\rho_B(b)\otimes y)U_\beta^*) &= S_{L^2(B), (L^2(\G), \hat{W}_{21})}^* (\eta_1\otimes 1)S_{L^2(B), (L^2(\G), \hat{W}_{21})} F(U_\beta) F(\rho_B(b)\otimes y)F(U_\beta^*)\\
    &= S_{L^2(B), (L^2(\G), \hat{W}_{21})}^* (\eta_1\otimes 1) U_F S_{L^2(B), (L^2(\G), \mathbb{I})} F(\rho_B(b)\otimes y) F(U_\beta^*)\\
    &= S_{L^2(B), (L^2(\G), \hat{W}_{21})}^* U_G (\eta_1\otimes 1) S_{L^2(B), (L^2(\G), \mathbb{I})} F(\rho_B(b)\otimes y) F(U_\beta^*)\\
    &= G(U_\beta) S_{L^2(B), (L^2(\G), \mathbb{I})}^* (\eta_1\otimes 1) (\rho_F(b)\otimes y)S_{L^2(B), (L^2(\G), \mathbb{I})} F(U_\beta^*)\\
    &= G(U_\beta) S_{L^2(B), (L^2(\G), \mathbb{I})}^* (\rho_G(b)\otimes y) (\eta_1\otimes 1) S_{L^2(B), (L^2(\G), \mathbb{I})} F(U_\beta^*)\\
    &= G(U_\beta) G(\rho_B(b)\otimes y) S_{L^2(B), (L^2(\G), \mathbb{I})}^* (\eta_1\otimes 1) S_{L^2(B), (L^2(\G), \mathbb{I})} F(U_\beta^*)\\
    &= G(U_\beta) G(\rho_B(b)\otimes y) S_{L^2(B), (L^2(\G), \mathbb{I})}^* (\eta_1\otimes 1) U_F^* S_{L^2(B), (L^2(\G), \hat{W}_{21})}\\
    &= G(U_\beta)G(\rho_B(b)\otimes y) S_{L^2(B), (L^2(\G), \mathbb{I})}^* U_G^* (\eta_1\otimes 1) S_{L^2(B), (L^2(\G), \hat{W}_{21})}\\
    &= G(U_\beta(\rho_B(b)\otimes y)U_\beta^*) \eta_2.
\end{align*}

Next, we claim that for $b\in B$ and $x,y \in B(L^2(\G))$, we have
\begin{equation}
    \eta_3 F(U_{\beta, 12}U_{\beta,13}(\rho_B(b)\otimes x\otimes y)U_{\beta,13}^* U_{\beta,12}^*)= G(U_{\beta, 12}U_{\beta,13}(\rho_B(b)\otimes x\otimes y)U_{\beta,13}^* U_{\beta,12}^*) \eta_3.
\end{equation}

Indeed, introducing the notations
\begin{align*}
    S_{\mathbb{I}}^{1,1}&:= S_{L^2(B), (L^2(\G), \mathbb{I})}, & S_{\hat{W}_{21}}^{1,1}&:= S_{L^2(B), (L^2(\G), \hat{W}_{21})},\\
    S_{\hat{W}_{32}\hat{W}_{31}}^{1,2}&:= S_{L^2(B), (L^2(\G), \hat{W}_{21})\otop (L^2(\G), \hat{W}_{21})}, & S_{\hat{W}_{31}}^{1,2}&:= S_{L^2(B), (L^2(\G), \hat{W}_{21})\otop (L^2(\G), \mathbb{I})},\\
    S_{\hat{W}_{32}}^{1,2}&:= S_{L^2(B), (L^2(\G), \mathbb{I})\otop (L^2(\G), \hat{W}_{21})}, & S_{\mathbb{I}, \hat{W}_{21}}^{2,1} &:= S_{L^2(B) \lhd (L^2(\G), \mathbb{I}), (L^2(\G), \hat{W}_{21})},\\ 
     S_{\hat{W}_{21}, \mathbb{I}}^{2,1} &:= S_{L^2(B) \lhd (L^2(\G), \hat{W}_{21}), (L^2(\G), \mathbb{I})}, & S_{\hat{W}_{21}, \hat{W}_{21}}^{2,1}&:= S_{L^2(B)\lhd (L^2(\G), \hat{W}_{21}), (L^2(\G), \hat{W}_{21})},
\end{align*}
we calculate:
\begin{align*}
    &\eta_3 F(U_{\beta,12}U_{\beta,13}(\rho_B(b)\otimes x\otimes y) U_{\beta,13}^* U_{\beta,12}^*) \\
    &= S_{\hat{W}_{21},\hat{W}_{21}}^{2,1,*} (\eta_2\otimes 1)S_{\hat{W}_{21},\hat{W}_{21}}^{2,1} F(U_{\beta,12})F(U_{\beta,13}(\rho_B(b)\otimes x \otimes y)U_{\beta,13}^*)F(U_{\beta,12}^*)\\
    &= S_{\hat{W}_{21},\hat{W}_{21}}^{2,1,*} (\eta_2\otimes 1)(F(U_\beta)\otimes 1)S_{\mathbb{I},\hat{W}_{21}}^{2,1} F(U_{\beta,13}(\rho_B(b)\otimes x\otimes y)U_{\beta,13}^*)F(U_{\beta,12}^*)\\
    &= S_{\hat{W}_{21},\hat{W}_{21}}^{2,1,*} (\eta_2\otimes 1)(F(U_\beta)\otimes 1)(S_{\mathbb{I}}^{1,1,*} \otimes 1)  S^{1,2}_{\hat{W}_{32}} F(U_{\beta,13}(\rho_B(b)\otimes x\otimes y)U_{\beta,13}^*)F(U_{\beta,12}^*)\\
&=  S_{\hat{W}_{21},\hat{W}_{21}}^{2,1,*} (\eta_2\otimes 1)(F(U_\beta)\otimes 1) (S_\mathbb{I}^{1,1,*}\otimes 1 ) S_{\hat{W}_{32}}^{1,2} F(\Sigma_{23}) F(U_{\beta,12}(\rho_B(b)\otimes y\otimes x)U_{\beta,12}^*) F(\Sigma_{23})F(U_{\beta,12}^*)\\
&= S_{\hat{W}_{21}, \hat{W}_{21}}^{2,1,*} (\eta_2\otimes 1)(F(U_\beta)\otimes 1)(S_\mathbb{I}^{1,1,*}\otimes 1) \Sigma_{23} S_{\hat{W}_{31}}^{1,2} F(U_{\beta,12}(\rho_B(b)\otimes y \otimes x)U_{\beta,12}^*)F(\Sigma_{23})F(U_{\beta,12}^*)\\
&= S_{\hat{W}_{21}, \hat{W}_{21}}^{2,1,*} (S_{\hat{W}_{21}}^{1,1,*}\otimes 1)(\eta_1\otimes 1\otimes 1) (U_F\otimes 1) \Sigma_{23}  S_{\hat{W}_{31}}^{1,2} F(U_{\beta,12}(\rho_B(b)\otimes y \otimes x)U_{\beta,12}^*)F(\Sigma_{23})F(U_{\beta,12}^*)\\
&=  S_{\hat{W}_{32}\hat{W}_{31}}^{1,2,*}(U_G\otimes 1)(\eta_1\otimes 1 \otimes 1)\Sigma_{23} S_{\hat{W}_{31}}^{1,2} F(U_{\beta,12}(\rho_B(b)\otimes y \otimes x)U_{\beta,12}^*)F(\Sigma_{23})F(U_{\beta,12}^*)\\
&= S_{\hat{W}_{32}\hat{W}_{31}}^{1,2,*}(S_{\hat{W}_{21}}^{1,1}\otimes 1)(G(U_\beta)\otimes 1)(S_{\mathbb{I}}^{1,1,*}\otimes 1)(\eta_1\otimes 1 \otimes 1)\Sigma_{23} S_{\hat{W}_{31}}^{1,2} F(U_{\beta,12}(\rho_B(b)\otimes y \otimes x)U_{\beta,12}^*)F(\Sigma_{23})F(U_{\beta,12}^*)\\
&=S_{\hat{W}_{21}, \hat{W}_{21}}^{2,1,*}(G(U_\beta)\otimes 1)(S_{\mathbb{I}}^{1,1,*}\otimes 1)(\eta_1\otimes 1 \otimes 1)\Sigma_{23} S_{\hat{W}_{31}}^{1,2} F(U_{\beta,12}(\rho_B(b)\otimes y \otimes x)U_{\beta,12}^*)F(\Sigma_{23})F(U_{\beta,12}^*)\\
&=G(U_{\beta,12}) S_{\mathbb{I}, \hat{W}_{21}}^{2,1,*}(S_{\mathbb{I}}^{1,1,*}\otimes 1)(\eta_1\otimes 1 \otimes 1)\Sigma_{23} S_{\hat{W}_{31}}^{1,2} F(U_{\beta,12}(\rho_B(b)\otimes y \otimes x)U_{\beta,12}^*)F(\Sigma_{23})F(U_{\beta,12}^*)\\
&= G(U_{\beta,12}) S_{\hat{W}_{32}}^{1,2,*} \Sigma_{23} (\eta_1\otimes 1 \otimes 1) S_{\hat{W}_{31}}^{1,2} F(U_{\beta,12}(\rho_B(b)\otimes y \otimes x)U_{\beta,12}^*)F(\Sigma_{23})F(U_{\beta,12}^*)\\
&= G(U_{\beta,12})G(\Sigma_{23}) S_{\hat{W}_{31}}^{1,2,*} (\eta_1\otimes 1 \otimes 1) S_{\hat{W}_{31}}^{1,2} F(U_{\beta,12}(\rho_B(b)\otimes y \otimes x)U_{\beta,12}^*)F(\Sigma_{23})F(U_{\beta,12}^*)\\
&= G(U_{\beta,12}) G(\Sigma_{23}) S_{\hat{W}_{21}, \mathbb{I}}^{2,1,*}(\eta_2\otimes 1)S_{\hat{W}_{21}, \mathbb{I}}^{2,1} F(U_{\beta,12}(\rho_B(b)\otimes y \otimes x)U_{\beta,12}^*)F(\Sigma_{23})F(U_{\beta,12}^*)\\
&= G(U_{\beta,12})G(\Sigma_{23})S_{\hat{W}_{21}, \mathbb{I}}^{2,1,*} (\eta_2\otimes 1) (F(U_\beta(\rho_B(b)\otimes y)U_\beta^*)\otimes x)S_{\hat{W}_{21}, \mathbb{I}}^{2,1} F(\Sigma_{23})F(U_{\beta,12}^*)\\
&= G(U_{\beta,12})G(\Sigma_{23})S_{\hat{W}_{21}, \mathbb{I}}^{2,1,*} (G(U_\beta(\rho_B(b)\otimes y)U_\beta^*)\otimes x)(\eta_2\otimes 1)S_{\hat{W}_{21}, \mathbb{I}}^{2,1} F(\Sigma_{23}) F(U_{\beta,12}^*)\\
&= G(U_{\beta,12})G(\Sigma_{23}) G(U_{\beta,12}(\rho_B(b)\otimes y\otimes x)U_{\beta,12}^*)S_{\hat{W}_{21}, \mathbb{I}}^{2,1,*} (\eta_2\otimes 1)S_{\hat{W}_{21}, \mathbb{I}}^{2,1}F(\Sigma_{23})F(U_{\beta,12}^*)\\
&= G(U_{\beta,12})G(U_{\beta,13}(\rho_B(b)\otimes x\otimes y)U_{\beta,13}^*)G(\Sigma_{23}) S_{\hat{W}_{21}, \mathbb{I}}^{2,1,*
}(\eta_2\otimes 1)S_{\hat{W}_{21}, \mathbb{I}}^{2,1}F(\Sigma_{23})F(U_{\beta,12}^*)\\
&= G(U_{\beta,12})G(U_{\beta,13}(\rho_B(b)\otimes x\otimes y)U_{\beta,13}^*)G(\Sigma_{23}) S_{\hat{W}_{31}}^{1,2,*}(\eta_1\otimes 1 \otimes 1)S_{\hat{W}_{31}}^{1,2}F(\Sigma_{23})F(U_{\beta,12}^*)\\
&= G(U_{\beta,12})G(U_{\beta,13}(\rho_B(b)\otimes x\otimes y)U_{\beta,13}^*) S_{\hat{W}_{32}}^{1,2,*}(\eta_1\otimes 1 \otimes 1)S_{\hat{W}_{32}}^{1,2}F(U_{\beta,12}^*)\\
&= G(U_{\beta,12})G(U_{\beta,13}(\rho_B(b)\otimes x\otimes y)U_{\beta,13}^*) S_{\mathbb{I}, \hat{W}_{21}}^{2,1,*}(S_\mathbb{I}^{1,1,*}\otimes 1)(\eta_1\otimes 1 \otimes 1) (S_\mathbb{I}^{1,1}\otimes 1) S_{\mathbb{I}, \hat{W}_{21}}^{2,1}F(U_{\beta,12}^*)\\
&= G(U_{\beta,12})G(U_{\beta,13}(\rho_B(b)\otimes x\otimes y)U_{\beta,13}^*) S_{\mathbb{I}, \hat{W}_{21}}^{2,1,*}(S_\mathbb{I}^{1,1,*}\otimes 1)(\eta_1\otimes 1 \otimes 1) (S_\mathbb{I}^{1,1}\otimes 1) (F(U_{\beta}^*)\otimes 1)S_{\hat{W}_{21}, \hat{W}_{21}}^{2,1}\\
&= G(U_{\beta,12})G(U_{\beta,13}(\rho_B(b)\otimes x\otimes y)U_{\beta,13}^*) S_{\mathbb{I}, \hat{W}_{21}}^{2,1,*}(S_\mathbb{I}^{1,1,*}\otimes 1)(\eta_1\otimes 1 \otimes 1) (U_F^*\otimes 1)(S_{\hat{W}_{21}}^{1,1}\otimes 1)S_{\hat{W}_{21}, \hat{W}_{21}}^{2,1}\\
&= G(U_{\beta,12})G(U_{\beta,13}(\rho_B(b)\otimes x\otimes y)U_{\beta,13}^*) S_{\mathbb{I}, \hat{W}_{21}}^{2,1,*}(S_\mathbb{I}^{1,1,*}\otimes 1) (U_G^*\otimes 1)(\eta_1\otimes 1 \otimes 1)(S_{\hat{W}_{21}}^{1,1}\otimes 1)S_{\hat{W}_{21}, \hat{W}_{21}}^{2,1}\\
&= G(U_{\beta,12})G(U_{\beta,13}(\rho_B(b)\otimes x\otimes y)U_{\beta,13}^*) S_{\mathbb{I}, \hat{W}_{21}}^{2,1,*}(G(U_{\beta}^*)\otimes 1)(S_{\hat{W}_{21}}^{1,1,*}\otimes 1)(\eta_1\otimes 1 \otimes 1)(S_{\hat{W}_{21}}^{1,1}\otimes 1)S_{\hat{W}_{21}, \hat{W}_{21}}^{2,1}\\
&=  G(U_{\beta,12})G(U_{\beta,13}(\rho_B(b)\otimes x\otimes y)U_{\beta,13}^*) G(U_{\beta,12}^*)S_{\hat{W}_{21}, \hat{W}_{21}}^{2,1,*}(S_{\hat{W}_{21}}^{1,1,*}\otimes 1)(\eta_1\otimes 1 \otimes 1)(S_{\hat{W}_{21}}^{1,1}\otimes 1)S_{\hat{W}_{21}, \hat{W}_{21}}^{2,1}\\
&= G(U_{\beta,12}U_{\beta,13}(\rho_B(b)\otimes x \otimes y)U_{\beta,13}^*U_{\beta,12}^*) \eta_3.
\end{align*}

 It follows from the bijectivity of \eqref{Rieffelbijection} that there exists a unique natural transformation $\eta: F\implies G$ such that $\eta_{L^2(B)\lhd ((L^2(\G),\hat{W}_{21})\obot (L^2(\G), \hat{W}_{21}))}= \eta_3$. Making use of an isometric intertwiner $T: (L^2(\G), \hat{W}_{21})\to (L^2(\G), \hat{W}_{21})\obot (L^2(\G), \hat{W}_{21})$, it follows that $\eta_{L^2(B)\lhd (L^2(\G), \hat{W}_{21})} = \eta_2$. Since $L^2(B)\lhd (L^2(\G), \hat{W}_{21})$ generates $\Rep(B)$ and since the $\hat{\G}$-representation $(L^2(\G), \hat{W}_{21})$ generates $\Rep(L^\infty(\G))$, we find 
$\eta \in \Nat_{\Rep(L^\infty(\G))}(F,G)$. We then automatically have $\eta_{L^2(B)}= \eta_1$, by the definition of $\eta_2$. Arguing as in the proof of Theorem \ref{restrictionsisos}, it follows that $\eta \in \Nat_{\Rep(L^\infty(\G))}(F,G)= \Nat_{\Rep(\hat{\G})}(F,G)$. Thus, \eqref{injsur} is surjective, finishing the proof.
\end{proof}

It is worth mentioning the following special case of the equivariant Eilenberg-Watts theorem (take $A= B=\C$, which we view as $\G$-$W^*$-algebras), known to experts in the field:

\begin{Cor} Consider the action $\operatorname{Hilb}\curvearrowleft \Rep(\hat{\G})$ through the forgetful functor. Then we have a canonical equivalence
    $\Rep(\G)\simeq \End_{\Rep(\hat{\G})}(\Hilb)$
   of $W^*$-categories.
\end{Cor}

\begin{Rem}
    Let $L^\infty(\mathbb{H} \backslash \G)$ be a von Neumann subalgebra of $L^\infty(\G)$ satisfying the condition $\Delta(L^\infty(\mathbb{H}\backslash \G))\subseteq L^\infty(\mathbb{H}\backslash \G)\ovot L^\infty(\G)$, i.e.\ $L^\infty(\mathbb{H}\backslash \G)\subseteq L^\infty(\G)$ is a \emph{coideal von Neumann subalgebra}. Considering the dual coideal von Neumann subalgebra $L^\infty(\hat{\mathbb{H}}):= L^\infty(\hat{\G})\cap L^\infty(\mathbb{H} \backslash \G)'\subseteq L^\infty(\hat{\G})$ (cfr.\ \cite{KS14}), it follows from Theorem \ref{main2} that
    $$\Corr^{\hat{\G}}(L^\infty(\hat{\mathbb{H}}), L^\infty(\hat{\mathbb{H}}))\simeq \End_{\Rep(\G)}(\Rep(L^\infty(\hat{\mathbb{H}}))).$$
    This generalizes \cite{DCDT24}*{Theorem 2.10} from the compact to the locally compact setting (note however the switch from left to right module actions).
\end{Rem}

The following result is also clear now (compare with Proposition \ref{Moritaeq}):
\begin{Prop}
    Let $(A,\alpha)$ and $(B, \beta)$ be two $\G$-$W^*$-algebras. The following are equivalent:
    \begin{enumerate}\setlength\itemsep{-0.5em}
        \item $(A, \alpha)\sim_\G (B, \beta)$.
        \item $\Rep(A)$ and $\Rep(B)$ are equivalent as $\Rep(\hat{\G})$-$W^*$-module categories.
    \end{enumerate}
\end{Prop}

\section{Drinfeld center of  \texorpdfstring{$\Rep(\G)$}{TEXT}}\label{5}

Let $\G$ be a locally compact quantum group. Instead of working with \emph{right} $\G$-representations as in the previous sections, it will turn out to be very convenient to switch conventions and to work with \emph{left} $\G$-representations instead. To be very concrete, in this section $\Rep(\G)$ denotes the $W^*$-tensor category such that:
\begin{itemize}\setlength\itemsep{-0.5em}
    \item Objects are pairs $(\mcH, U_\mcH)$ where $\mcH$ is a Hilbert space and $U_\mcH\in L^\infty(\G)\ovot B(\mcH)$ is a unitary such that $(\Delta \otimes \id)(U_\mcH)= U_{\mcH, 13}U_{\mcH, 23}$. 
    \item If $\mcH, \mcH'\in \Rep(\G)$, we put $\mathscr{L}^{\G}(\mcH, \mcH')= \{x\in B(\mcH, \mcH'): (1\otimes x)U = U'(1\otimes x)\}$ for the associated morphism space of intertwiners.
    \item The tensor product is given by
    $(\mcH, U_\mcH) \otop (\mcK, U_\mcK):= (\mcH\otimes \mcK, U_{\mcH, 12}U_{\mcK, 13}).$
\end{itemize}

When we write $L^2(\G)\in \Rep(\G)$ in this section, we always mean $(L^2(\G),W)$, unless explicitly mentioned otherwise. In these conventions, the assignments  \begin{align*}
    &(\mcH, U)\mapsto [\phi_U= \phi_\mcH: C_0^u(\hat{\G})\to B(\mcH): L_\sharp^1(\G)\ni \omega\mapsto  (\omega \otimes \id)(U)], \quad \phi\mapsto (\id \otimes \phi)(\wW)
\end{align*}
set up a bijection between the unitary $\G$-representations on Hilbert spaces and the non-degenerate $*$-representations of $C_0^u(\hat{\G})$ on Hilbert spaces. It is important to note that we have
$$\phi_{\mcH \otop \mcK}:= (\phi_\mcH \otimes \phi_\mcK)\circ \hat{\Delta}^{u,\opp}, \quad \mcH, \mcK\in \Rep(\G).$$

We denote the Drinfeld double of $\G$ by $D(\G)$ and we write $\mathcal{Z}(\Rep(\G))$ for the Drinfeld center of the $W^*$-tensor category $\Rep(\G)$. The main goal of this section is to prove that there is a canonical isomorphism $\Rep(D(\G))\cong \mathcal{Z}(\Rep(\G))$ of braided $W^*$-tensor categories. This result was already established for compact quantum groups in \cite{NY18}*{Section 3}.  First, we give some preliminaries in which we recall the relevant definitions.

\subsection{Drinfeld double}

Let $\G$ be a locally compact quantum group. Its \emph{Drinfeld double} $D(\G)$ \cites{BV05,NV10} (see also \cite{DCK24}*{Section 2.3}) is the locally compact quantum group with $L^\infty(D(\G)):= L^\infty(\G)\ovot L^\infty(\hat{\G})$ and coproduct given by
$$\Delta_{D(\G)}(z):= W_{32}\Sigma_{23} (\Delta \otimes \hat{\Delta})(z) \Sigma_{23} W_{32}^*, \quad z \in L^\infty(D(\G)).$$

The locally compact quantum groups $\G$ and $\hat{\G}$ are closed quantum subgroups of $D(\G)$ (in the sense of Vaes \cite{DKSS12}) in a canonical way (see e.g.\ \cite{DKV24}*{Lemma 7.13} for details), meaning there are canonical unital, isometric, normal coproduct-preserving $*$-homomorphisms
$$\gamma_{\G\subseteq D(\G)}: L^\infty(\hat{\G})\to L^\infty(\widehat{D(\G)}), \quad \gamma_{\hat{\G}\subseteq D(\G)}: L^\infty(\G)\to L^\infty(\widehat{D(\G)}).$$
They automatically restrict to non-degenerate $*$-homomorphisms (see \cite{DKSS12}*{Theorem 3.3})
$$\gamma_{\G\subseteq D(\G)}: C_0^r(\hat{\G})\to M(C_0^r(\widehat{D(\G)})), \quad \gamma_{\hat{\G}\subseteq D(\G)}: C_0^r(\G)\to M(C_0^r(\widehat{D(\G)})),$$
and these admit the universal versions
$$\gamma_{\G\subseteq D(\G)}^u: C_0^u(\hat{\G})\to M(C_0^u(\widehat{D(\G)})); \quad \gamma_{\widehat{\G}\subseteq D(\G)}^u: C_0^u(\G)\to M(C_0^u(\widehat{D(\G)})),$$
satisfying
$\gamma_{\G\subseteq D(\G)} \circ \pi_{\hat{\G}}= \pi_{\widehat{D(\G)}}\circ \gamma_{\G\subseteq D(\G)}^u$ and $\gamma_{\hat{\G}\subseteq D(\G)} \circ \pi_{\G}= \pi_{\widehat{D(\G)}}\circ \gamma_{\hat{\G}\subseteq D(\G)}^u$.

By a \emph{Yetter-Drinfeld representation} of $\G$, we mean a triple $(\mcH, U_\mcH, \hat{U}_{\mcH})$ such that
$\mcH$ is a Hilbert space,
$U_{\mcH}\in L^\infty(\G)\ovot B(\mcH)$ is a unitary $\G$-representation,  $\hat{U}_{\mcH}\in L^\infty(\hat{\G})\ovot B(\mcH)$ is a unitary $\hat{\G}$-representation 
and the Yetter-Drinfeld commutation relation 
\begin{equation}\label{commutation}W_{12}U_{\mcH, 13} \hat{U}_{\mcH,23}W_{12}^*= \hat{U}_{\mcH, 23}U_{\mcH, 13}
\end{equation}
holds. Note that \eqref{commutation} is an identity in the multiplier $C^*$-algebra $M(C_0^r(\G)\otimes C_0^r(\hat{\G})\otimes \mathcal{K}(\mcH))$.  The assignment 
$(\mcH, U_\mcH, \hat{U}_\mcH)\mapsto U_{\mcH, 13}\hat{U}_{\mcH,23}$ defines a bijection between the Yetter-Drinfeld representations of $\G$ and the $D(\G)$-representations. Its inverse is constructed as follows: if $\mcH\in \Rep(D(\G))$, we obtain the non-degenerate $*$-representation $\phi_{\mcH}\circ \gamma_{\G \subseteq D(\G)}^u: C_0^u(\hat{\G})\to B(\mcH)$ that corresponds to a unitary $\G$-representation on $\mcH$, and similarly $\phi_\mcH \circ \gamma_{\hat{\G}\subseteq D(\G)}^u: C_0^u(\G)\to B(\mcH)$ corresponds to a unitary $\hat{\G}$-representation on $\mcH$.

The above assignment is compatible with tensor products, in the sense that
$$(\mcH, U_\mcH, \hat{U}_\mcH) \otop_{D(\G)} (\mcK, U_\mcK, \hat{U}_{\mcK})= (\mcH\otimes \mcK, U_{\mcH}\otop_\G U_{\mcK}, \hat{U}_{\mcH}\otop_{\hat{\G}} \hat{U}_{\mcK}).$$

\subsection{Drinfeld center} Let $\G$ be a locally compact quantum group. Consider the associated $W^*$-tensor category $\Rep(\G)$ of unitary $\G$-representations on Hilbert spaces. Its \emph{Drinfeld center} $\mathcal{Z}(\Rep(\G))$ is the $W^*$-category of \emph{unitary half-braidings}. More precisely, objects in this category consist of pairs $(\mcH, C_{\mcH, -})$ where $\mcH \in \Rep(\G)$ and $C_{\mcH, -}: \mcH \otop - \implies - \otop \mcH$ is a natural unitary isomorphism in $\Rep(\G)$, such that the braid relation
\begin{equation}\label{braidrelation}
    C_{\mcH, \mcK \otop \mcL}=  (\id_{\mcK}\otimes C_{\mcH, \mcL})\circ (C_{\mcH, \mcK}\otimes \id_{\mcL}), \quad \mcK, \mcL \in \Rep(\G)
\end{equation}
is satisfied.

A morphism from the object $(\mcH, C_{\mcH,-})$ to the object $(\widetilde{\mcH}, D_{\widetilde{\mcH},-})$ consists of an intertwiner $x: \mcH \to \widetilde{\mcH}$ of $\G$-representations such that the diagram 
$$
\begin{tikzcd}
\mcH\otimes \mcK \arrow[d, "x\otimes 1"'] \arrow[rr, "{C_{\mcH, \mcK}}"] &  & \mcK\otimes \mcH \arrow[d, "1\otimes x"] \\
\widetilde{\mcH}\otimes \mcK \arrow[rr, "{D_{\widetilde{\mcH}, \mcK}}"]  &  & \mcK\otimes \widetilde{\mcH}            
\end{tikzcd}$$
commutes for all $\mcK \in \Rep(\G)$.

Given $(\mcH, C_{\mcH, -}),(\widetilde{\mcH}, D_{\widetilde{\mcH}, -})\in \mathcal{Z}(\Rep(\G))$, we also have
$(\mcH \otop \widetilde{\mcH}, (C_{\mcH, -}\otimes \id_{\widetilde{\mcH}})\circ (\id_{\mcH}\otimes D_{\widetilde{\mcH}, -}))\in \mathcal{Z}(\Rep(\G)).$
In this way, $\mathcal{Z}(\Rep(\G))$ becomes a $W^*$-tensor category. Moreover, the unitaries
$$C_{\mcH, \tilde{\mcH}}: (\mcH, C_{\mcH, -})\otimes (\widetilde{\mcH}, D_{\widetilde{\mcH},-})\to (\widetilde{\mcH}, D_{\widetilde{\mcH},-})\otimes (\mcH, C_{\mcH, -})$$
turn $\mathcal{Z}(\Rep(\G))$ into a braided $W^*$-tensor category.

\subsection{From representation to half-braiding.} Let $(\mcH, U_{\mcH}, \hat{U}_{\mcH})\in \Rep(D(\G))$. Given $\mcK \in \Rep(\G)$, we consider the associated non-degenerate $*$-homomorphism $\phi_{\mcK}: C_0^u(\hat{\G})\to B(\mcK)$, as well as the half-lifted version $\hat{\Uu}_\mcH \in M(C_0^u(\hat{\G})\otimes \mathcal{K}(\mcH))$ of $\hat{U}_\mcH \in M(C_0^r(\hat{\G})\otimes \mathcal{K}(\mcH))$. Similarly as in the previous section, we will employ the notation $$\hat{\Uu}^{\mcK}_\mcH:= (\phi_\mcK \otimes \id)(\hat{\Uu}_\mcH)\in B(\mcK \otimes \mcH),$$
and we define the unitaries
$$C_{\mcH, \mcK}:= \hat{\Uu}^\mcK_\mcH \circ \Sigma_{\mcH, \mcK}: \mcH \otimes \mcK \to \mcK \otimes \mcH, \quad \mcK \in \Rep(\G).$$

Let us now verify that $(\mcH, C_{\mcH,-})\in \mathcal{Z}(\Rep(\G))$. It is easily verified that these unitaries are natural in $\mcK\in \Rep(\G)$. Next, we note that $C_{\mcH, \mcK}$ is an intertwiner of $\G$-representations if and only if 
    \begin{equation}\label{to prove}
        \hat{\Uu}^{\mcK}_{\mcH, 23} U_{\mcH, 13} = U_{\mcK, 12} U_{\mcH, 13} \hat{\Uu}^{\mcK}_{\mcH,23}U_{\mcK, 12}^*.
    \end{equation}
We now prove \eqref{to prove}. 
We consider the half-lifted comultiplication
$\hat{\Delta}^{r,u}: C_0^r(\hat{\G}) \to M(C_0^r(\hat{\G})\otimes C_0^u(\hat{\G}))$ uniquely determined by
$\hat{\Delta}^{r,u}\circ \pi_{\hat{\G}}= (\pi_{\hat{\G}}\otimes \id)\circ \hat{\Delta}^u.$
Applying $\id \otimes \hat{\Delta}^{r,u}\otimes \id$ to the identity \eqref{commutation} and using the identities $(\id \otimes \hat{\Delta}^{r,u})(W)= \wW_{13}W_{12}$ and $(\hat{\Delta}^{r,u}\otimes \id)(\hat{U}_{\mcH})= \hat{U}_{\mcH,13}\hat{\Uu}_{\mcH,23}$, we find
\begin{align*}
    \wW_{13}W_{12} U_{\mcH, 14} \hat{U}_{\mcH,24} \hat{\Uu}_{\mcH,34} W_{12}^* \wW_{13}^*=\hat{U}_{\mcH, 24} \hat{\Uu}_{\mcH, 34} U_{\mcH,14}.
\end{align*}
Rearranging and invoking \eqref{commutation}, we find
$$\hat{U}_{\mcH, 24}U_{\mcH,14} W_{12 }  \hat{\Uu}_{\mcH,34} W_{12}^* \wW_{13}^*=W_{12}U_{\mcH, 14} \hat{U}_{\mcH,24} \hat{\Uu}_{\mcH,34} W_{12}^* \wW_{13}^* = \hat{U}_{\mcH, 24} \wW_{13}^* \hat{\Uu}_{\mcH,34} U_{\mcH,14}.$$
Making the obvious cancellations and leaving out the second leg leads to the identity
$$U_{\mcH,13}\hat{\Uu}_{\mcH,23} \wW_{12}^* = \wW_{12}^* \hat{\Uu}_{\mcH,23}U_{\mcH,13}.$$
    Rearranging, applying the map $\id \otimes \phi_{\mcK}\otimes \id$ to this equation, and keeping in mind that $(\id \otimes \phi_{\mcK})(\wW)= U_{\mcK}$, the equation \eqref{to prove} readily follows. 
    
    Finally, given $\mcK, \mcL\in \Rep(\G)$, we compute
    \begin{align*} C_{\mcH,\mcK \otop \mcL} &= (\phi_{\mcK \otop \mcL}\otimes \id)(\hat{\Uu}_{\mcH})\Sigma_{\mcH, \mcK \otimes \mcL}\\
    &= (\phi_\mcK \otimes \phi_{\mcL}\otimes \id)(\hat{\Delta}^{u, \opp} \otimes \id)(\hat{\Uu}_{\mcH})\Sigma_{\mcH, \mcK \otimes \mcL}\\
    &= (\phi_\mcK \otimes \phi_\mcL\otimes \id)(\hat{\Uu}_{\mcH,23}\hat{\Uu}_{\mcH,13})\Sigma_{\mcH, \mcK \otimes \mcL}\\
    &= (\phi_\mcL \otimes \id)(\hat{\Uu}_{\mcH})_{23} (\phi_\mcK\otimes \id)(\hat{\Uu}_{\mcH})_{13} \Sigma_{\mcH, \mcK \otimes \mcL}\\
    &= (\phi_\mcL \otimes \id)(\hat{\Uu}_{\mcH})_{23}(1\otimes \Sigma_{\mcH, \mcL})(1\otimes \Sigma_{\mcH, \mcL}^*)(\phi_\mcK \otimes \id)(\hat{\Uu}_{\mcH})_{13} \Sigma_{\mcH, \mcK \otimes \mcL}= (1 \otimes C_{\mcH, \mcL})(C_{\mcH, \mcK}\otimes 1).
    \end{align*}
We conclude that $(\mcH, C_{\mcH, -})\in \mathcal{Z}(\Rep(\G))$. If $(\mcH, U_{\mcH}, \hat{U}_{\mcH}), (\widetilde{\mcH}, U_{\widetilde{\mcH}}, \hat{U}_{\widetilde{\mcH}})\in \Rep(D(\G))$, and $x: \mcH \to \widetilde{\mcH}$ is an intertwiner of $D(\G)$-representations, then $x$ is both an intertwiner of the corresponding $\G$-representations and $\hat{\G}$-representations. From this, we deduce that $x$ is a morphism in $\mathcal{Z}(\Rep(\G))$ between the corresponding unitary half-braidings. Thus, we obtain a  canonical functor
$$F: \Rep(D(\G))\to \mathcal{Z}(\Rep(\G))$$
that is the identity on morphism spaces.

\subsection{From half-braiding to representation.}

We will require the following lemma:

\begin{Lem}\label{trivialbraiding}
    Let $(\mcH, C_{\mcH, -})\in \mathcal{Z}(\Rep(\G))$. If $\mcK$ is a Hilbert space, endowed with the trivial $\G$-representation, then $C_{\mcH, \mcK}= \Sigma_{\mcH, \mcK}$.
\end{Lem}
\begin{proof} If $x\in B(\mcK)$, then by naturality of $C_{\mcH, \mcK}$ in the second variable, we find
$$C_{\mcH, \mcK}\Sigma_{\mcK, \mcH}(x\otimes 1) = C_{\mcH, \mcK}(1\otimes x)\Sigma_{\mcK, \mcH} = (x\otimes 1)C_{\mcH, \mcK}\Sigma_{\mcK, \mcH}.$$
From this, it follows that $C_{\mcH, \mcK}= (1\otimes y)\Sigma_{\mcH, \mcK}$ for some unitary $y\in B(\mcH)$. We now prove that $y = 1$. We first do this in case that $\mcK= \C$ is one-dimensional. In that case, writing $C_{\mcH, \C}= (1\otimes y)\Sigma_{\mcH, \C}$, we consider the following diagram, where $\mcL\in \Rep(\G)$ is fixed:
$$
\begin{tikzcd}
\mcH \otimes \mcL \arrow[rrrr, "y\otimes 1"] \arrow[rd, "\cong", no head] \arrow[rrddd, "{C_{\mcH, \mcL}}"', bend right] &                                                                                                               &                                                          &                                                                     & \mcH \otimes \mcL \arrow[ld, "\cong"', no head] \arrow[llddd, "{C_{\mcH, \mcL}}", bend left] \\
   & \mcH \otimes \C \otimes \mcL \arrow[rr, "{C_{\mcH, \C}\otimes 1}"] \arrow[rd, "{C_{\mcH, \C \otop \mcL}}"'] &  & \C\otimes \mcH \otimes \mcL \arrow[ld, "{1\otimes C_{\mcH, \mcL}}"] &              \\
&              & \C \otimes \mcL \otimes \mcH \arrow[d, "\cong", no head] &            &             \\
 &   & \mcL\otimes \mcH        &       &    \end{tikzcd}$$
Here, the middle triangle commutes by the braid relation \eqref{braidrelation} and the left outer subdiagram commutes by naturality in the second variable. The two remaining outer subdiagrams commute trivially. The take-away from this is that the outer part of the diagram commutes, forcing $y=1$. 

Back to the general case, we fix an isometry $t: \C \to \mcK$, which is of course an intertwiner of $\G$-representations. By naturality in the second variable, we have a commutative diagram
$$
\begin{tikzcd}
\mcH\otimes \mcK \arrow[rr, "C_{\mcH, \mcK}=(1\otimes y)\Sigma"]           &  & \mcK\otimes \mcH                        \\
\mcH\otimes \C \arrow[u, "1\otimes t"] \arrow[rr, "C_{\mcH, \C}=\Sigma"] &  & \C\otimes \mcH \arrow[u, "t\otimes 1"']
\end{tikzcd}$$
which forces $t\otimes y = t\otimes 1$, whence $y=1$.\end{proof}

If $(\mcH, C_{\mcH, -})\in \mathcal{Z}(\Rep(\G))$, let us verify that $(\mcH, U_{\mcH}, \hat{U}_{\mcH}:= C_{\mcH, L^2(\G)}\circ \Sigma_{L^2(\G),\mcH}) \in \Rep(D(\G))$. Indeed, we first note that if $\hat{x}'\in L^\infty(\hat{\G})'=\mathcal{L}^{\G}(L^2(\G))$, then by naturality in the second variable
\begin{align*}
    \hat{U}_{\mcH}(\hat{x}'\otimes 1)&=C_{\mcH, L^2(\G)}(1\otimes \hat{x}')\Sigma_{L^2(\G), \mcH}= (\hat{x}'\otimes 1)C_{\mcH, L^2(\G)}\Sigma_{L^2(\G),\mcH}= (\hat{x}'\otimes 1)\hat{U}_{\mcH}, 
\end{align*}
and it follows that $\hat{U}_{\mcH}\in L^\infty(\hat{\G})\ovot B(\mcH).$ Next, we verify that $\hat{U}_{\mcH}$ is a $\hat{\G}$-representation. To do this, we calculate
\begin{align*}
    \hat{U}_{\mcH, 23} \hat{U}_{\mcH, 13} W_{12}&= (1\otimes C_{\mcH, L^2(\G)})(1\otimes \Sigma)C_{\mcH, L^2(\G),13} \Sigma_{13} W_{12}\\
    &= (1\otimes C_{\mcH, L^2(\G)})(C_{\mcH, L^2(\G)}\otimes 1)(\Sigma \otimes 1)(1\otimes \Sigma)W_{12}\\
    &= C_{\mcH, L^2(\G)\otop L^2(\G)} (\Sigma\otimes 1)(1\otimes \Sigma)W_{12}\\
    &= C_{\mcH, L^2(\G)\otop L^2(\G)}W_{23} (\Sigma\otimes 1)(1\otimes \Sigma)\\
    &= W_{12} C_{\mcH, (L^2(\G),W) \otop (L^2(\G), \mathbb{I})} (\Sigma\otimes 1)(1\otimes \Sigma)\\
    &= W_{12}(1\otimes \Sigma)(C_{\mcH, L^2(\G)}\otimes 1)(\Sigma\otimes 1)(1\otimes \Sigma)= W_{12} \hat{U}_{\mcH, 13}.
\end{align*}
Here, the third equality follows from the braid equation \eqref{braidrelation}. The fifth equation follows from the fact that $W \in \mathcal{L}^\G((L^2(\G),W) \otop (L^2(\G),\mathbb{I}), (L^2(\G),W)\otop (L^2(\G), W))$ and the naturality in the second variable. The sixth equality follows from the braid relation \eqref{braidrelation} combined with Lemma \ref{trivialbraiding}.

Finally, from the fact that $C_{\mcH, L^2(\G)}: \mcH\otimes L^2(\G)\to L^2(\G)\otimes \mcH$ is an intertwiner of $\G$-representations, it follows that
\begin{align*}
    W_{12}U_{\mcH, 13}\hat{U}_{\mcH, 23}W_{12}^*&= W_{12} U_{\mcH, 13} (1\otimes C_{\mcH, L^2(\G)})(1\otimes \Sigma)W_{12}^*\\
    &= (1\otimes C_{\mcH, L^2(\G)})U_{\mcH, 12}W_{13}(1\otimes \Sigma)W_{12}^*\\
    &= (1\otimes C_{\mcH, L^2(\G)})U_{\mcH, 12} (1\otimes \Sigma)\\
    &= (1\otimes C_{\mcH, L^2(\G)}) (1\otimes \Sigma) U_{\mcH, 13}= \hat{U}_{\mcH, 23} U_{\mcH, 13},
\end{align*}
so that $(\mcH, U_{\mcH}, \hat{U}_{\mcH})\in \Rep(D(\G))$.

If $(\mcH, C_{\mcH,-}), (\widetilde{\mcH}, D_{\widetilde{\mcH}, -})\in \mathcal{Z}(\Rep(\G))$ and if $x: \mcH \to \widetilde{\mcH}$ is a morphism between these objects in $\mathcal{Z}(\Rep(\G))$ (which, by definition, is an intertwiner of $\G$-representations), then we also have
\begin{align*}
    (1\otimes x) C_{\mcH, L^2(\G)}\Sigma = D_{\widetilde{\mcH}, L^2(\G)}(x\otimes 1)\Sigma= D_{\widetilde{\mcH}, L^2(\G)} \Sigma (1\otimes x),
\end{align*}
so that $x$ is also an intertwiner of the associated $\hat{\G}$-representations. As such, $x$ is an intertwiner of the associated $D(\G)$-representations. It is then clear that we obtain a functor
$$G: \mathcal{Z}(\Rep(\G))\to \Rep(D(\G))$$
which acts identically on morphisms.

\subsection{Isomorphism of W*-tensor categories}

\begin{Lem}\label{equality}
    If $(\mcH, C_{\mcH, -}), (\mcH, D_{\mcH, -})\in \mathcal{Z}(\Rep(\G))$ and $C_{\mcH, L^2(\G)}= D_{\mcH, L^2(\G)}$, then $C_{\mcH, \mcK}= D_{\mcH, \mcK}$ for all $\mcK\in \Rep(\G)$.
\end{Lem}

\begin{proof} Since $L^2(\G)$ is a generator for $\Rep(L^\infty(\hat{\G}))\subseteq \Rep(\G)$, it follows that 
$C_{\mcH, \mcK}= D_{\mcH, \mcK}$ for all $\mcK \in \Rep(L^\infty(\hat{\G})).$
If $\mcK \in \Rep(\G)$, then $\mcK\otop L^2(\G)\in \Rep(L^\infty(\hat{\G}))$, and thus the braid relation \eqref{braidrelation} implies that
$$(1\otimes C_{\mcH, L^2(\G)}) (C_{\mcH, \mcK}\otimes 1)= C_{\mcH,\mcK \otop L^2(\G)}= D_{\mcH, \mcK\otop L^2(\G)}=(1\otimes D_{\mcH, L^2(\G)}) (D_{\mcH, \mcK}\otimes 1)= (1\otimes C_{\mcH, L^2(\G)})(D_{\mcH, \mcK}\otimes 1),$$
from which it follows that $C_{\mcH, \mcK}= D_{\mcH, \mcK}.$
\end{proof}

\begin{Theorem}\label{main3}
The functors $F$ and $G$ are inverse to each other, i.e. $F\circ G = \id_{\mathcal{Z}(\Rep(\G))}$ and $G\circ F = \id_{\Rep(D(\G))}.$ Moreover, $F\circ \otop = \otimes \circ (F\times F)$ and $G\circ \otimes= \otop \circ (G\times G)$.  Thus, $\mathcal{Z}(\Rep(\G))\cong \Rep(D(\G))$ as $W^*$-tensor categories.
\end{Theorem}
\begin{proof}
    Starting from $(\mcH, U_{\mcH}, \hat{U}_{\mcH})\in \Rep(D(\G))$, we associate to it
$(\mcH, \hat{\Uu}_{\mcH}^{-}\circ\Sigma_{\mcH, -})\in \mathcal{Z}(\Rep(\G))$, which in turn gives rise to the $\hat{\G}$-representation
$$\hat{\Uu}_\mcH^{L^2(\G)}\circ \Sigma_{\mcH, L^2(\G)}\circ \Sigma_{L^2(\G), \mcH}= (\phi_{L^2(\G)}\otimes \id)(\hat{\Uu}_{\mcH})= (\pi_{\hat{\G}}\otimes \id)(\hat{\Uu}_{\mcH})= \hat{U}_{\mcH}.$$
Therefore, $G\circ F = \id_{\Rep(D(\G))}.$

Conversely, given $(\mcH, C_{\mcH, -})\in \mathcal{Z}(\Rep(\G))$, the associated $\hat{\G}$-representation
$\hat{U}_{\mcH}:= C_{\mcH, L^2(\G)} \Sigma_{L^2(\G), \mcH}$ induces the unitary half-braiding 
$(\mcH, \hat{\Uu}_{\mcH}^{-}\circ \Sigma_{\mcH, -})$.
We then have
$$\hat{\Uu}^{L^2(\G)}_\mcH \Sigma_{\mcH, L^2(\G)}=(\phi_{L^2(\G)}\otimes \id)(\hat{\Uu}_\mcH)\Sigma_{\mcH, L^2(\G)}= (\pi_{\hat{\G}}\otimes \id)(\hat{\Uu}_{\mcH}) \Sigma_{\mcH, L^2(\G)}= \hat{U}_{\mcH}\Sigma_{\mcH, L^2(\G)}= C_{\mcH, L^2(\G)},$$
so Lemma \ref{equality} implies that $F\circ G = \id_{\mathcal{Z}(\Rep(\G))}$. Verifying the identities $F\circ \otop= \otimes \circ (F\times F)$ and $G\circ \otimes= \otop \circ (G\times G)$ is routine.
\end{proof}

\subsection{Compatibility with the braiding.}

As we have seen, $\mathcal{Z}(\Rep(\G))$ is naturally a braided $W^*$-tensor category. On the other hand, the locally compact quantum group $D(\G)$ is quasi-triangular with respect to the $R$-matrix
$$\hat{R}:= (\gamma_{\G\subseteq D(\G)}\otimes \gamma_{\hat{\G}\subseteq D(\G)})(\hat{W}) \in M(C_0^r(\widehat{D(\G)})\otimes C_0^r(\widehat{D(\G)}))$$
(see \cite{DCK24}*{Remark 2.4, Proposition 2.7}). Considering its universal lift $\hat{R}^u \in M(C_0^u(\widehat{D(\G)})\otimes C_0^u(\widehat{D(\G)}))$, uniquely determined by $(\pi_{\widehat{D(\G)}}\otimes \pi_{\widehat{D(\G)}})(\hat{R}^u) = \hat{R}$ \cite{MW12}*{Proposition 4.14}, it follows from \cite{DCK24}*{Proposition 3.4} that the unitaries
\begin{equation}\label{braidingspecific}
    (\phi_{\mcK}\otimes \phi_\mcH)(\hat{R}^u)\circ \Sigma_{\mcH, \mcK}: \mcH \otop \mcK\to \mcK\otop \mcH, \quad \mcH, \mcK\in \Rep(D(\G))
    \end{equation}
turn $\Rep(D(\G))$ into a braided $W^*$-tensor category. 

\begin{Prop}\label{mainbraiding}
    The monoidal functors $F,G$ preserve the braiding. Thus, $\Rep(D(\G))\cong \mathcal{Z}(\Rep(\G))$ as braided $W^*$-tensor categories.
\end{Prop}
\begin{proof} The universal lift $\hat{R}^u$  is explicitly given by
$$\hat{R}^u=(\gamma_{\G\subseteq D(\G)}^u\otimes \gamma_{\hat{\G}\subseteq D(\G)}^u)(\widehat{\WW})\in M(C_0^u(\widehat{D(\G)})\otimes C_0^u(\widehat{D(\G)})),$$
because we have
\begin{align*}
    (\pi_{\widehat{D(\G)}}\otimes \pi_{\widehat{D(\G)}})(\hat{R}^u)= (\gamma_{\G\subseteq D(\G)}\otimes \gamma_{\hat{\G}\subseteq D(\G)} )(\pi_{\hat{\G}}\otimes \pi_{\G})(\widehat{\WW}) = \hat{R}.
\end{align*}
Consequently, given $(\mcH, U_\mcH, \hat{U}_\mcH), (\mcK, U_\mcK, \hat{U}_\mcK)\in \Rep(D(\G))$, we can calculate
\begin{align*}
    (\phi_\mcK \otimes \phi_\mcH)(\hat{R}^u)&= (\phi_\mcK\otimes \phi_\mcH)(\gamma_{\G\subseteq D(\G)}^u\otimes \gamma_{\hat{\G}\subseteq D(\G)}^u)(\widehat{\WW})\\
    &= (\phi_{U_\mcK}\otimes \phi_{\hat{U}_\mcH}) (\widehat{\WW})\\
    &= (\phi_{U_\mcK}\otimes \id)(\id \otimes \phi_{\hat{U}_\mcH})(\widehat{\WW})\\
    &= (\phi_{U_\mcK}\otimes \id)(\hat{\Uu}_\mcH).
\end{align*}
Thus, the braiding \eqref{braidingspecific} on $\Rep(D(\G))$ is given by the unitaries
    $$(\phi_{U_\mcK}\otimes \id)(\hat{\Uu}_\mcH)\circ \Sigma_{\mcH, \mcK}: \mcH \otop \mcK\to \mcK \otop \mcH, \quad \mcH, \mcK \in \Rep(D(\G)).$$
    As such, it is clear that $F$ and $G$ preserve the braiding.
\end{proof}

\textbf{Acknowledgments:} The research of the author happened at the Department of Mathematics and Data Science of the Vrije Universiteit Brussel, supported by Fonds voor Wetenschappelijk Onderzoek (Flanders), via an FWO Aspirant fellowship, grant 1162524N. The author would like to thank K. De Commer and J. Krajczok for useful discussions during this project.

\end{document}